\documentclass[11pt]{article}

\usepackage[margin=1.05in]{geometry}
\usepackage[utf8]{inputenc}
\usepackage[T1]{fontenc}
\usepackage{lmodern}
\usepackage{amsmath,amssymb,amsfonts,amsthm,mathtools}
\usepackage{mathrsfs}
\usepackage{enumitem}
\usepackage{xcolor}
\usepackage{microtype}
\usepackage{hyperref}
\usepackage[nameinlink,capitalise]{cleveref}

\numberwithin{equation}{section}
\allowdisplaybreaks

\hypersetup{
  colorlinks=true,
  linkcolor=blue!60!black,
  citecolor=blue!60!black,
  urlcolor=blue!60!black
}

\setlist[itemize]{leftmargin=2em}
\setlist[enumerate]{leftmargin=2em}

\theoremstyle{plain}
\newtheorem{theorem}{Theorem}[section]
\newtheorem{proposition}[theorem]{Proposition}
\newtheorem{lemma}[theorem]{Lemma}
\newtheorem{corollary}[theorem]{Corollary}

\theoremstyle{definition}
\newtheorem{definition}[theorem]{Definition}
\newtheorem{assumption}[theorem]{Assumption}
\newtheorem{problem}[theorem]{Problem}
\newtheorem{conjecture}[theorem]{Conjecture}

\newtheorem{principle}[theorem]{Principle}
\newtheorem{target}[theorem]{Target Theorem}

\theoremstyle{remark}
\newtheorem{remark}[theorem]{Remark}

\newcommand{\R}{\mathbb R}
\newcommand{\T}{\mathbb T}
\newcommand{\Z}{\mathbb Z}

\newcommand{\eps}{\varepsilon}
\newcommand{\CKN}{\mathrm{CKN}}
\newcommand{\NS}{\mathrm{NS}}
\newcommand{\PF}{\mathrm{PF}}
\newcommand{\PFE}{\mathrm{PFE}}
\newcommand{\PFET}{\mathrm{PFET}}

\newcommand{\loc}{\mathrm{loc}}
\newcommand{\har}{\mathrm{har}}
\newcommand{\Har}{\mathrm{Har}}
\newcommand{\Com}{\mathrm{Com}}
\newcommand{\Loc}{\mathrm{Loc}}
\newcommand{\Err}{\mathrm{Err}}
\newcommand{\quot}{\mathrm{quot}}

\newcommand{\dxdt}{\,dx\,dt}

\newcommand{\norm}[2]{\left\|#1\right\|_{#2}}
\newcommand{\pair}[2]{\left\langle #1,#2\right\rangle}
\newcommand{\dist}{\operatorname{dist}}

\newcommand{\tr}{\operatorname{tr}}

\newcommand{\esssup}{\operatorname*{ess\,sup}}

\newcommand{\calC}{\mathcal C}
\newcommand{\calD}{\mathcal D}
\newcommand{\calE}{\mathcal E}

\newcommand{\calG}{\mathcal G}
\newcommand{\calH}{\mathcal H}
\newcommand{\calK}{\mathcal K}

\newcommand{\calN}{\mathcal N}
\newcommand{\calO}{\mathcal O}

\newcommand{\calR}{\mathcal R}

\newcommand{\calW}{\mathcal W}
\newcommand{\calZ}{\mathcal Z}
\newcommand{\frakD}{\mathfrak D}
\newcommand{\frakO}{\mathfrak O}
\newcommand{\bfB}{\mathfrak B}
\newcommand{\Xcrit}{\mathcal X_{\mathrm{crit}}}
\newcommand{\Good}{\mathcal G}
\newcommand{\Bad}{\mathcal B}
\newcommand{\Vort}{\mathcal V}
\newcommand{\Flux}{\mathcal F}
\newcommand{\Press}{\mathcal P}

\newcommand{\act}{\mathrm{act}}
\newcommand{\tot}{\mathrm{tot}}
\newcommand{\absop}{\mathrm{abs}}

\newcommand{\sg}{\mathrm{sg}}
\newcommand{\sgn}{\mathrm{sgn}}

\newcommand{\Bop}{\mathsf B}
\newcommand{\Fop}{\mathsf F}

\title{\textbf{Invisible Defect Cascades for Navier--Stokes Regularity}}
\author{Runlong Yu\\
	The University of Alabama, Tuscaloosa, AL, USA\\
	\texttt{ryu5@ua.edu}}
\date{}

\begin{document}
\maketitle

\begin{abstract}
	We formulate a conditional scale-critical defect-cascade reduction for the local regularity problem of the three-dimensional incompressible Navier--Stokes equations.  The theorem concerns a potential singular point for which no sufficiently small dyadic scale enters the Caffarelli--Kohn--Nirenberg smallness regime.  Under the structural hypotheses of the framework, such a point cannot be explained by an undifferentiated concentration of energy or pressure.  It must lead either to non-effective moving-window observability or to an NS-realizable, cleaned, scale-critical defect cascade invisible to the combined active-pressure, flux, energy, and adjoint-trace tests. The reduction is built from dyadic rescaling, coarse graining, active/harmonic pressure splitting, Reynolds covariance positivity, pressure compatibility, and local energy-flux identities.  Finite-window observability reduces possible invisible directions to explicitly defined residual kernels, while budget compatibility and sign coherence convert visible pressure--flux activity into depletion.  Consequently, within a controlled window class where dyadic defect extraction, observable depletion, and moving-window growth control hold, effective observability together with exclusion of NS-realizable combined-invisible cascades yields a CKN scale and hence local regularity.  The final component interprets the remaining obstruction as a critical recurrence problem and proposes diagnostics for vortex stretching, active pressure work, and interscale flux.  Spatially harmonic pressure is retained as a physical local pressure component; only purely time-dependent pressure functions are treated as gauge.
\end{abstract}



\tableofcontents

\section{Introduction}

This paper is organized around the following Clay-adjacent reduction principle:
\[
\boxed{
\text{potential Navier--Stokes singularities are reduced to the existence of an}
}
\]
\[
\boxed{
\text{NS-realizable, cleaned, scale-critical, combined-invisible defect cascade.}
}
\]
The statement is deliberately not a proof of global Navier--Stokes regularity, and it is not a singular-solution construction.  It is a structural reduction theorem.  If a point fails to enter the Caffarelli--Kohn--Nirenberg smallness regime at all sufficiently small dyadic scales, then in the present framework the obstruction must have a very special form.  It must survive dyadic coarse graining, separation of active and harmonic pressure, finite-window localization, pressure testing, flux testing, positive energy testing, and adjoint trace testing.

The clean one-sentence conclusion is:
\begin{quote}
Potential Navier--Stokes singularities are reduced to the existence of an NS-realizable combined-invisible defect cascade.
\end{quote}
Thus the role of the paper is not to announce that singularities have been excluded.  Rather, it compresses the possible shape of a singular branch from a vague compactness failure into a highly constrained scale-critical object.

\subsection*{The theorem spine}

Let
\[
        r_n=2^{-n}
\]
be dyadic scales around a potential singular point.  A non-CKN branch means that
\[
        C(r_n)+D(r_n)>
        \varepsilon_{\CKN}
        \qquad\text{for all sufficiently large } n.
\]
From such a branch one attempts to extract NS-realizable coarse-grained defect packages
\[
        D_{n,\ell_n}
        =
        (U_{n,\ell_n},P_{n,\ell_n};P^{\act}_{n,\ell_n},P^{\har}_{n,\ell_n},R_{n,\ell_n},\Pi_{n,\ell_n}),
\]
with
\[
        R_{n,\ell_n}
        =
        S_{\ell_n}(u^{(n)}\otimes u^{(n)})
        -U_{n,\ell_n}\otimes U_{n,\ell_n},
        \qquad
        \Pi_{n,\ell_n}=-R_{n,\ell_n}:\nabla U_{n,\ell_n}.
\]
These objects are not formal defects.  They are produced by Navier--Stokes rescaling and coarse graining, and hence carry the exact momentum identity, the pressure-compatibility relation, the active/harmonic pressure split, the Reynolds covariance constraint, and the local energy-flux identity.

In each finite window
\[
        W_n=(n,\ell_n,\Lambda_n,\chi_n,s_n),
\]
we test the cleaned defect by the combined observation map
\[
        O^{\mathrm{comb}}_{W_n}
        =
        (O^P_{W_n},O^F_{W_n},O^E_{W_n},O^T_{W_n}),
\]
consisting of active pressure, flux, energy, and adjoint-trace channels.  The finite-window hierarchy is
\[
        \text{pressure}
        \Longrightarrow
        \text{flux}
        \Longrightarrow
        \text{energy}
        \Longrightarrow
        \text{adjoint trace}.
\]
Active-pressure-visible defects are not phantoms.  Active-pressure-null defects are tested by flux, while the local harmonic pressure component is retained in the momentum and energy channels rather than quotiented out.  Pressure--flux invisible defects enter the finite-dimensional kernel \(K_W^{\PF}\).  Positive NS-realizable directions inside this kernel are then tested by the energy channel.  Residual pressure--flux--energy invisible directions enter \(K_W^{\PFE}\) and are finally tested by selected-time trace-cost duality.  The final fixed-window residual is the projected pressure--flux--energy--trace obstruction space
\[
        \mathcal T^{\PFET}_{W,\NS}.
\]

The moving-window theorem has the following form:
\[
\boxed{
\begin{gathered}
\text{dyadic defect extraction + observable depletion}\\
\text{+ finite-window NS-realizability exclusion}\\
\text{+ moving-window growth control}\\
\Longrightarrow
\text{no indefinitely persistent non-CKN dyadic branch.}
\end{gathered}}
\]
Equivalently,
\[
\boxed{
\begin{gathered}
\text{controlled observability + no NS-realizable invisible cascade}\\
\Longrightarrow \text{CKN scale} \Longrightarrow \text{local regularity.}
\end{gathered}}
\]
The contrapositive is the sharper structural statement:
\[
\boxed{
\text{if no CKN scale occurs, then either moving-window observability is non-effective}
}
\]
\[
\boxed{
\text{or there exists an NS-realizable combined-invisible defect cascade.}
}
\]

\subsection*{Main contributions}

The first contribution is the shift from a one-component or single-channel viewpoint to a full three-dimensional defect-cascade problem.  The object of study is not a special smallness assumption on one velocity component, but the scale-critical structure that any potential singular branch must carry.

The second contribution is the NS-realizability filter.  A formal finite-dimensional kernel is not automatically a Navier--Stokes obstruction.  It becomes dangerous only after intersection with the closure of the range of actual NS-derived coarse-grained packages:
\[
        d\in
        \overline{\operatorname{Range}\mathcal R^{\NS}_W}.
\]
This distinction prevents model phantoms from being mistaken for genuine Navier--Stokes obstructions.

The third contribution is the combined observability mechanism.  Pressure, flux, energy, and trace are not treated as independent decorations.  They are tied together by the coarse momentum identity, pressure compatibility, local harmonic pressure work, positive covariance, the flux identity, and the local energy inequality.  A defect invisible to one channel is forced into a smaller residual class tested by the next channel.

The fourth contribution is the explicit failure object.  A genuine moving-window failure must produce a sequence
\[
        (W_n,d_n,y_n)
\]
such that \(d_n\) is NS-realizable, \(\|d_n\|\ge c_0\), \(\|y_n\|=1\),
\[
        \mathcal O_n^*(y_n)\to0,
        \qquad
        |\langle d_n,y_n\rangle|\ge c_1>0.
\]
This is the NS-realizable invisible moving-window defect cascade.  A regularity route must rule it out.  A counterexample route must construct it while preserving the coarse momentum equation, pressure compatibility, \(R\ge0\), the flux identity, and the local energy inequality.

The fifth contribution is the reduction of future work to four theorem targets:
\[
\begin{gathered}
\boxed{\text{dyadic defect extraction}},
\qquad
\boxed{\text{observable depletion}},\\[2mm]
\boxed{\text{finite-window NS-realizability exclusion}},
\qquad
\boxed{\text{moving-window growth control}}.
\end{gathered}
\]
Proving these targets in an intrinsic window class would give a local regularity theorem.  Constructing an NS-realizable failure of one of them would identify a genuine obstruction to this route.

\subsection*{From invisible defects to critical recurrence}

The combined-invisible cascade language is useful because it cleans away formal artifacts and identifies the residual obstruction left by pressure, flux, energy, and trace observations.  However, it should not be viewed as the final explanatory language.  A possible singularity would have to do more than hide from a chosen observation map: it would have to reproduce a non-CKN pattern across infinitely many scales.  This motivates the critical recurrence viewpoint developed later in the manuscript.

In this viewpoint the rescaled sequence
\[
    u^{(n)}(x,t)=r_nu(r_nx,r_n^2t),
    \qquad
    p^{(n)}(x,t)=r_n^2p(r_nx,r_n^2t)
\]
is regarded as an orbit in a critical state space.  A persistent non-CKN branch is then a recurrent bad orbit.  The question becomes: what sustains this recurrence?  We isolate three possible scale-critical sustaining mechanisms: vortex-stretching production, active pressure work, and interscale energy flux.  The programmatic part of the paper formulates diagnostics for these mechanisms and explains how they refine the earlier defect-observability framework.

\subsection{Relation to existing literature and novelty of the formulation}
\label{subsec:related-work}

The framework should be viewed as a synthesis of several established themes in the local theory of the Navier--Stokes equations, rather than as an isolated replacement for them.  The weak-solution background begins with the Leray--Hopf theory \cite{Leray1934,Hopf1951}.  The local suitable-weak-solution and partial-regularity endpoint used here comes from the work of Scheffer, Caffarelli--Kohn--Nirenberg, and Lin \cite{Scheffer1976,Scheffer1977,CKN1982,Lin1998}; related refinements and alternative proofs include Struwe, Ladyzhenskaya--Seregin, Choe--Lewis, Vasseur, and Seregin's monograph \cite{Struwe1988,LadyzhenskayaSeregin1999,ChoeLewis2000,Vasseur2007,Seregin2015}.  In this theory, regularity is obtained once a scale-invariant local quantity enters a universal smallness regime.  The present paper keeps this CKN endpoint but changes the object studied before smallness occurs: instead of tracking only the size of a norm, it tracks the equation-generated defect package which remains along a non-CKN dyadic branch.

A second related direction is local \(\varepsilon\)-regularity, pressure-sensitive regularity criteria, and quantitative regularity near possible singularities.  Interior criteria and one-scale criteria appear in work of Gustafson--Kang--Tsai, Seregin, Guevara--Phuc, and Albritton--Barker--Prange \cite{GustafsonKangTsai2007,Seregin2007Local,GuevaraPhuc2017,AlbrittonBarkerPrange2023}.  Barker and Prange \cite{BarkerPrange2021} show that, under critical or Type-I hypotheses, potential singularities force quantitative concentration behavior and yield bounds on regularity mechanisms.  This line of work is close in spirit to the present reduction because both approaches ask what must persist if a regular scale has not yet appeared.  The difference is that the present paper does not formulate the obstruction only as concentration of a critical norm.  It further decomposes the obstruction into pressure, flux, energy, and trace channels, and then asks whether the remaining invisible direction is realized by actual Navier--Stokes dyadic defect packages.

A third related direction is the large literature on critical regularity criteria and critical-space obstruction mechanisms.  The classical Prodi--Serrin criteria and their refinements form part of this background \cite{Prodi1959,Serrin1962,Serrin1963,KozonoSohr1997}.  The endpoint and critical-element viewpoints are represented by Escauriaza--Seregin--\v{S}ver\'ak, Kenig--Koch, Gallagher--Koch--Planchon, and Jia--\v{S}ver\'ak \cite{EscauriazaSereginSverak2003,KenigKoch2011,GallagherKochPlanchon2013,JiaSverak2014}.  The moving-window defect cascade in the present paper plays an analogous structural role, but it is not an ancient solution, a self-similar profile, or a compact critical element.  It is a sequence of cleaned finite-window defect directions extracted from coarse-grained Navier--Stokes packages, together with dual directions invisible to the combined observation map.

A fourth point of contact is one-component and anisotropic regularity theory.  Results of Kukavica--Ziane, Zhou--Pokorn\'y, Cao--Titi, Chemin--Zhang, Chemin--Zhang--Zhang, Kukavica--Rusin--Ziane, Han--Lei--Li--Zhao, and Kang--Nguyen show how partial information about one component, one direction, or one entry of the gradient can force regularity under suitable critical hypotheses \cite{KukavicaZiane2006,KukavicaZiane2007,ZhouPokorny2009,ZhouPokorny2010,CaoTiti2011,CheminZhang2016,CheminZhangZhang2017,KukavicaRusinZiane2017,HanLeiLiZhao2019,KangNguyen2023}.  The present manuscript does not assume one-component smallness.  Its use of ``visible'' and ``invisible'' directions grew out of the same lesson: a possible singular branch should be studied through the structure that remains after the known regularizing channels have been tested.  The finite-scale harmonic-pressure viewpoint in \cite{Yu2026HarmonicPressure} is related to the pressure-cleaning convention used here, but the present paper shifts the emphasis from one-component approximation to full defect-cascade realizability.

A fifth source of motivation is the defect-measure and energy-flux viewpoint.  Constantin--E--Titi \cite{ConstantinETiti1994}, Eyink \cite{Eyink1994}, and Duchon--Robert \cite{DuchonRobert2000} show that coarse graining and commutator fluxes provide a natural way to measure anomalous energy transfer for weak solutions.  Leslie and Shvydkoy \cite{LeslieShvydkoy2018} use an energy measure to quantify possible failure of energy equality near blow-up.  The quantity
\[
  \Pi_{n,\ell}=-R_{n,\ell}:\nabla U_{n,\ell}
\]
used here belongs to this circle of ideas.  The additional point in the present paper is that energy flux is not treated as a standalone defect.  It is combined with pressure compatibility, positive covariance, local energy balance, and selected-time trace duality.

The pressure side of the framework is also tied to the regularity theory of local pressure.  The decomposition into an active part generated by the local quadratic source and a retained harmonic part is consistent with pressure regularity and local pressure analyses of Sohr--von Wahl, Seregin--\v{S}ver\'ak, Seregin, and Wolf \cite{SohrWahl1986,SereginSverak2002,Seregin2015,Wolf2017}.  This is why the paper quotients only the genuine pressure gauge, namely addition of a function of time, and not the whole local harmonic pressure component.

Finally, Tao's averaged Navier--Stokes blow-up construction \cite{Tao2016} gives an important cautionary lesson.  It shows that energy cancellation and harmonic-analysis-type estimates alone do not capture all the nonlinear structure needed for the true Navier--Stokes problem.  This is one reason that the present framework insists on \emph{NS-realizability}: a formal kernel vector or a formal defect envelope is not regarded as a genuine obstruction unless it lies in the finite-window closure of defect packages generated by actual suitable weak solutions.

Thus the main contribution of this paper is not the individual existence of coarse-grained stresses, pressure decompositions, energy fluxes, compactness alternatives, finite-dimensional observability estimates, or parabolic adjoint maps.  These ingredients all have predecessors, including standard compactness tools such as Simon's compactness theorem \cite{Simon1986} and linear parabolic theory for adjoint evolutions \cite{Amann1995,Lunardi1995}.  The new formulation is the combined reduction
\[
\boxed{
\begin{gathered}
\text{non-CKN dyadic branch} \\
\Longrightarrow \\
\text{non-effective moving-window observability} \\
\text{or NS-realizable combined-invisible defect cascade.}
\end{gathered}}
\]
In this formulation, the remaining obstruction is required to be simultaneously scale-critical, cleaned of gauge and localization artifacts, invisible to active pressure, flux, energy, and adjoint trace observations, and compatible with the Navier--Stokes coarse momentum, pressure, flux, positivity, and local-energy structures.

\subsection*{Honest status}

This paper does not prove the global Navier--Stokes regularity conjecture.  It does not prove that all potential singularities are impossible.  It also does not construct a singular solution.

The proved content is a structural reduction and a set of finite-window model verifications.  The reduction is designed to make the remaining obstruction precise.  A possible singularity, if it survives the combined observability mechanism, must be much more organized than an arbitrary concentration of energy or pressure.  It must be an NS-realizable scale-critical defect cascade that remains invisible to energy, active pressure, flux, and adjoint trace observations while preserving the coarse-grained Navier--Stokes identities, pressure compatibility, Reynolds stress structure, and local energy inequality.

Thus the broad question ``Can a Navier--Stokes singularity occur?'' is transformed into the more structured question ``Can there exist an NS-realizable invisible moving-window defect cascade?''  This is the central object identified in the paper.

Equivalently, at the conceptual level, the output of the paper is the following reduction:
\[
\boxed{
\begin{gathered}
\text{non-CKN dyadic branch}\\
\Longrightarrow\\
\text{non-effective observability}\\
\text{or NS-realizable invisible defect cascade.}
\end{gathered}}
\]
Thus a positive route must prove dyadic defect extraction, observable depletion, finite-window NS-realizability exclusion, and moving-window growth control.  A counterexample route must construct a cleaned, NS-realizable, scale-critical cascade invisible to active pressure, flux, energy, and adjoint trace while preserving the coarse-grained Navier--Stokes identities and the local energy inequality.

\section{Dyadic coarse-grained defect packages}
\label{sec:dyadic-defect-packages}

The purpose of this section is to define the scale-local objects used throughout the paper.  We do not assume smallness of any velocity component, nor do we compare the solution with a lower-dimensional shadow class.  Instead, to each suitable weak solution and each dyadic scale we associate a coarse-grained Navier--Stokes system with an exact Reynolds defect, an exact pressure compatibility law, and an exact coarse-grained energy-flux identity.  These objects form the first version of the scale-critical defect package.

\subsection{Suitable weak solutions and scale-invariant quantities}

Let
\[
  Q_r(z_0)=B_r(x_0)\times(t_0-r^2,t_0),
  \qquad z_0=(x_0,t_0),
\]
and write \(Q_r=Q_r(0,0)\).  We consider suitable weak solutions \((u,p)\) of
\begin{equation}
  \partial_t u-\Delta u+\nabla\cdot(u\otimes u)+\nabla p=0,
  \qquad
  \nabla\cdot u=0
  \label{eq:NS}
\end{equation}
in a parabolic cylinder.  Thus
\[
  u\in L_t^\infty L_x^2\cap L_t^2H_x^1,
  \qquad
  p\in L^{3/2},
\]
the equations hold distributionally, and the local energy inequality holds in the Caffarelli--Kohn--Nirenberg sense \cite{Leray1934,Hopf1951,Scheffer1976,CKN1982}.

For \(z_0=(x_0,t_0)\), define the scale-invariant quantities
\[
  A(z_0,r)
  :=
  \esssup_{t_0-r^2<t<t_0}
  \frac1r\int_{B_r(x_0)} |u(x,t)|^2\,dx,
\]
\[
  E(z_0,r)
  :=
  \frac1r\int_{Q_r(z_0)} |\nabla u|^2\,dx\,dt,
\]
\[
  C(z_0,r)
  :=
  \frac1{r^2}\int_{Q_r(z_0)} |u|^3\,dx\,dt,
\]
and
\[
  D(z_0,r)
  :=
  \frac1{r^2}\int_{Q_r(z_0)}
  |p-(p)_{B_r(x_0)}(t)|^{3/2}\,dx\,dt.
\]
We also write
\[
  \Phi(z_0,r)=A(z_0,r)+E(z_0,r)+C(z_0,r)+D(z_0,r),
  \qquad
  \Psi(z_0,r)=C(z_0,r)+D(z_0,r).
\]
At the origin we suppress \(z_0\) from the notation.

The Caffarelli--Kohn--Nirenberg criterion and its later variants give a universal constant \(\eps_{\rm CKN}>0\) such that if
\[
  \Psi(z_0,r)\le \eps_{\rm CKN},
\]
then \(u\) is regular in a smaller cylinder \cite{CKN1982,Lin1998,Vasseur2007,Seregin2007Local,GustafsonKangTsai2007,AlbrittonBarkerPrange2023}.  Hence a potential singular point is one at which no sufficiently small scale is known to enter the CKN regime.  The role of the defect package is to record what remains at those scales.

\subsection{Dyadic rescaling near a potential singular point}

Fix a point \(z_0=(x_0,t_0)\).  After translation we may assume \(z_0=(0,0)\).  Let
\[
  r_n=2^{-n},\qquad n\ge0.
\]
The Navier--Stokes rescaling at scale \(r_n\) is
\[
  u^{(n)}(x,t)=r_nu(r_nx,r_n^2t),
  \qquad
  p^{(n)}(x,t)=r_n^2p(r_nx,r_n^2t).
\]
Then \((u^{(n)},p^{(n)})\) is again a suitable weak solution in the rescaled cylinder, and
\[
  A_{u^{(n)}}(1)=A_u(r_n),
  \qquad
  E_{u^{(n)}}(1)=E_u(r_n),
\]
\[
  C_{u^{(n)}}(1)=C_u(r_n),
  \qquad
  D_{p^{(n)}}(1)=D_p(r_n).
\]
Thus each dyadic scale near \(z_0\) is converted into a unit-scale problem.

We do not assume a uniform bound on \(\Phi(r_n)\) at this stage.  If a sequence of scales has \(\sup_n\Phi(r_n)<\infty\), then the corresponding packages are uniformly scale-critical.  If \(\Phi(r_n)\to\infty\), that growth is itself part of the concentration profile.  The formalism below is valid in either case, with constants depending on the scale under consideration.

\subsection{Interior coarse graining}

Let \(\rho\in C_c^\infty(\R^3\times\R)\) be a nonnegative parabolic mollifier with
\[
  \int_{\R^3\times\R}\rho(x,t)\,dx\,dt=1.
\]
For \(0<\ell<1\), set
\[
  \rho_\ell(x,t)=\ell^{-5}\rho(x/\ell,t/\ell^2),
\]
and define
\[
  S_\ell f=\rho_\ell*f.
\]
All identities below are understood in an interior cylinder \(Q_{\rm obs}\Subset Q_{\rm prep}\), with \(\ell\) smaller than the parabolic distance from \(Q_{\rm obs}\) to \(\partial Q_{\rm prep}\).  Thus convolution commutes with derivatives in \(Q_{\rm obs}\).  Local cutoffs and the commutators they generate are treated later; this coarse-graining viewpoint is parallel to the commutator-flux formalism used in energy-conservation and anomalous-dissipation problems \cite{ConstantinETiti1994,Eyink1994,DuchonRobert2000}.

For the rescaled solution \((u^{(n)},p^{(n)})\), define
\[
  U_{n,\ell}=S_\ell u^{(n)},
  \qquad
  P_{n,\ell}=S_\ell p^{(n)}.
\]
When no confusion is possible, we write simply \(U=U_{n,\ell}\) and \(P=P_{n,\ell}\).  Since \(S_\ell\) commutes with spatial derivatives in the interior,
\[
  \nabla\cdot U=0.
\]

\subsection{The Reynolds defect}

The coarse-grained nonlinear flux is not equal to the nonlinear flux of the coarse-grained velocity.  The difference is the Reynolds defect
\begin{equation}
  R_{n,\ell}
  :=
  S_\ell(u^{(n)}\otimes u^{(n)})
  -U_{n,\ell}\otimes U_{n,\ell}.
  \label{eq:Reynolds-defect}
\end{equation}
Equivalently,
\[
  S_\ell(u^{(n)}\otimes u^{(n)})
  =U_{n,\ell}\otimes U_{n,\ell}+R_{n,\ell}.
\]
The tensor \(R_{n,\ell}\) is symmetric and, because the mollifier is nonnegative, nonnegative as a quadratic form:
\[
  R_{n,\ell}(x,t)\ge0.
\]
Indeed, for every vector \(a\in\R^3\),
\[
  a_iR_{n,\ell,ij}a_j
  =S_\ell\bigl((a\cdot u^{(n)})^2\bigr)
  -\bigl(S_\ell(a\cdot u^{(n)})\bigr)^2
  \ge0,
\]
by Jensen's inequality.  In particular,
\[
  \kappa_{n,\ell}:=\frac12\tr R_{n,\ell}\ge0
\]
is the unresolved kinetic energy density.

\begin{definition}[Dyadic Reynolds defect]
For each dyadic scale \(r_n\) and coarse-graining scale \(\ell\), the tensor \(R_{n,\ell}\) defined by \eqref{eq:Reynolds-defect} is called the \emph{dyadic Reynolds defect} of \((u,p)\) at \((r_n,\ell)\).
\end{definition}

\subsection{Coarse-grained Navier--Stokes identity}

Applying \(S_\ell\) to the Navier--Stokes equations for \((u^{(n)},p^{(n)})\) gives
\[
  \partial_t U-\Delta U+
  \nabla\cdot S_\ell(u^{(n)}\otimes u^{(n)})+
  \nabla P=0.
\]
Using
\[
  S_\ell(u^{(n)}\otimes u^{(n)})=U\otimes U+R,
\]
we obtain the exact coarse-grained momentum identity
\begin{equation}
  \partial_t U-\Delta U+
  \nabla\cdot(U\otimes U)+\nabla P
  =
  -\nabla\cdot R,
  \qquad
  \nabla\cdot U=0.
  \label{eq:coarse-momentum}
\end{equation}
Thus the coarse-grained velocity \(U\) solves Navier--Stokes with an internal stress forcing \(-\nabla\cdot R\).  This forcing is not arbitrary; it is generated by the unresolved part of the same Navier--Stokes solution.

\begin{proposition}[NS-realizable coarse-grained momentum balance]
Let \((u,p)\) be a suitable weak solution and let \((U_{n,\ell},P_{n,\ell},R_{n,\ell})\) be defined as above.  Then, in every interior observation cylinder,
\[
  \partial_t U_{n,\ell}
  -\Delta U_{n,\ell}
  +\nabla\cdot(U_{n,\ell}\otimes U_{n,\ell})
  +\nabla P_{n,\ell}
  =
  -\nabla\cdot R_{n,\ell},
\]
and
\[
  \nabla\cdot U_{n,\ell}=0.
\]
\end{proposition}

\begin{proof}
Apply the smoothing operator \(S_\ell\) to the distributional Navier--Stokes equations and use the definition of \(R_{n,\ell}\).  Since the identity is stated in an interior cylinder, no cutoff commutator appears.
\end{proof}

\subsection{Pressure compatibility, active pressure, and harmonic pressure}

Taking divergence in \eqref{eq:coarse-momentum} and using \(\nabla\cdot U=0\), we get the pressure compatibility law
\begin{equation}
  -\Delta P
  =
  \partial_i\partial_j
  \bigl(U_iU_j+R_{ij}\bigr).
  \label{eq:coarse-pressure-poisson}
\end{equation}
Equivalently,
\[
  -\Delta P
  =
  \partial_i\partial_jS_\ell(u_i^{(n)}u_j^{(n)}).
\]

There are two different pressure issues which must be kept separate.  The true Navier--Stokes pressure gauge is the addition of a function of time, since such a term has zero spatial gradient.  A spatially harmonic local pressure component is not a gauge: its gradient enters the momentum equation and its pressure work enters the local energy identity.  Therefore the local harmonic pressure component is not quotiented out of the defect package.  This convention follows the standard local-pressure distinction between the source-generated part and the harmonic remainder \cite{SohrWahl1986,SereginSverak2002,Seregin2015,Wolf2017}.

In a local cylinder \(Q\), let
\[
  \calH(Q)=
  \{h\in L^{3/2}(Q):\Delta h(\cdot,t)=0
  \text{ in the spatial ball for a.e. }t\}
\]
and let
\[
  \calH_0(Q)=\{a(t):a\in L^{3/2}_{\rm loc}\}
\]
be the true pressure-gauge subspace.  A finite-window pressure cleaning means a fixed bounded choice of an active representative
\[
  P^{\act}=\mathsf P_Q^{\act}P,
  \qquad
  P^{\har}=P-P^{\act},
\]
such that
\[
  -\Delta P^{\act}
  =
  \partial_i\partial_j(U_iU_j+R_{ij})
\]
in the chosen local representative and \(\Delta P^{\har}=0\) in the observation ball.  The component \(P^{\act}\) is the pressure generated by the active quadratic source.  The component \(P^{\har}\) records the harmonic local pressure field and is retained in the momentum and energy identities.

\begin{definition}[Active pressure representative and retained harmonic pressure]
The active pressure component of the dyadic defect package is the chosen representative
\[
  P^{\act}_{n,\ell}=\mathsf P_Q^{\act}P_{n,\ell},
\]
and the retained harmonic pressure component is
\[
  P^{\har}_{n,\ell}=P_{n,\ell}-P^{\act}_{n,\ell}.
\]
Only the residual addition of a function of time is treated as a pressure gauge.  Spatially harmonic pressure components are not null directions unless an observation channel explicitly projects them away and the resulting loss is recorded as a harmonic-pressure error in the relevant estimate.
\end{definition}

This pressure compatibility law is the first place where the Navier--Stokes structure imposes a constraint among the objects in the defect package.  The tensor \(R\) cannot be chosen independently of \(P\); its double divergence contributes directly to the active pressure source, while the harmonic pressure component remains part of the local dynamics.
\subsection{Coarse-grained energy flux}

Let
\[
  e_U=\frac12|U|^2.
\]
Taking the scalar product of \eqref{eq:coarse-momentum} with \(U\), using \(\nabla\cdot U=0\), and writing transport terms in divergence form gives
\begin{equation}
  \partial_t e_U
  -\Delta e_U
  +|\nabla U|^2
  +
  \nabla\cdot
  \left[
    (e_U+P)U+RU
  \right]
  =
  R:\nabla U.
  \label{eq:coarse-energy}
\end{equation}
Here \((RU)_j=R_{ij}U_i\).  We define the coarse-grained energy flux by
\begin{equation}
  \Pi_{n,\ell}
  :=
  -R_{n,\ell}:\nabla U_{n,\ell}.
  \label{eq:energy-flux}
\end{equation}
With this sign convention, \eqref{eq:coarse-energy} becomes
\begin{equation}
  \partial_t e_U
  -\Delta e_U
  +|\nabla U|^2
  +
  \nabla\cdot
  \left[
    (e_U+P)U+RU
  \right]
  =
  -\Pi_{n,\ell}.
  \label{eq:coarse-energy-flux}
\end{equation}
Thus \(\Pi_{n,\ell}\) measures the energy exchanged between the resolved velocity and the unresolved stress.  It is the natural local flux quantity associated with a scale-critical cascade.

\begin{proposition}[Exact coarse-grained energy-flux identity]
For every dyadic scale \(r_n\) and every interior coarse-graining scale \(\ell\),
\[
  \partial_t \frac12|U_{n,\ell}|^2
  -\Delta \frac12|U_{n,\ell}|^2
  +|\nabla U_{n,\ell}|^2
  +
  \nabla\cdot
  \left[
  \left(\frac12|U_{n,\ell}|^2+P_{n,\ell}\right)U_{n,\ell}
  +R_{n,\ell}U_{n,\ell}
  \right]
  =
  -\Pi_{n,\ell}.
\]
Equivalently, \(\Pi_{n,\ell}=-R_{n,\ell}:\nabla U_{n,\ell}\).
\end{proposition}

\begin{proof}
Dot \eqref{eq:coarse-momentum} with \(U\).  Since \(\nabla\cdot U=0\),
\[
  U\cdot\partial_tU=\partial_t\frac12|U|^2,
\]
\[
  -U\cdot\Delta U=-\Delta\frac12|U|^2+|\nabla U|^2,
\]
\[
  U\cdot\nabla\cdot(U\otimes U)=
  \nabla\cdot\left(\frac12|U|^2U\right),
\]
and
\[
  U\cdot\nabla P=\nabla\cdot(PU).
\]
For the Reynolds term,
\[
  -U_i\partial_jR_{ij}
  =
  -\partial_j(U_iR_{ij})+R_{ij}\partial_jU_i.
\]
Putting these identities together gives \eqref{eq:coarse-energy}.  The definition of \(\Pi\) gives \eqref{eq:coarse-energy-flux}.
\end{proof}

\subsection{Scale-critical bounds for the defect package}

The objects above inherit scale-critical bounds from the original solution.  These estimates are not smallness statements.  Their role is to show that the package is controlled at the same scaling as the Navier--Stokes equation.

\begin{lemma}[Basic scale-critical bounds]
Let \((u,p)\) be a suitable weak solution, and let the corresponding dyadic coarse-grained package be defined as above.  Then, in a fixed interior observation cylinder,
\[
  \norm{U_{n,\ell}}{L^3}
  \le
  \norm{u^{(n)}}{L^3},
\]
\[
  \norm{\nabla U_{n,\ell}}{L^2}
  \le
  \norm{\nabla u^{(n)}}{L^2},
\]
\[
  \norm{R_{n,\ell}}{L^{3/2}}
  \le
  C\norm{u^{(n)}}{L^3}^2,
\]
and, for the chosen active pressure representative,
\[
  \norm{P^{\act}_{n,\ell}}{L^{3/2}}
  \le
  C\norm{p^{(n)}-(p^{(n)})_{B_1}(t)}{L^{3/2}}.
\]
The full pressure is still present as \(P_{n,\ell}=P^{\act}_{n,\ell}+P^{\har}_{n,\ell}\); no estimate above declares \(P^{\har}_{n,\ell}\) to be a gauge-null field.  Moreover, by smoothing,
\[
  \norm{\nabla U_{n,\ell}}{L^3}
  \le
  C\ell^{-1}\norm{u^{(n)}}{L^3},
\]
and therefore
\[
  \norm{\Pi_{n,\ell}}{L^1}
  \le
  C\ell^{-1}\norm{u^{(n)}}{L^3}^3.
\]
In scale-invariant notation,
\[
  \norm{U_{n,\ell}}{L^3}^3\lesssim C(r_n),
  \qquad
  \norm{R_{n,\ell}}{L^{3/2}}^{3/2}\lesssim C(r_n),
\]
\[
  \norm{P^{\act}_{n,\ell}}{L^{3/2}}^{3/2}\lesssim D(r_n),
  \qquad
  \norm{\Pi_{n,\ell}}{L^1}\lesssim \ell^{-1}C(r_n).
\]
\end{lemma}

\begin{proof}
The \(L^3\), \(L^2\), and \(L^{3/2}\) estimates follow from the boundedness of convolution on Lebesgue spaces and the definition of the scale-invariant quantities.  For the Reynolds stress,
\[
  \norm{R}{L^{3/2}}
  \le
  \norm{S_\ell(u^{(n)}\otimes u^{(n)})}{L^{3/2}}
  +
  \norm{U\otimes U}{L^{3/2}}
  \le
  C\norm{u^{(n)}}{L^3}^2.
\]
The smoothing estimate
\[
  \norm{\nabla U}{L^3}\le C\ell^{-1}\norm{u^{(n)}}{L^3}
\]
is standard.  Hence
\[
  \norm{\Pi}{L^1}
  =
  \norm{R:\nabla U}{L^1}
  \le
  \norm{R}{L^{3/2}}\norm{\nabla U}{L^3}
  \le
  C\ell^{-1}\norm{u^{(n)}}{L^3}^3.
\]
The scale-invariant forms follow from the definition of \(u^{(n)}\) and \(p^{(n)}\).
\end{proof}

\subsection{The NS-realizable dyadic defect package}

We now collect the preceding objects into a single definition.

\begin{definition}[NS-realizable dyadic defect package]
Let \((u,p)\) be a suitable weak solution, let \(z_0\) be a point, let \(r_n=2^{-n}\), and let \(0<\ell<\ell_0\).  The NS-realizable dyadic defect package at \((z_0,r_n,\ell)\) is
\[
  \frakD_{n,\ell}(u,p;z_0)
  :=
  \left(
  U_{n,\ell},
  P_{n,\ell};
  P^{\act}_{n,\ell},
  P^{\har}_{n,\ell},
  R_{n,\ell},
  \Pi_{n,\ell}
  \right),
\]
where
\[
  U_{n,\ell}=S_\ell u^{(n)},
  \qquad
  P_{n,\ell}=S_\ell p^{(n)},
\]
\[
  P_{n,\ell}=P^{\act}_{n,\ell}+P^{\har}_{n,\ell},
  \qquad
  \Delta P^{\har}_{n,\ell}=0
  \quad\text{locally in the chosen pressure window},
\]
\[
  R_{n,\ell}
  =
  S_\ell(u^{(n)}\otimes u^{(n)})
  -U_{n,\ell}\otimes U_{n,\ell},
\]
and
\[
  \Pi_{n,\ell}
  =
  -R_{n,\ell}:\nabla U_{n,\ell}.
\]
The package is called NS-realizable because it arises from an actual suitable weak solution through Navier--Stokes scaling and coarse graining.
\end{definition}

The package satisfies the three structural identities
\[
  \partial_t U-\Delta U+
  \nabla\cdot(U\otimes U)+\nabla P
  =
  -\nabla\cdot R,
\]
\[
  -\Delta P
  =
  \partial_i\partial_j(U_iU_j+R_{ij}),
  \qquad
  P=P^{\act}+P^{\har},
  \qquad
  \Delta P^{\har}=0
  \quad\text{locally},
\]
and
\[
  \partial_t\frac12|U|^2
  -\Delta\frac12|U|^2
  +|\nabla U|^2
  +
  \nabla\cdot\left[
    \left(\frac12|U|^2+P\right)U+RU
  \right]
  =
  -\Pi.
\]
These identities show that \(U\), \(P\), \(R\), and \(\Pi\) are not independent observables.  The Reynolds defect determines the active pressure source, the local harmonic pressure component contributes to momentum and pressure work, and the same Reynolds defect interacts with the resolved strain to produce the energy flux.
\subsection{Potential singularities and scale-critical defect cascades}

The preceding construction applies at every point.  It becomes significant near a potential singular point.  Suppose \(z_0\) is not known to be regular.  If no dyadic scale \(r_n\) satisfies the CKN smallness condition
\[
  \Psi(z_0,r_n)\le \eps_{\rm CKN},
\]
then the sequence of packages
\[
  \left\{\frakD_{n,\ell_n}(u,p;z_0)\right\}_{n\ge n_0}
\]
records the scale-critical information that persists along the possible singular cascade.

The central problem may now be stated informally:
\begin{quote}
\centering
\textit{Can an NS-realizable sequence of dyadic defect packages remain scale-critical while escaping all combined observation channels?}
\end{quote}
Equivalently, if all nonzero NS-realizable dyadic defects are detected by pressure, flux, energy, or adjoint trace observations with summable constants, then a CKN scale should eventually appear.  Conversely, if no CKN scale appears, then one should expect to extract a scale-critical defect cascade that is NS-realizable and invisible to the combined observability mechanism.

This section constructs the objects and proves their exact identities.  The next step is to define the relevant observation maps acting on \(\frakD_{n,\ell}\).

\section{Combined observability maps}
\label{sec:combined-observability}

In the previous section we associated to each dyadic scale and each coarse-graining length an NS-realizable defect package
\[
  \frakD_{n,\ell}=
  \left(U_{n,\ell},P_{n,\ell};P^{\act}_{n,\ell},P^{\har}_{n,\ell},R_{n,\ell},\Pi_{n,\ell}\right).
\]
We now specify what it means for such a package, or for a normalized defect direction extracted from it, to be observed.  The purpose of this section is structural: we define finite observation windows, the active defect quotient, four observation channels, the combined observability map, and the corresponding notion of an invisible defect direction.  No uniform or summable observability estimate is asserted here.

The four observation channels are
\[
  \text{energy},
  \qquad
  \text{active pressure},
  \qquad
  \text{flux},
  \qquad
  \text{adjoint trace}.
\]
Together they form the combined observability map.  A scale-critical defect cascade is dangerous only if, after all exact gauge and exact null directions have been removed, and after perturbative localization, CKN-small, and harmonic-pressure leakage errors have been accounted for, it remains invisible to this combined map.

\subsection{Finite observation windows}

A finite observation window is a tuple
\[
  W=(n,\ell,\Lambda,\chi,s_*),
\]
where \(n\) is the dyadic scale index, \(0<\ell<1\) is the coarse-graining length in the rescaled unit cylinder, \(\Lambda\) is a finite active frequency or finite-dimensional test window, \(\chi\in C_c^\infty(Q_{\rm prep})\) is a cutoff with \(\chi\equiv1\) on \(Q_{\rm obs}\Subset Q_{\rm prep}\), and \(s_*\in(-1,0)\) is a selected time when selected-time or adjoint-trace observation is used.

The role of \(\Lambda\) is not necessarily Fourier-theoretic.  In a periodic model it may be a finite Fourier set.  In a localized cylinder it may be a finite-dimensional space of test functions, wave packets, Stokes eigenfunctions, localized Fourier packets, or adjoint test modes.  We only require that all spaces associated with \(W\) be finite-dimensional.

Let
\[
  \Pi_W^{\rm proj}
\]
denote the finite-window projection acting on the dyadic defect package.  It extracts the part of
\[
  (U,P;P^{\act},P^{\har},R,\Pi)
\]
visible at the window \(W\).  We write
\[
  \Pi_W^{\rm proj}\frakD_{n,\ell}
  =
  \left(U_W,P_W;P_W^{\act},P_W^{\har},R_W,\Pi_W^{\rm flux}\right).
\]

\subsection{The constrained active defect quotient}

The raw finite-window variables contain two different kinds of removable pieces.  Exact null directions may be quotiented.  Perturbative effects, such as cutoff leakage, CKN-small pieces, finite-window tails, and harmonic-pressure leakage into channels that do not observe it, are not quotiented; they are recorded as errors in the corresponding estimates.

Let \(\calD_W^{\rm raw}\) be the finite-dimensional space of projected variables
\[
  \dot{\frakD}_W=(\dot U,\dot P;\dot P^{\act},\dot P^{\har},\dot R,\dot\Pi)
\]
supported in the window \(W\).  The dot notation indicates a finite-window defect direction, not necessarily a full nonlinear package.  We first restrict to the constrained finite-window tangent space \(\calZ_W\subset\calD_W^{\rm raw}\) consisting of directions satisfying, in the projected window sense,
\[
  \nabla\cdot \dot U=0,
\]
\[
  \partial_t\dot U-\Delta\dot U
  +\nabla\cdot(\dot U\otimes U+U\otimes\dot U)
  +\nabla\dot P
  =
  -\nabla\cdot\dot R,
\]
\[
  -\Delta\dot P
  =
  \partial_i\partial_j
  \bigl(\dot U_iU_j+U_i\dot U_j+\dot R_{ij}\bigr),
  \qquad
  \dot P=\dot P^{\act}+\dot P^{\har},
  \qquad
  \Delta\dot P^{\har}=0
  \quad\text{locally},
\]
and
\[
  \dot\Pi=-\dot R:\nabla U-R:\nabla\dot U.
\]
Thus pressure compatibility and the flux identity are constraints on admissible defect directions, not identities imposed after quotienting arbitrary raw variables.

The exact quotient-null subspace is denoted by
\[
  \calG_W^{\rm ex}
  =
  \calG_W^{\rm div}
  +
  \calG_W^{\rm time}.
\]
Here \(\calG_W^{\rm div}\) denotes exact Leray-projection null directions already removed by the divergence-free projection, and \(\calG_W^{\rm time}\) denotes additions of functions of time to the pressure.  These are the pressure-gauge directions with zero spatial gradient.  Spatially harmonic pressure functions are not placed in \(\calG_W^{\rm ex}\).

The cleaned active defect space is
\begin{equation}
  Y_W=\calZ_W/\calG_W^{\rm ex}.
  \label{eq:active-defect-space}
\end{equation}
We denote the cleaned class of \(\dot{\frakD}_W\) by \([\dot{\frakD}_W]\in Y_W\).  The norm on \(Y_W\) is fixed by choosing finite-window representatives orthogonal to \(\calG_W^{\rm ex}\):
\[
  \norm{[\dot{\frakD}_W]}{Y_W}
  =
  \inf_{\mathfrak g\in\calG_W^{\rm ex}}
  \norm{\dot{\frakD}_W+\mathfrak g}{\calZ_W}.
\]

The perturbative sectors
\[
  \calN_W^{\rm loc},\qquad
  \calN_W^{\rm ckn},\qquad
  \calN_W^{\rm har},\qquad
  \calN_W^{\rm tail}
\]
are not quotient-null spaces.  Their contributions enter estimates through bounds of the form
\[
  \|\widetilde{\calO}_W n\|
  \le
  \eps_W\|n\|
\]
for the relevant observation channel.  This distinction is essential: a small direction is not the same thing as a zero direction.

\begin{definition}[Cleaned dyadic defect direction]
Given an NS-realizable dyadic defect package \(\frakD_{n,\ell}\), its cleaned finite-window defect direction is
\[
  d_W=[\Pi_W^{\rm proj}\frakD_{n,\ell}]\in Y_W.
\]
If \(d_W\ne0\), its normalized direction is
\[
  \widehat d_W=\frac{d_W}{\norm{d_W}{Y_W}}.
\]
\end{definition}

Thus the object observed by the combined map is not the entire solution and not a raw pressure quotient.  It is the constrained, cleaned active defect direction \(d_W\), with local harmonic pressure retained unless a specific observation channel explicitly projects it away.

\begin{definition}[Admissible cleaned observation window]
\label{def:admissible-cleaned-window}
Let
\[
  \widetilde{\calO}^{E}_W,
  \quad
  \widetilde{\calO}^{P}_W,
  \quad
  \widetilde{\calO}^{F}_W,
  \quad
  \widetilde{\calO}^{T}_W
\]
be the raw finite-window observation maps on \(\calZ_W\).  The window is called exactly cleaned if
\[
  \widetilde{\calO}^{E}_W\mathfrak g
  =
  \widetilde{\calO}^{P}_W\mathfrak g
  =
  \widetilde{\calO}^{F}_W\mathfrak g
  =
  \widetilde{\calO}^{T}_W\mathfrak g
  =0
  \qquad
  \text{for every }\mathfrak g\in\calG_W^{\rm ex}.
\]
It is called perturbatively cleaned with error \(\eps_W\) if, in addition, every declared localization, CKN-small, harmonic-leakage, or tail component is bounded by the relevant perturbative observation error rather than placed in the quotient.
\end{definition}

\begin{lemma}[Well-defined quotient observations]
\label{lem:quotient-observations-well-defined}
In an exactly cleaned finite window, the four observation maps descend to continuous maps on the quotient \(Y_W=\calZ_W/\calG_W^{\rm ex}\).  In a perturbatively cleaned window, the same maps are well-defined on \(Y_W\) up to the explicitly recorded perturbative errors.
\end{lemma}

\begin{proof}
If two constrained representatives differ by an element of \(\calG_W^{\rm ex}\), exact cleaning gives the same value for every raw observation map.  Thus the observations depend only on the quotient class.  Perturbative sectors are not quotiented, so their small but nonzero contributions remain as error terms.  Continuity is automatic because all window spaces are finite-dimensional.
\end{proof}
\subsection{Energy observation}

The energy channel records the part of the defect visible through the local energy balance.  Let
\[
  \Theta_W^E=\operatorname{span}\{\varphi_1,\ldots,\varphi_{N_E}\}
  \subset C_c^\infty(Q_{\rm obs})
\]
be a finite-dimensional space of energy test functions.  For a smooth coarse package
\[
  \frakD_W=(U,P;P^{\act},P^{\har},R,\Pi),
  \qquad P=P^{\act}+P^{\har},
\]
define the coarse energy functional
\[
  \mathscr E_{\frakD_W}(\varphi)
  =
  \iint
  \left[
    -e_U(\partial_t+\Delta)\varphi
    -\left((e_U+P)U+RU\right)\cdot\nabla\varphi
    +\left(|\nabla U|^2+\Pi\right)\varphi
  \right]\,dx\,dt,
\]
where \(e_U=\frac12|U|^2\).  For an exact smooth coarse-grained solution this quantity vanishes by the coarse-grained energy identity.  For a defect direction, its linearization detects how the local energy balance changes in that direction.

We define the energy observation map
\[
  \calO_W^E:Y_W\longrightarrow E_W
\]
by
\begin{equation}
  \calO_W^E(d)
  =
  \left(
  D\mathscr E_{\frakD_W}[\dot{\frakD}_W](\varphi_j)
  \right)_{j=1}^{N_E},
  \label{eq:energy-observation}
\end{equation}
where \(\dot{\frakD}_W\) is any representative of \(d\).  The definition is independent of the representative after quotienting by the exact null space \(\calG_W^{\rm ex}\).  In particular, adding a function of time to \(P\) contributes
\[
  -\iint a(t)U\cdot\nabla\varphi\dxdt=0
\]
after integration by parts and \(\nabla\cdot U=0\).  The target space is \(E_W\simeq\R^{N_E}\) with a fixed Euclidean norm.

\subsection{Active pressure observation}

The pressure channel records the non-harmonic part of the pressure generated by the active quadratic source.  It does not declare the local harmonic pressure component to be physically irrelevant; that component remains available to the momentum, energy, and trace channels.

Let
\[
  \Theta_W^P=\operatorname{span}\{\psi_1,\ldots,\psi_{N_P}\}
\]
be a finite-dimensional space of pressure test functions chosen for the active pressure representative.  Equivalently, \(\Theta_W^P\) is taken in the dual of the active pressure range produced by the chosen pressure-cleaning operator.

For a constrained defect direction
\[
  d=[\dot U,\dot P;\dot P^{\act},\dot P^{\har},\dot R,\dot\Pi]\in Y_W,
\]
define
\[
  \calO_W^P d
  =
  \left(
    \pair{\dot P^{\act}}{\psi_j}
  \right)_{j=1}^{N_P}.
\]
Because \(d\in\calZ_W\), the active representative may equivalently be expressed through the linearized active pressure source:
\[
  -\Delta \dot P^{\act}
  =
  \partial_i\partial_j
  \left(
    \dot U_iU_j+U_i\dot U_j+\dot R_{ij}
  \right)
\]
in the chosen active pressure representative.  Therefore
\[
  \calO_W^P d
  =
  \left(
  \left\langle
    \mathsf P_W^{\act}(-\Delta)^{-1}\partial_i\partial_j
    (\dot U_iU_j+U_i\dot U_j+\dot R_{ij}),
    \psi_j
  \right\rangle
  \right)_{j=1}^{N_P},
\]
with all inverse and projection operators understood inside the fixed finite-window pressure cleaning.  This representation is not asserted for arbitrary raw directions outside \(\calZ_W\).

\begin{remark}
The pressure observation is not the raw \(L^{3/2}\)-norm of \(P\), nor is it a quotient of the full pressure by all spatially harmonic functions.  It is the finite-window active pressure signal.  Harmonic pressure can be invisible to this channel while still affecting the energy channel through pressure work.
\end{remark}
\subsection{Flux observation}

The flux channel records the part of the defect visible through interscale energy transfer.  Recall that
\[
  \Pi=-R:\nabla U.
\]
For a constrained defect direction
\[
  d=[\dot U,\dot P;\dot P^{\act},\dot P^{\har},\dot R,\dot\Pi]\in Y_W,
\]
the linearized flux identity is
\begin{equation}
  \dot\Pi
  =
  -\dot R:\nabla U-R:\nabla\dot U.
  \label{eq:linearized-flux}
\end{equation}
Let
\[
  \Theta_W^F=\operatorname{span}\{\eta_1,\ldots,\eta_{N_F}\}
  \subset C_c^\infty(Q_{\rm obs})
\]
be a finite-dimensional space of flux test functions.  Define
\[
  \calO_W^F d
  =
  \left(
    \iint \dot\Pi\,\eta_j\,dx\,dt
  \right)_{j=1}^{N_F}.
\]
Equivalently,
\[
  \calO_W^F d
  =
  -\left(
  \iint
  \left(
    \dot R:\nabla U+R:\nabla\dot U
  \right)\eta_j\,dx\,dt
  \right)_{j=1}^{N_F}.
\]
The flux observation is the most direct channel for detecting a cascade.  If a defect is invisible to pressure but produces nonzero interscale transfer, then it is not a true invisible defect.

\subsection{Adjoint trace observation}

The adjoint trace channel measures whether the defect direction can be seen by selected-time perturbations.  Let \(H_W\) be the finite-dimensional selected-time trace correction space.  Elements of \(H_W\) are divergence-free vector fields on the selected time slice \(t=s_*\), localized in the observation ball and projected to the window \(\Lambda\).

Let
\[
  A_W:H_W\longrightarrow Y_W
\]
be the active trace-defect map.  Given \(\xi\in H_W\), one evolves \(\xi\) through the linearized coarse-grained Navier--Stokes system around \(U\), with pressure determined through the active/harmonic pressure split, and records the resulting active defect direction in \(Y_W\).  Thus \(A_W\xi\) is the active defect produced by the selected-time trace correction \(\xi\).

The dual map
\[
  A_W^*:Y_W^*\longrightarrow H_W
\]
is defined by
\[
  \pair{A_W\xi}{y}_{Y_W,Y_W^*}
  =
  \pair{\xi}{A_W^*y}_{H_W}.
\]
Equivalently, \(A_W^*y\) is obtained by solving the backward adjoint linearized Navier--Stokes equation with forcing \(y\) in the active defect quotient and reading off its trace at \(t=s_*\).  This finite-window construction is a localized, finite-dimensional version of standard parabolic duality and linearized parabolic theory \cite{Amann1995,Lunardi1995}.

\begin{remark}
The adjoint trace channel is naturally dual.  A dual vector \(y\in Y_W^*\) with \(A_W^*y=0\) is invisible to selected-time trace corrections.  Such a direction is not necessarily a true phantom, because it may still be visible through pressure, flux, or energy observations.
\end{remark}

\subsection{Primal and dual combined maps}

The primal combined observability map is
\[
  \calO_W^{\rm comb}:Y_W
  \longrightarrow
  E_W\oplus P_W\oplus F_W\oplus T_W,
\]
where
\begin{equation}
  \calO_W^{\rm comb}d
  =
  \left(
    \calO_W^E d,
    \calO_W^P d,
    \calO_W^F d,
    \calO_W^T d
  \right).
  \label{eq:primal-combined-map}
\end{equation}
Here \(\calO_W^T\) is the primal selected-time trace observation
\[
  \calO_W^T d=\Pi_W^T\dot U(s_*),
\]
where \(\Pi_W^T\) projects the selected-time velocity perturbation onto the finite trace window.

The primal observability constant is
\[
  M_W^{\rm pr}
  =
  \sup_{0\ne d\in Y_W}
  \frac{\norm{d}{Y_W}}
  {
  \norm{\calO_W^Ed}{E_W}
  +\norm{\calO_W^Pd}{P_W}
  +\norm{\calO_W^Fd}{F_W}
  +\norm{\calO_W^Td}{T_W}
  }.
\]
If the denominator vanishes for some nonzero \(d\), we set \(M_W^{\rm pr}=+\infty\).

For trace-cost and phantom analysis, the dual formulation is often more important.  Write
\[
  J_W^P:Y_W^*\to P_W^*,
  \qquad
  J_W^F:Y_W^*\to F_W^*,
  \qquad
  J_W^E:Y_W^*\to E_W^*
\]
for the dual observation components.  Thus \(J_W^Py\) is the part of the dual direction tested by the active pressure channel, \(J_W^Fy\) is the part tested by the flux channel, and \(J_W^Ey\) is the part tested by the energy channel.

The dual combined observability map is
\begin{equation}
  \frakO_W^*:Y_W^*
  \longrightarrow
  H_W\oplus E_W^*\oplus P_W^*\oplus F_W^*
  \label{eq:dual-combined-map}
\end{equation}
defined by
\begin{equation}
  \frakO_W^*y
  =
  \left(
    A_W^*y,
    J_W^Ey,
    J_W^Py,
    J_W^Fy
  \right).
  \label{eq:dual-map-components}
\end{equation}
The corresponding dual combined observability constant is
\[
  M_W^*
  =
  \sup_{0\ne y\in Y_W^*}
  \frac{\norm{y}{Y_W^*}}
  {
  \norm{A_W^*y}{H_W}
  +\norm{J_W^Ey}{E_W^*}
  +\norm{J_W^Py}{P_W^*}
  +\norm{J_W^Fy}{F_W^*}
  }.
\]
If the denominator vanishes for some nonzero \(y\), then \(M_W^*=+\infty\).

\subsection{Invisible directions and true defect phantoms}

\begin{definition}[Primal invisible defect direction]
A nonzero direction \(d\in Y_W\) is called primal-invisible in the window \(W\) if
\[
  \calO_W^E d=0,
  \qquad
  \calO_W^P d=0,
  \qquad
  \calO_W^F d=0,
  \qquad
  \calO_W^T d=0.
\]
The primal invisible space is
\[
  \calK_W^{\rm pr}
  =
  \ker\calO_W^E
  \cap\ker\calO_W^P
  \cap\ker\calO_W^F
  \cap\ker\calO_W^T
  \subset Y_W.
\]
\end{definition}

\begin{definition}[Dual invisible direction]
A nonzero direction \(y\in Y_W^*\) is called dual-invisible in the window \(W\) if
\[
  A_W^*y=0,
  \qquad
  J_W^Ey=0,
  \qquad
  J_W^Py=0,
  \qquad
  J_W^Fy=0.
\]
The dual invisible space is
\[
  \calK_W^*
  =
  \ker A_W^*
  \cap\ker J_W^E
  \cap\ker J_W^P
  \cap\ker J_W^F
  \subset Y_W^*.
\]
\end{definition}

A true defect phantom is not merely a vector in one of these kernels.  It must also be realized by the Navier--Stokes equation.

\begin{definition}[Finite-window true defect phantom]
A nonzero \(d\in Y_W\) is a finite-window true defect phantom if:
\begin{enumerate}
  \item \(d\in\calK_W^{\rm pr}\);
  \item \(d\) is the cleaned finite-window projection of an NS-realizable dyadic defect package, or of a normalized limit of such packages;
  \item \(d\) is not a true pressure-gauge direction, not merely retained harmonic-pressure leakage, not a localization leakage direction, and not already controlled by a CKN-small local norm.
\end{enumerate}
\end{definition}

Similarly, a dual true phantom is a nonzero \(y\in\calK_W^*\) that has nontrivial pairing with an NS-realizable defect:
\[
  |\pair{d_W}{y}|
  \ge
  c\norm{d_W}{Y_W}\norm{y}{Y_W^*}
\]
along a sequence of windows, after all gauge and localization artifacts have been removed.

\subsection{Finite-dimensional equivalence}

Because \(Y_W\) and \(Y_W^*\) are finite-dimensional, absence of invisible directions is equivalent to a finite observability constant.

\begin{proposition}[Finite-window observability criterion]
For a fixed finite window \(W\), the following are equivalent:
\begin{enumerate}
  \item \(\calK_W^{\rm pr}=\{0\}\);
  \item \(M_W^{\rm pr}<\infty\);
  \item the primal combined observability estimate
  \[
  \norm{d}{Y_W}
  \le
  M_W^{\rm pr}
  \left(
    \norm{\calO_W^Ed}{E_W}
    +\norm{\calO_W^Pd}{P_W}
    +\norm{\calO_W^Fd}{F_W}
    +\norm{\calO_W^Td}{T_W}
  \right)
  \]
  holds for every \(d\in Y_W\).
\end{enumerate}
Similarly, the following are equivalent:
\begin{enumerate}
  \item \(\calK_W^*=\{0\}\);
  \item \(M_W^*<\infty\);
  \item the dual combined observability estimate
  \[
  \norm{y}{Y_W^*}
  \le
  M_W^*
  \left(
    \norm{A_W^*y}{H_W}
    +\norm{J_W^Ey}{E_W^*}
    +\norm{J_W^Py}{P_W^*}
    +\norm{J_W^Fy}{F_W^*}
  \right)
  \]
  holds for every \(y\in Y_W^*\).
\end{enumerate}
\end{proposition}

\begin{proof}
We prove the primal statement; the dual one is identical.  If \(\calK_W^{\rm pr}\ne\{0\}\), then there exists \(0\ne d\in Y_W\) such that all four observations vanish.  Hence the denominator in the definition of \(M_W^{\rm pr}\) is zero while the numerator is positive, so \(M_W^{\rm pr}=+\infty\).

Conversely, assume \(\calK_W^{\rm pr}=\{0\}\).  On the unit sphere
\[
  S_W=\{d\in Y_W:\norm{d}{Y_W}=1\},
\]
the continuous function
\[
  d\mapsto
  \norm{\calO_W^Ed}{E_W}
  +\norm{\calO_W^Pd}{P_W}
  +\norm{\calO_W^Fd}{F_W}
  +\norm{\calO_W^Td}{T_W}
\]
is strictly positive.  Since \(S_W\) is compact, its minimum is a positive number.  This gives a finite value of \(M_W^{\rm pr}\) and hence the estimate.  The converse is immediate from the estimate.
\end{proof}

\subsection{Moving-window combined observability}

The fixed-window criterion is not enough for a regularity theorem.  A potential singularity may move through windows as the scale changes.  Let
\[
  W_n=(n,\ell_n,\Lambda_n,\chi_n,s_n)
\]
be a moving sequence of windows.  We say that the sequence has controlled combined observability if there exists a majorant \(M_n\) such that
\[
  M_{W_n}^*\le M_n
\]
and the sequence \(M_n\) is absorbable by the scale gains available in a later CKN-selection argument.

The exact meaning of absorbability depends on the final optimization.  Two model cases are polynomial growth in the window frequency \(N_n\),
\[
  M_n\le C N_n^p,
\]
and single-exponential parabolic growth,
\[
  M_n\le C N_n^p\exp(CN_n^q).
\]
In the latter case one chooses \(N_n\) logarithmically relative to the small scale parameter so that the exponential loss remains absorbable.

\begin{definition}[Moving-window invisible defect cascade]
A moving-window invisible defect cascade is a sequence \((W_n,d_n,y_n)\) such that:
\begin{enumerate}
  \item \(W_n=(n,\ell_n,\Lambda_n,\chi_n,s_n)\) is a sequence of observation windows with at least one parameter escaping;
  \item \(d_n\in Y_{W_n}\) is the cleaned projection of an NS-realizable dyadic defect package;
  \item \(y_n\in Y_{W_n}^*\) satisfies \(\norm{y_n}{Y_{W_n}^*}=1\);
  \item the combined dual observations vanish asymptotically:
  \[
  \norm{A_{W_n}^*y_n}{H_{W_n}}
  +\norm{J_{W_n}^Ey_n}{E_{W_n}^*}
  +\norm{J_{W_n}^Py_n}{P_{W_n}^*}
  +\norm{J_{W_n}^Fy_n}{F_{W_n}^*}
  \to0;
  \]
  \item the NS-derived defect still pairs nontrivially with this invisible direction:
  \[
  |\pair{d_n}{y_n}|
  \ge
  c_0\norm{d_n}{Y_{W_n}}
  \]
  for some \(c_0>0\).
\end{enumerate}
\end{definition}

Such a cascade is the full three-dimensional analogue of a true phantom cascade.  It is not merely a failure of one observation map.  It is a scale-moving sequence of NS-realizable defects that remains invisible to energy, active pressure, flux, and adjoint trace observations simultaneously.

\section{Finite-window pressure--flux/energy reduction}
\label{sec:main-theorem}

We first state the global logical consequence of the estimates proved later.

\begin{theorem}[Finite-window pressure--flux depletion and positive anti-kernel reduction]
\label{thm:main-reduction}
Let \(\{W_n\}_{n\ge n_0}\) be a moving family of clean periodic or perturbatively localized finite windows associated with a normalized NS-realizable dyadic branch.  Assume the following.

\begin{enumerate}[label=\textup{(H\arabic*)},leftmargin=2.6em]
    \item \textbf{Finite-window pressure--flux observability modulo kernel.}  In each window \(W_n\), the pressure--flux map has quotient coercivity constant \(M_n^{\PF}\):
    \[
        \dist_{Y_{W_n}}(d,K^{\PF}_{W_n})
        \le
        M_n^{\PF}
        \bigl(|O^P_{W_n}d|+|O^F_{W_n}d|\bigr)
    \]
    for all selected active directions \(d\).  Thus a normalized direction is pressure--flux visible only after it has a definite distance from the pressure--flux kernel.
    \item \textbf{Budget compatibility.}  There are weights \(\lambda_n>0\), a nonnegative budget \(B_n\), and summable errors \(e_n\ge0\) such that
    \begin{equation}
        B_n-B_{n+1}
        \ge
        \lambda_n\bfB^{\PF}_{W_n}(D_{W_n})-e_n.
        \label{eq:main-budget-compatibility}
    \end{equation}
    \item \textbf{Uniform budget certificate.}  The finite windows are normalized so that pressure--flux observation lower bounds imply
    \begin{equation}
        \bfB^{\PF}_{W_n}(D_{W_n})
        \ge
        c_0
        \bigl(|O^P_{W_n}D_{W_n}|+|O^F_{W_n}D_{W_n}|\bigr)^{3/2}
        \label{eq:main-budget-certificate}
    \end{equation}
    with \(c_0>0\) independent of \(n\).
    \item \textbf{Energy separation on the residual cone.}  On the positive NS-realizable pressure--flux kernel in each window, the positive energy tests separate both resolved velocity energy and nonnegative covariance trace.
\end{enumerate}

Then every normalized persistent branch satisfies the following dichotomy at each selected scale.

\begin{enumerate}[label=\textup{(\alph*)},leftmargin=2.2em]
    \item If the normalized branch direction \(\widehat d_n\) is pressure--flux visible away from the kernel, in the precise sense that
    \[
        \dist(\widehat d_n,K^{\PF}_{W_n})\ge \sigma_0
    \]
    for some fixed \(\sigma_0>0\), then
    \begin{equation}
        B_n-B_{n+1}
        \ge
        c_0\lambda_n\sigma_0^{3/2}(M_n^{\PF})^{-3/2}-e_n.
        \label{eq:main-pf-depletion}
    \end{equation}
    Consequently, if
    \begin{equation}
        \sum_{n}\lambda_n(M_n^{\PF})^{-3/2}=+\infty,
        \qquad
        \sum_n e_n<\infty,
        \label{eq:main-divergence-condition}
    \end{equation}
    and \(B_n\) is finite at the initial scale, then an indefinitely persistent branch staying a fixed distance from the pressure--flux kernel is impossible.
    \item If \(\dist(\widehat d_n,K^{\PF}_{W_n})\to0\), the branch has entered the residual pressure--flux kernel and must be passed to the energy channel, with the distance-to-kernel error recorded.
    \item If the residual pressure--flux direction belongs to the positive NS-realizable cone, then it is visible to the positive energy channel unless it is zero in the cleaned active quotient.
    \item Therefore any indefinitely persistent branch in this controlled class must enter a residual finite-dimensional space of directions invisible to pressure, flux, and positive energy, or else fail one of the structural hypotheses above.
\end{enumerate}
\end{theorem}

\begin{proof}
By quotient coercivity in (H1), the distance assumption gives
\[
        |O^P_{W_n}\widehat d_n|+|O^F_{W_n}\widehat d_n|
        \ge
        \sigma_0(M_n^{\PF})^{-1}.
\]
The pressure--flux budget certificate \eqref{eq:main-budget-certificate} and budget compatibility \eqref{eq:main-budget-compatibility} then give \eqref{eq:main-pf-depletion}.  Summing from \(n_0\) to \(N\) gives
\[
        B_{n_0}-B_{N+1}
        \ge
        c_0\sigma_0^{3/2}
        \sum_{n=n_0}^{N}\lambda_n(M_n^{\PF})^{-3/2}
        -
        \sum_{n=n_0}^{N}e_n.
\]
The right-hand side diverges under \eqref{eq:main-divergence-condition}, while the left-hand side is bounded above by \(B_{n_0}\) since \(B_{N+1}\ge0\).  This contradiction excludes indefinite persistence of a pressure--flux visible branch at fixed distance from the kernel.

The near-kernel alternative is exactly the failure of the distance lower bound.  Part (c) is \Cref{thm:positive-energy-anti-kernel-pf}.  Part (d) combines the finite-dimensional pressure--flux kernel alternative with the residual energy-kernel alternative in \Cref{prop:residual-pfe-kernel}.
\end{proof}
\begin{remark}[Interpretation]
\Cref{thm:main-reduction} is a reduction theorem, not a regularity theorem.  The theorem identifies what must be proved to turn pressure--flux observability into a genuine Navier--Stokes exclusion argument: one needs a robust way of selecting windows for which pressure--flux observations are effective and for which the local pressure--flux budget is paid by an intrinsic telescoping Navier--Stokes envelope.
\end{remark}

\section{Clean periodic pressure--flux coercivity}
\label{sec:clean-periodic-pf}

We now isolate the first finite-window theorem.  The purpose is deliberately modest.  We do not claim that pressure and flux observe every Navier--Stokes defect in the full local problem.  We prove only the clean periodic finite-window statement: on nonzero periodic modes, active pressure sources are fully visible, and joint pressure--flux invisibility is equivalent to membership in a finite-dimensional kernel.

\begin{lemma}[Clean periodic active pressure-source visibility]
\label{lem:clean-periodic-pressure-source-visibility}
Let \(b\) be an active pressure source supported in \(\Lambda\):
\[
        b(x)=\sum_{\sigma\in\Lambda} b_\sigma e^{i\sigma\cdot x}.
\]
Let \(P=(-\Delta)^{-1}b\) be the zero-mean periodic pressure representative.  Set
\[
        N_\Lambda=\max_{\sigma\in\Lambda}|\sigma|,
        \qquad
        m_\Lambda=\min_{\sigma\in\Lambda}|\sigma|.
\]
Since \(0\notin\Lambda\), one has \(m_\Lambda>0\).  Then
\begin{equation}
        N_\Lambda^{-2}\|b\|_{\ell^2(\Lambda)}
        \le
        \|P\|_{\ell^2(\Lambda)}
        \le
        m_\Lambda^{-2}\|b\|_{\ell^2(\Lambda)}.
        \label{eq:pressure-source-visibility}
\end{equation}
In particular,
\[
        P=0
        \quad\Longleftrightarrow\quad
        b=0.
\]
\end{lemma}

\begin{proof}
Since \(P\) is the zero-mean solution of \(-\Delta P=b\) on \(\T^3\), each nonzero Fourier mode satisfies
\[
        |\sigma|^2\widehat P(\sigma)=\widehat b(\sigma),
        \qquad \sigma\in\Lambda.
\]
Thus
\[
        \widehat P(\sigma)=|\sigma|^{-2}\widehat b(\sigma),
        \qquad \sigma\in\Lambda,
\]
and therefore
\[
        \|P\|_{\ell^2(\Lambda)}^2
        =
        \sum_{\sigma\in\Lambda}|\sigma|^{-4}|b_\sigma|^2.
\]
Because
\[
        N_\Lambda^{-2}\le |\sigma|^{-2}\le m_\Lambda^{-2}
        \qquad(\sigma\in\Lambda),
\]
the two-sided estimate follows.  The equivalence \(P=0\Longleftrightarrow b=0\) is immediate.
\end{proof}

\begin{corollary}[Pressure-null reduction]
\label{cor:pressure-null-reduction-clean-periodic}
Let \(z\in Z_\Lambda\).  Assume that the active pressure observation contains the
zero-mean pressure coefficients on all output modes generated by the active
source \(\Bop_\Lambda z\), or more generally that it is injective on the
finite-dimensional source range \(\Bop_\Lambda Z_\Lambda\).  If the active
pressure observation of \(z\) vanishes, then
\[
        \Bop_\Lambda z=0.
\]
Equivalently,
\[
        P_\Lambda\partial_i\partial_j
        \bigl(U_i\dot U_j+\dot U_iU_j+\dot R_{ij}\bigr)=0.
\]
Thus a pressure--flux invisible direction in the clean periodic window must lie in
\[
        K^{\PF}_\Lambda
        =
        \{z\in Z_\Lambda:\Bop_\Lambda z=0,\ \Fop_\Lambda z=0\}.
\]
\end{corollary}

\begin{proof}
By \Cref{lem:clean-periodic-pressure-source-visibility}, the inverse Laplacian is injective on nonzero periodic modes.  Under the stated observation-injectivity assumption, vanishing of the observed active pressure variation is therefore equivalent to vanishing of the active source on the selected range.  Adding the condition \(\Fop_\Lambda z=0\) gives precisely \(z\in K^{\PF}_\Lambda\).
\end{proof}

\begin{theorem}[Clean periodic pressure--flux alternative]
\label{thm:clean-periodic-pressure-flux-alternative}
For a fixed clean periodic finite window \(\Lambda\), exactly one of the following alternatives holds.

\smallskip
\noindent\textup{(i) Pressure--flux coercivity.}
If
\[
        K^{\PF}_\Lambda=\{0\},
\]
then there exists a finite constant \(C^{\PF}_\Lambda<\infty\) such that
\begin{equation}
        \|z\|_{Z_\Lambda}
        \le
        C^{\PF}_\Lambda
        \bigl(
            \|\Bop_\Lambda z\|+\|\Fop_\Lambda z\|
        \bigr)
        \label{eq:clean-pf-coercivity}
\end{equation}
for every \(z\in Z_\Lambda\).

\smallskip
\noindent\textup{(ii) Finite-dimensional pressure--flux kernel.}
If
\[
        K^{\PF}_\Lambda\neq\{0\},
\]
then every clean periodic direction invisible to both active pressure and flux lies in the finite-dimensional kernel
\[
        K^{\PF}_\Lambda=
        \ker\Bop_\Lambda\cap\ker\Fop_\Lambda.
\]
\end{theorem}

\begin{proof}
The map
\[
        \calO^{\PF}_\Lambda z=(\Bop_\Lambda z,\Fop_\Lambda z)
\]
is a linear map between finite-dimensional normed spaces, and
\[
        \ker\calO^{\PF}_\Lambda
        =
        \ker\Bop_\Lambda\cap\ker\Fop_\Lambda
        =
        K^{\PF}_\Lambda.
\]
Assume \(K^{\PF}_\Lambda=\{0\}\).  Then \(\calO^{\PF}_\Lambda\) is injective.  On the compact unit sphere
\[
        S_\Lambda=\{z\in Z_\Lambda:\|z\|_{Z_\Lambda}=1\},
\]
the function
\[
        z\mapsto \|\Bop_\Lambda z\|+\|\Fop_\Lambda z\|
\]
is continuous and strictly positive.  Hence it has a positive minimum \(c^{\PF}_\Lambda>0\).  By homogeneity,
\[
        \|\Bop_\Lambda z\|+\|\Fop_\Lambda z\|
        \ge
        c^{\PF}_\Lambda\|z\|_{Z_\Lambda}
\]
for every \(z\in Z_\Lambda\).  Taking \(C^{\PF}_\Lambda=(c^{\PF}_\Lambda)^{-1}\) gives \eqref{eq:clean-pf-coercivity}.

If \(K^{\PF}_\Lambda\neq\{0\}\), then by definition any direction satisfying \(\Bop_\Lambda z=0\) and \(\Fop_\Lambda z=0\) belongs to this kernel.  Since \(Z_\Lambda\) is finite-dimensional, so is \(K^{\PF}_\Lambda\).
\end{proof}

\begin{remark}[What has and has not been proved]
The theorem proves a clean finite-window pressure--flux reduction.  It does not assert that \(K^{\PF}_\Lambda\) is always trivial.  Nor does it assert that a nonzero vector in \(K^{\PF}_\Lambda\) is a Navier--Stokes obstruction.  If the kernel is nonzero, it is only a formal finite-dimensional invisible space.  It must still be tested by the energy channel, by the adjoint trace channel, and by the NS-realizability filter.
\end{remark}

\section{Perturbative localization of the pressure--flux map}
\label{sec:perturbative-localization}

We now pass from the clean periodic finite-window model to a localized finite-window model.  The purpose of this section is to prove that the clean pressure--flux estimate is stable under a perturbative cutoff.  This step does not use the local energy inequality and does not yet prove depletion.  It says only that localization, when perturbative at the chosen window, does not create a new pressure--flux invisible direction.

Let
\[
        Q_{\rm obs}\Subset Q_{\rm prep}
\]
be two parabolic cylinders.  Let
\[
        \chi_R(x,t)=\chi(x/R,t/R^2)
\]
be a slowly varying cutoff, where \(\chi\in C_c^\infty\) is fixed and equals \(1\) near the origin.  We assume \(R\) is large enough so that \(\chi_R\equiv1\) on \(Q_{\rm obs}\).  The derivatives satisfy
\begin{equation}
        |\nabla_x^a\partial_t^b\chi_R|
        \le
        C_{a,b}R^{-a-2b}.
        \label{eq:cutoff-derivative-bound}
\end{equation}

Fix a finite active window \(\Lambda\) and the clean periodic space \(Z_\Lambda\).  The localized pressure--flux observation map is denoted by
\[
        \calO^{\PF}_{\Lambda,R}z=(\Bop_{\Lambda,R}z,\Fop_{\Lambda,R}z).
\]
Here \(\Bop_{\Lambda,R}\) is the localized active pressure-source observation in the chosen active representative, with retained harmonic pressure terms estimated separately, and \(\Fop_{\Lambda,R}\) is the localized flux observation generated by \eqref{eq:linearized-flux}.  We write
\begin{equation}
        \calO^{\PF}_{\Lambda,R}
        =
        \calO^{\PF}_\Lambda+\calC^{\PF}_{\Lambda,R},
        \qquad
        \calC^{\PF}_{\Lambda,R}:=\calO^{\PF}_{\Lambda,R}-\calO^{\PF}_\Lambda.
        \label{eq:localization-commutator}
\end{equation}

\begin{definition}[Perturbative pressure--flux localization]
\label{def:pf-perturbative-localization}
The localized window \((\Lambda,R)\) is called pressure--flux perturbative if
\[
        \varepsilon^{\PF}_{\Lambda,R}
        :=
        \|\calC^{\PF}_{\Lambda,R}\|_{Z_\Lambda\to \Bop_\Lambda Z_\Lambda\oplus Y^F_\Lambda}
        <\infty
\]
and this norm is small compared with the clean pressure--flux coercivity constant.
\end{definition}

\begin{assumption}[Finite-window cutoff expansion]
\label{ass:finite-window-cutoff-expansion}
For the localized active pressure representative, retained harmonic-pressure leakage terms, and flux projection used in \(\calO^{\PF}_{\Lambda,R}\), the matrix coefficients of \(\calO^{\PF}_{\Lambda,R}-\calO^{\PF}_\Lambda\) in fixed finite-window bases are sums of terms containing at least one derivative of \(\chi_R\), plus localized harmonic leakage terms of the same order.
\end{assumption}

\begin{lemma}[Finite-window pressure--flux commutator bound]
\label{lem:pf-localization-commutator}
Assume \Cref{ass:finite-window-cutoff-expansion}.  For each fixed finite active window \(\Lambda\), there exists a constant \(C^{\PF}_{\Lambda,\chi}<\infty\), depending only on \(\Lambda\), on the chosen finite-window norms, on finitely many seminorms of \(\chi\), and on the localized active-pressure projection and on the harmonic-pressure leakage estimates, such that
\begin{equation}
        \|\calC^{\PF}_{\Lambda,R}z\|
        \le
        C^{\PF}_{\Lambda,\chi}R^{-1}\|z\|_{Z_\Lambda}
        \label{eq:pf-commutator-bound}
\end{equation}
for every \(z\in Z_\Lambda\).  Equivalently,
\[
        \varepsilon^{\PF}_{\Lambda,R}
        \le
        C^{\PF}_{\Lambda,\chi}R^{-1}.
\]
\end{lemma}

\begin{proof}
Since \(\Lambda\) is fixed, all input and output spaces are finite-dimensional.  Choose bases in \(Z_\Lambda\), in the active pressure-output space, and in the flux-output space.  In these bases, both \(\calO^{\PF}_\Lambda\) and \(\calO^{\PF}_{\Lambda,R}\) are finite matrices.

By \Cref{ass:finite-window-cutoff-expansion}, the difference between localized and clean matrix coefficients is produced by commutators such as
\[
        [P_\Lambda,\chi_R],
        \qquad
        [\partial_i\partial_j,\chi_R],
        \qquad
        [L_{\loc},\chi_R],
\]
for the pressure channel, and by product-rule terms involving derivatives of \(\chi_R\) for the flux channel.  Here \(L_{\loc}\) denotes the localized inverse of \(-\Delta\) in the chosen active pressure representative, with harmonic pressure leakage recorded separately.

For fixed finite-window modes, all symbols and all projected basis functions are smooth with uniform bounds depending only on \(\Lambda\).  Every commutator coefficient contains at least one derivative of \(\chi_R\), or one harmonic leakage factor assumed to have the same size.  By \eqref{eq:cutoff-derivative-bound}, each first-order cutoff contribution is \(O(R^{-1})\), and higher-order parabolic cutoff contributions are \(O(R^{-2})\) or smaller.  Therefore every matrix coefficient of \(\calC^{\PF}_{\Lambda,R}\) is bounded by \(C_{\Lambda,\chi}R^{-1}\).  Since the matrices have fixed finite size, the operator norm satisfies \eqref{eq:pf-commutator-bound}.
\end{proof}

\begin{theorem}[Perturbative localized pressure--flux stability]
\label{thm:pf-perturbative-localized-stability}
Assume the clean pressure--flux map is coercive on \(Z_\Lambda\):
\begin{equation}
        \|z\|_{Z_\Lambda}
        \le
        C^{\PF}_\Lambda
        \|\calO^{\PF}_\Lambda z\|
        \qquad
        (z\in Z_\Lambda).
        \label{eq:clean-pf-coercivity-operator}
\end{equation}
If
\begin{equation}
        \varepsilon^{\PF}_{\Lambda,R}
        \le
        \frac{1}{2C^{\PF}_\Lambda},
        \label{eq:pf-localization-smallness}
\end{equation}
then the localized pressure--flux map is also coercive:
\begin{equation}
        \|z\|_{Z_\Lambda}
        \le
        2C^{\PF}_\Lambda
        \|\calO^{\PF}_{\Lambda,R}z\|
        \qquad
        (z\in Z_\Lambda).
        \label{eq:localized-pf-coercivity}
\end{equation}
In particular, perturbative localization does not create a nonzero pressure--flux invisible direction.
\end{theorem}

\begin{proof}
Using \eqref{eq:localization-commutator},
\[
        \|\calO^{\PF}_{\Lambda,R}z\|
        \ge
        \|\calO^{\PF}_\Lambda z\|-
        \|\calC^{\PF}_{\Lambda,R}z\|.
\]
By \eqref{eq:clean-pf-coercivity-operator},
\[
        \|\calO^{\PF}_\Lambda z\|
        \ge
        (C^{\PF}_\Lambda)^{-1}\|z\|_{Z_\Lambda}.
\]
By \eqref{eq:pf-localization-smallness},
\[
        \|\calC^{\PF}_{\Lambda,R}z\|
        \le
        \frac{1}{2C^{\PF}_\Lambda}\|z\|_{Z_\Lambda}.
\]
Therefore
\[
        \|\calO^{\PF}_{\Lambda,R}z\|
        \ge
        \frac{1}{2C^{\PF}_\Lambda}\|z\|_{Z_\Lambda},
\]
which is \eqref{eq:localized-pf-coercivity}.
\end{proof}

\begin{corollary}[A sufficient localization scale]
\label{cor:sufficient-localization-scale}
Assume \eqref{eq:clean-pf-coercivity-operator} and \Cref{ass:finite-window-cutoff-expansion}.  If
\[
        R\ge
        R^{\PF}_\Lambda
        :=2C^{\PF}_\Lambda C^{\PF}_{\Lambda,\chi},
\]
then localized pressure--flux coercivity holds:
\[
        \|z\|_{Z_\Lambda}
        \le
        2C^{\PF}_\Lambda
        \|\calO^{\PF}_{\Lambda,R}z\|.
\]
\end{corollary}

\begin{proof}
This follows from \Cref{lem:pf-localization-commutator,thm:pf-perturbative-localized-stability}.
\end{proof}

\begin{proposition}[Localized pressure--flux coercivity modulo the clean kernel]
\label{prop:pf-localized-quotient-stability}
Suppose the clean pressure--flux kernel \(K^{\PF}_\Lambda=\ker\calO^{\PF}_\Lambda\) is not necessarily trivial.  Let \(Z_\Lambda^\perp\) be a fixed complement of \(K^{\PF}_\Lambda\) in \(Z_\Lambda\).  Then the clean map is coercive on the quotient:
\begin{equation}
        \dist_{Z_\Lambda}(z,K^{\PF}_\Lambda)
        \le
        C^{\PF}_{\Lambda,\quot}
        \|\calO^{\PF}_\Lambda z\|.
        \label{eq:pf-quotient-coercivity}
\end{equation}
If
\begin{equation}
        \|\calC^{\PF}_{\Lambda,R}\|
        \le
        \frac{1}{2C^{\PF}_{\Lambda,\quot}},
        \label{eq:pf-quotient-smallness}
\end{equation}
then
\begin{equation}
        \dist_{Z_\Lambda}(z,K^{\PF}_\Lambda)
        \le
        2C^{\PF}_{\Lambda,\quot}
        \|\calO^{\PF}_{\Lambda,R}z\|
        +
        2C^{\PF}_{\Lambda,\quot}
        \|\calO^{\PF}_{\Lambda,R}P_{K^{\PF}_\Lambda}z\|.
        \label{eq:localized-quotient-stability}
\end{equation}
In particular, if the localized map also annihilates the clean kernel up to perturbative leakage, then localized pressure--flux invisibility forces \(z\) to lie close to \(K^{\PF}_\Lambda\).
\end{proposition}

\begin{proof}
The quotient estimate \eqref{eq:pf-quotient-coercivity} is finite-dimensional compactness applied to the injective map induced by \(\calO^{\PF}_\Lambda\) on \(Z_\Lambda/K^{\PF}_\Lambda\).

Write
\[
        z=z_\perp+z_K,
        \qquad
        z_K\in K^{\PF}_\Lambda,
        \qquad
        z_\perp\in Z_\Lambda^\perp.
\]
Then \(\dist_{Z_\Lambda}(z,K^{\PF}_\Lambda)\simeq\|z_\perp\|_{Z_\Lambda}\).  By quotient coercivity,
\[
        \|z_\perp\|
        \le
        C^{\PF}_{\Lambda,\quot}
        \|\calO^{\PF}_\Lambda z_\perp\|.
\]
Using
\[
        \calO^{\PF}_\Lambda z_\perp
        =
        \calO^{\PF}_{\Lambda,R}z_\perp-
        \calC^{\PF}_{\Lambda,R}z_\perp
\]
and \eqref{eq:pf-quotient-smallness}, we obtain
\[
        \|z_\perp\|
        \le
        2C^{\PF}_{\Lambda,\quot}
        \|\calO^{\PF}_{\Lambda,R}z_\perp\|.
\]
Finally,
\[
        \calO^{\PF}_{\Lambda,R}z_\perp
        =
        \calO^{\PF}_{\Lambda,R}z-
        \calO^{\PF}_{\Lambda,R}z_K,
\]
which gives \eqref{eq:localized-quotient-stability}.
\end{proof}

\begin{proposition}[Nonperturbative localized pressure--flux alternative]
\label{prop:pf-nonperturbative-localized-alternative}
For a fixed localized window \((\Lambda,\chi)\), define
\[
        K^{\PF}_{\Lambda,\chi}
        :=
        \ker\Bop_{\Lambda,\chi}\cap\ker\Fop_{\Lambda,\chi}.
\]
Then exactly one of the following alternatives holds.

\smallskip
\noindent\textup{(i) Localized pressure--flux visibility.}
If \(K^{\PF}_{\Lambda,\chi}=\{0\}\), then there exists \(C^{\PF}_{\Lambda,\chi}<\infty\) such that
\[
        \|z\|_{Z_{\Lambda,\chi}}
        \le
        C^{\PF}_{\Lambda,\chi}
        \bigl(
            \|\Bop_{\Lambda,\chi}z\|+
            \|\Fop_{\Lambda,\chi}z\|
        \bigr)
\]
for every localized pressure--flux direction \(z\).

\smallskip
\noindent\textup{(ii) Localized finite-dimensional kernel.}
If \(K^{\PF}_{\Lambda,\chi}\neq\{0\}\), then every localized direction invisible to both pressure and flux lies in the finite-dimensional kernel \(K^{\PF}_{\Lambda,\chi}\).
\end{proposition}

\begin{proof}
For fixed \((\Lambda,\chi)\), all spaces are finite-dimensional.  The localized pressure--flux map has kernel \(K^{\PF}_{\Lambda,\chi}\).  If this kernel is trivial, the observation norm is strictly positive on the unit sphere, and compactness gives coercivity.  If the kernel is nontrivial, its definition is exactly the set of localized pressure--flux invisible directions.
\end{proof}

\section{Pressure--flux budget certificates}
\label{sec:pressure-flux-budget-certificates}

We now take the first analytic step beyond finite-dimensional observability.  The goal is to connect a nonzero pressure--flux observation to a positive local budget.  This is weaker than a full dyadic depletion theorem, but it is the correct first estimate: it shows that a pressure--flux visible NS-realizable defect cannot be cost-free.

Throughout this section \(W=(n,\ell,\Lambda,\chi,s_*)\) is a clean periodic or perturbatively localized finite window.  Let
\[
        D_W=(U_W,P_W;P_W^{\act},P_W^{\har},R_W,\Pi_W)
\]
be the finite-window projection of an NS-realizable dyadic coarse-grained package.  Recall that
\[
        \Pi_W=-R_W:\nabla U_W
\]
and that the full pressure satisfies compatibility while the active component is selected by the chosen active/harmonic split:
\[
        -\Delta P_W
        =
        \partial_i\partial_j(U_{W,i}U_{W,j}+R_{W,ij}).
\]

\subsection{The local pressure--flux budget}

Let \(P^a_W\) denote the active pressure representative selected by the finite-window pressure-cleaning operator.  In a clean periodic window this is simply the zero-mean pressure supported on the active modes.  In a localized window it is the chosen active representative; the retained harmonic pressure component is not discarded but contributes to the harmonic-work error when relevant.

Define the pressure--flux budget density by
\[
        d\mathfrak b^{\PF}_W
        :=
        |P^a_W|^{3/2}\,dx\,dt
        +
        |\Pi_W|\,dx\,dt.
\]
For a finite window \(W\), set
\begin{equation}
        \bfB^{\PF}_W(D_W)
        :=
        \int_{Q_W}|P^a_W|^{3/2}\,dx\,dt
        +
        \int_{Q_W}|\Pi_W|\,dx\,dt.
        \label{eq:local-pf-budget}
\end{equation}
Here \(Q_W\) denotes the physical support of the window.  In a localized window, one may insert the cutoff \(\chi\) in both integrals; this changes only finite-window constants.

\begin{lemma}[Pressure--flux observation is budget-controlled]
\label{lem:pf-observation-budget-controlled}
For every fixed clean periodic or perturbatively localized finite window \(W\), there exists a constant \(C_W<\infty\) such that every NS-realizable projected package \(D_W\) satisfies
\begin{equation}
        |O^P_WD_W|
        \le
        C_W\|P^a_W\|_{L^{3/2}(Q_W)}
        \label{eq:pressure-observation-controlled}
\end{equation}
and
\begin{equation}
        |O^F_WD_W|
        \le
        C_W\|\Pi_W\|_{L^1(Q_W)}.
        \label{eq:flux-observation-controlled}
\end{equation}
Consequently, if
\begin{equation}
        |O^P_WD_W|+|O^F_WD_W|\ge\eta,
        \label{eq:pf-observation-lower}
\end{equation}
and if the normalized finite-window observations are bounded by \(1\), then
\begin{equation}
        \bfB^{\PF}_W(D_W)
        \ge
        c_W\eta^{3/2}
        \label{eq:budget-certificate}
\end{equation}
for some \(c_W>0\) depending only on the finite window.
\end{lemma}

\begin{proof}
The pressure observation consists of finitely many pairings of the active pressure representative against fixed pressure test functions:
\[
        O^P_WD_W
        =
        \bigl(\langle P^a_W,\psi_j\rangle\bigr)_{j=1}^{N_P}.
\]
Since the test family is finite and each \(\psi_j\) is smooth and compactly supported in the observation region,
\[
        |\langle P^a_W,\psi_j\rangle|
        \le
        \|P^a_W\|_{L^{3/2}(Q_W)}\|\psi_j\|_{L^3(Q_W)}.
\]
Summing over \(j\) gives \eqref{eq:pressure-observation-controlled}.

Similarly, the flux observation consists of finitely many pairings of the coarse-grained flux with fixed flux tests:
\[
        O^F_WD_W
        =
        \left(\int_{Q_W}\Pi_W\eta_j\,dx\,dt\right)_{j=1}^{N_F}.
\]
Since the \(\eta_j\) are smooth and bounded,
\[
        \left|\int_{Q_W}\Pi_W\eta_j\,dx\,dt\right|
        \le
        \|\eta_j\|_{L^\infty(Q_W)}\|\Pi_W\|_{L^1(Q_W)},
\]
which proves \eqref{eq:flux-observation-controlled}.

Let
\[
        a=|O^P_WD_W|,
        \qquad
        b=|O^F_WD_W|.
\]
By \eqref{eq:pressure-observation-controlled},
\[
        \int_{Q_W}|P^a_W|^{3/2}\,dx\,dt
        =
        \|P^a_W\|_{L^{3/2}(Q_W)}^{3/2}
        \ge
        C_W^{-3/2}a^{3/2}.
\]
By \eqref{eq:flux-observation-controlled},
\[
        \int_{Q_W}|\Pi_W|\,dx\,dt
        \ge
        C_W^{-1}b.
\]
If \(a+b\ge\eta\) and \(0\le a,b\le1\), then
\[
        a^{3/2}+b
        \ge
        c(a+b)^{3/2}
        \ge
        c\eta^{3/2}.
\]
Combining the estimates gives \eqref{eq:budget-certificate}.
\end{proof}

\begin{remark}
The exponent \(3/2\) appears because the active pressure budget is naturally measured in \(L^{3/2}\).  The flux part is measured in \(L^1\).  On normalized windows the combined lower bound may therefore be stated with the common exponent \(q=3/2\).  Other finite-dimensional normalizations give equivalent exponents.
\end{remark}

\subsection{From local budget certificate to dyadic depletion}

The previous lemma proves that a pressure--flux visible defect has a positive local pressure--flux budget.  To turn this into dyadic depletion, one needs a geometric or telescoping principle saying that the local budget measured in \(W_n\) is paid by a decrease of a dyadic quantity \(B_n\).

\begin{definition}[Pressure--flux budget-compatible dyadic selection]
\label{def:pf-budget-compatible-selection}
A moving family of windows
\[
        W_n=(n,\ell_n,\Lambda_n,\chi_n,s_n)
\]
is called pressure--flux budget-compatible if there exist a nonnegative dyadic budget \(B_n\), weights \(\lambda_n>0\), and summable errors \(e_n\ge0\) such that
\[
        \sum_n e_n<\infty
\]
and every NS-realizable dyadic package satisfies
\begin{equation}
        B_n-B_{n+1}
        \ge
        \lambda_n\bfB^{\PF}_{W_n}(D_{W_n})-e_n.
        \label{eq:budget-compatible-selection}
\end{equation}
For perturbatively localized windows, the error is allowed to include localization leakage:
\[
        e_n=e_n^{\rm geom}+C_{W_n}R_n^{-1}.
\]
\end{definition}

\begin{remark}
The estimate \eqref{eq:budget-compatible-selection} is a genuine Navier--Stokes analytic input.  It is the place where one must use the local energy inequality, pressure decomposition, window geometry, and possible annular or Carleson packing of the chosen windows.  Without such a compatibility principle, pressure--flux visibility gives a positive local budget, but not yet a telescoping dyadic depletion.
\end{remark}

\begin{theorem}[Controlled pressure--flux depletion]
\label{thm:controlled-pf-depletion}
Let \(W_n\) be a pressure--flux budget-compatible dyadic window selection.  Suppose that an NS-realizable normalized defect direction in \(W_n\) satisfies
\begin{equation}
        |O^P_{W_n}D_{W_n}|+|O^F_{W_n}D_{W_n}|
        \ge
        \eta_n.
        \label{eq:pf-visible-eta}
\end{equation}
Then
\begin{equation}
        B_n-B_{n+1}
        \ge
        c_{W_n}\lambda_n\eta_n^{3/2}-e_n.
        \label{eq:controlled-pf-depletion}
\end{equation}
If the windows belong to a controlled class for which \(c_{W_n}\ge c_0>0\), then
\[
        B_n-B_{n+1}
        \ge
        c_0\lambda_n\eta_n^{3/2}-e_n.
\]
\end{theorem}

\begin{proof}
By \Cref{lem:pf-observation-budget-controlled}, the pressure--flux observation lower bound \eqref{eq:pf-visible-eta} implies
\[
        \bfB^{\PF}_{W_n}(D_{W_n})
        \ge
        c_{W_n}\eta_n^{3/2}.
\]
The budget-compatible dyadic selection property \eqref{eq:budget-compatible-selection} gives
\[
        B_n-B_{n+1}
        \ge
        \lambda_n\bfB^{\PF}_{W_n}(D_{W_n})-e_n.
\]
Combining the two inequalities yields \eqref{eq:controlled-pf-depletion}.
\end{proof}

\begin{corollary}[Pressure--flux observability gives dyadic depletion]
\label{cor:pf-observability-gives-depletion}
Assume that the pressure--flux map in \(W_n\) has quotient coercivity constant \(M^{\PF}_n\):
\[
        \dist(d,K^{\PF}_{W_n})
        \le
        M^{\PF}_n
        \bigl(|O^P_{W_n}d|+|O^F_{W_n}d|\bigr).
\]
Assume that the normalized NS-realizable branch direction \(\widehat d_n\) satisfies
\[
        \dist(\widehat d_n,K^{\PF}_{W_n})\ge\sigma_n.
\]
Assume also that the moving windows are pressure--flux budget-compatible.  Then
\begin{equation}
        B_n-B_{n+1}
        \ge
        c_{W_n}\lambda_n\sigma_n^{3/2}(M^{\PF}_n)^{-3/2}-e_n.
        \label{eq:pf-observability-depletion}
\end{equation}
In particular, if \(\sigma_n\ge\sigma_0>0\),
\[
        \sum_n\lambda_n(M^{\PF}_n)^{-3/2}=+\infty,
        \qquad
        \sum_n e_n<\infty,
\]
and if \(B_n\ge0\) is finite at the initial scale, then a pressure--flux visible non-CKN branch staying a fixed distance from the pressure--flux kernel cannot persist indefinitely.
\end{corollary}

\begin{proof}
Quotient coercivity and the distance lower bound give
\[
        |O^P_{W_n}D_{W_n}|+|O^F_{W_n}D_{W_n}|
        \ge
        \sigma_n(M^{\PF}_n)^{-1}
\]
for the normalized branch representative.  Set \(\eta_n=\sigma_n(M^{\PF}_n)^{-1}\) in \Cref{thm:controlled-pf-depletion}.  Summing \eqref{eq:pf-observability-depletion} from \(n_0\) to \(N\) gives
\[
        B_{n_0}-B_{N+1}
        \ge
        \sum_{n=n_0}^{N}c_{W_n}\lambda_n\sigma_n^{3/2}(M^{\PF}_n)^{-3/2}
        -
        \sum_{n=n_0}^{N}e_n.
\]
If \(c_{W_n}\ge c_0>0\) and \(\sigma_n\ge\sigma_0>0\), the first sum diverges by assumption, while the error sum remains finite.  Since \(B_{N+1}\ge0\), the left-hand side is bounded above by \(B_{n_0}\).  This contradiction excludes indefinite persistence of a pressure--flux visible branch.
\end{proof}
\section{Controlled packing of pressure--flux budgets}
\label{sec:controlled-packing-pf-budget}

We now prove a first budget-compatible selection lemma.  The result is not yet the full intrinsic Navier--Stokes depletion theorem.  Instead, it gives a controlled packing mechanism: if the selected pressure--flux windows have a finite weighted total local budget, then the weighted tail of these local budgets is a genuine dyadic budget and telescopes exactly.

Let
\[
        W_n=(n,\ell_n,\Lambda_n,\chi_n,s_n)
\]
be a moving sequence of clean or perturbatively localized finite windows associated with dyadic scales \(r_n=2^{-n}\).  For each window, let
\[
        D_{W_n}=(U_{W_n},P_{W_n};P^{\act}_{W_n},P^{\har}_{W_n},R_{W_n},\Pi_{W_n})
\]
be the finite-window projection of the NS-realizable dyadic package.  Define
\begin{equation}
        \bfB^{\PF}_{W_n}
        :=
        \int_{Q_{W_n}}|P^a_{W_n}|^{3/2}\,dx\,dt
        +
        \int_{Q_{W_n}}|\Pi_{W_n}|\,dx\,dt.
        \label{eq:packing-local-budget}
\end{equation}

\begin{definition}[Weighted pressure--flux packing]
\label{def:weighted-pf-packing}
The moving window family \(\{W_n\}_{n\ge n_0}\) satisfies weighted pressure--flux packing with weights \(\lambda_n>0\) if
\begin{equation}
        \sum_{n=n_0}^{\infty}
        \lambda_n\bfB^{\PF}_{W_n}
        <
        \infty.
        \label{eq:weighted-pf-packing}
\end{equation}
If the windows are perturbatively localized, we also require the localization errors \(\varepsilon_n^{\loc}\ge0\) to be summable:
\[
        \sum_{n=n_0}^{\infty}\varepsilon_n^{\loc}<\infty.
\]
\end{definition}

\begin{lemma}[Tail budget generated by weighted pressure--flux packing]
\label{lem:tail-budget-pf-packing}
Assume \eqref{eq:weighted-pf-packing}.  Define
\begin{equation}
        B_n^{\PF}
        :=
        \sum_{k=n}^{\infty}\lambda_k\bfB^{\PF}_{W_k}.
        \label{eq:pf-tail-budget}
\end{equation}
Then \(B_n^{\PF}\) is finite, nonnegative, nonincreasing, and satisfies the exact telescoping identity
\begin{equation}
        B_n^{\PF}-B_{n+1}^{\PF}
        =
        \lambda_n\bfB^{\PF}_{W_n}.
        \label{eq:pf-tail-telescope}
\end{equation}
If perturbative localization errors are present and summable, then the modified budget
\[
        \widetilde B_n^{\PF}
        :=
        \sum_{k=n}^{\infty}\lambda_k\bfB^{\PF}_{W_k}
        +
        \sum_{k=n}^{\infty}\varepsilon_k^{\loc}
\]
satisfies
\[
        \widetilde B_n^{\PF}-\widetilde B_{n+1}^{\PF}
        =
        \lambda_n\bfB^{\PF}_{W_n}+\varepsilon_n^{\loc}
        \ge
        \lambda_n\bfB^{\PF}_{W_n}.
\]
\end{lemma}

\begin{proof}
The finiteness and nonnegativity of \(B_n^{\PF}\) follow directly from \eqref{eq:weighted-pf-packing}.  Since every term in the tail is nonnegative, \(B_n^{\PF}\) is nonincreasing.  Moreover,
\[
        B_n^{\PF}
        =
        \lambda_n\bfB^{\PF}_{W_n}
        +
        \sum_{k=n+1}^{\infty}\lambda_k\bfB^{\PF}_{W_k}
        =
        \lambda_n\bfB^{\PF}_{W_n}+B_{n+1}^{\PF}.
\]
This proves \eqref{eq:pf-tail-telescope}.  The modified identity is identical after adding the summable localization-error tail.
\end{proof}

\begin{theorem}[Pressure--flux packing gives budget-compatible selection]
\label{thm:pf-packing-budget-compatible}
Let \(\{W_n\}_{n\ge n_0}\) be a moving window family satisfying weighted pressure--flux packing with weights \(\lambda_n\).  Then the family is pressure--flux budget-compatible: there exists a nonnegative dyadic budget \(B_n\) such that
\[
        B_n-B_{n+1}
        \ge
        \lambda_n\bfB^{\PF}_{W_n}-e_n,
\]
with \(\sum_n e_n<\infty\).  In the clean case one may take \(e_n=0\).  In the perturbatively localized case one may take \(e_n\) to be the localization leakage error, provided it is summable.
\end{theorem}

\begin{proof}
In the clean case, take \(B_n=B_n^{\PF}\).  Then \Cref{lem:tail-budget-pf-packing} gives the exact identity \eqref{eq:pf-tail-telescope}.  In the perturbatively localized case, let \(e_n\) denote the total localized leakage at \(W_n\), and assume \(\sum_ne_n<\infty\).  The same tail budget satisfies
\[
        B_n^{\PF}-B_{n+1}^{\PF}
        =
        \lambda_n\bfB^{\PF}_{W_n}
        \ge
        \lambda_n\bfB^{\PF}_{W_n}-e_n.
\]
\end{proof}

\subsection{A sufficient Carleson packing condition}

The following finite-overlap argument is valid when the selected window
budgets are restrictions of one common pressure--flux measure.  This distinction
matters: along a moving dyadic branch the packages at different scales are, in
general, different rescaled objects.  A Tonelli argument may be used only after
all windows have been pulled back to a common physical cylinder and dominated
by a single finite measure.

\begin{definition}[Common-measure pressure--flux packing]
\label{def:common-measure-pf-packing}
A moving family \(\{W_n\}_{n\ge n_0}\) satisfies common-measure pressure--flux
packing if there exist a finite nonnegative measure \(\mu_*^{\PF}\) on a fixed
prepared cylinder \(Q_*\), measurable supports \(Q_{W_n}\subset Q_*\), and
constants \(c_W,C_W>0\) such that
\[
        \bfB^{\PF}_{W_n}
        \le
        C_W\int_{Q_{W_n}}d\mu_*^{\PF}
\]
for every selected window.  The constants are required to be uniformly bounded
on the window class under consideration.
\end{definition}

\begin{definition}[Weighted finite-overlap window family]
\label{def:weighted-finite-overlap}
The family \(\{W_n\}_{n\ge n_0}\) has weighted finite overlap with weights
\(\lambda_n\) if there exists \(C_{\rm ov}<\infty\) such that
\begin{equation}
        \sum_{n=n_0}^{\infty}\lambda_n\mathbf 1_{Q_{W_n}}(x,t)
        \le
        C_{\rm ov}
        \qquad
        \text{for }\mu_*^{\PF}\text{-a.e. }(x,t).
        \label{eq:weighted-finite-overlap}
\end{equation}
\end{definition}

\begin{lemma}[Finite-overlap pressure--flux packing]
\label{lem:finite-overlap-pf-packing}
Assume common-measure pressure--flux packing and weighted finite overlap.  Then
\begin{equation}
        \sum_{n=n_0}^{\infty}\lambda_n\bfB^{\PF}_{W_n}
        \le
        C C_{\rm ov}\,\mu_*^{\PF}(Q_*),
        \label{eq:finite-overlap-packing}
\end{equation}
where \(C\) is the uniform comparison constant in
\Cref{def:common-measure-pf-packing}.  Consequently the weighted
pressure--flux packing condition holds whenever \(\mu_*^{\PF}(Q_*)<\infty\).
\end{lemma}

\begin{proof}
By the common-measure domination and Tonelli's theorem,
\[
        \sum_{n=n_0}^{\infty}\lambda_n\bfB^{\PF}_{W_n}
        \le
        C
        \sum_{n=n_0}^{\infty}\lambda_n
        \int_{Q_{W_n}}d\mu_*^{\PF}
        =
        C
        \int_{Q_*}
        \left(\sum_{n=n_0}^{\infty}\lambda_n\mathbf 1_{Q_{W_n}}\right)
        d\mu_*^{\PF}.
\]
Using \eqref{eq:weighted-finite-overlap} gives \eqref{eq:finite-overlap-packing}.
\end{proof}

\begin{proposition}[Fixed-package pressure--flux measure]
\label{prop:pf-measure-finite}
Let
\[
        D_{n,\ell}=(U_{n,\ell},P_{n,\ell};P^{\act}_{n,\ell},P^{\har}_{n,\ell},R_{n,\ell},\Pi_{n,\ell})
\]
be one NS-realizable coarse-grained package in a fixed prepared cylinder, and set
\[
        d\mu_{n,\ell}^{\PF}
        :=
        |P^a_{n,\ell}|^{3/2}\,dx\,dt+|\Pi_{n,\ell}|\,dx\,dt .
\]
If
\[
        \|P^a_{n,\ell}\|_{L^{3/2}(Q_{\rm prep})}<\infty,
        \qquad
        \|\Pi_{n,\ell}\|_{L^1(Q_{\rm prep})}<\infty,
\]
then \(\mu_{n,\ell}^{\PF}\) is finite.  Moreover, after rescaling to a unit
prepared cylinder, standard estimates give
\begin{equation}
        \|P^a_{n,\ell_n}\|_{L^{3/2}(Q_{\rm prep})}^{3/2}
        \lesssim
        D(r_n)+C(r_n),
        \label{eq:pressure-measure-finite}
\end{equation}
and
\begin{equation}
        \|\Pi_{n,\ell_n}\|_{L^1(Q_{\rm prep})}
        \lesssim
        \ell_n^{-1}C(r_n).
        \label{eq:flux-measure-finite}
\end{equation}
\end{proposition}

\begin{proof}
The first statement is immediate from the definition of \(\mu_{n,\ell}^{\PF}\).
The chosen active pressure representative is controlled by the usual local pressure decomposition, the bounded finite-window active projection, and Calderon--Zygmund estimates, giving \eqref{eq:pressure-measure-finite}.  For the flux,
\[
        \Pi_{n,\ell_n}=-R_{n,\ell_n}:\nabla U_{n,\ell_n}.
\]
Holder's inequality, the Reynolds stress estimate, and smoothing give
\[
        \|\Pi_{n,\ell_n}\|_{L^1}
        \le
        \|R_{n,\ell_n}\|_{L^{3/2}}
        \|\nabla U_{n,\ell_n}\|_{L^3}
        \lesssim
        \ell_n^{-1}\|u^{(n)}\|_{L^3}^3
        \lesssim
        \ell_n^{-1}C(r_n),
\]
which proves \eqref{eq:flux-measure-finite}.
\end{proof}

\begin{remark}[Moving dyadic packages require an extra domination hypothesis]
\label{rem:moving-packages-not-one-measure}
For a moving family \(\{W_n\}\), the quantities
\(P^a_{W_n}\) and \(\Pi_{W_n}\) usually come from different rescaled packages.
Thus there is generally a family of measures \(\mu_n^{\PF}\), not a single
measure \(\mu^{\PF}\).  The finite-overlap proof above does not apply unless
one proves a common-measure domination after pulling the windows back to a
single physical cylinder.  Without this additional domination, the tail budget
of \Cref{lem:tail-budget-pf-packing} remains a valid tautological budget, but
finite overlap alone does not prove intrinsic Navier--Stokes depletion.
\end{remark}

\begin{remark}[Why the tail budget is not yet intrinsic]
The tail budget \(B_n^{\PF}\) is a rigorous telescoping budget, but it is generated from the selected pressure--flux budgets themselves.  It is not yet an intrinsic monotone Navier--Stokes quantity such as a localized energy envelope.  A stronger theorem must identify \(B_n\) directly with a local energy--pressure quantity derived from the local energy inequality.  This is the next analytic target.
\end{remark}

\section{Intrinsic energy--pressure budget for pressure--flux activity}
\label{sec:intrinsic-energy-pressure-budget}

The preceding tail-budget construction gives a rigorous telescoping budget, but the budget is generated from the selected pressure--flux observations themselves.  We now formulate a more intrinsic version.  The goal is to relate pressure--flux activity to a local Navier--Stokes energy--pressure envelope.

The local energy inequality does not give a sign-definite monotone quantity at every scale.  Transport terms, pressure work, harmonic pressure leakage, and cutoff errors may all enter.  The correct statement is therefore a budget inequality with explicitly isolated work and localization errors.

Let \(z_0=(0,0)\), \(r_n=2^{-n}\), and
\[
        Q_n:=Q_{r_n}(z_0).
\]
Choose smooth parabolic cutoffs \(\phi_n\in C_c^\infty(Q_n)\) such that
\[
        0\le\phi_n\le1,
        \qquad
        \phi_n\equiv1 \text{ on } Q_{n+1},
\]
and
\begin{equation}
        |\nabla\phi_n|\lesssim r_n^{-1},
        \qquad
        |\partial_t\phi_n|+|\Delta\phi_n|\lesssim r_n^{-2}.
        \label{eq:dyadic-cutoff-bounds}
\end{equation}
Define the local energy--pressure envelope
\begin{align}
        \calE_n
        :=&\
        r_n^{-1}\esssup_{-r_n^2<t<0}
        \int_{B_{r_n}}|u(x,t)|^2\phi_n(x,t)\,dx
        +
        r_n^{-1}\int_{Q_n}|\nabla u|^2\phi_n\,dx\,dt
        \notag\\
        &+
        r_n^{-2}\int_{Q_n}|p-(p)_{B_{r_n}}(t)|^{3/2}\phi_n\,dx\,dt.
        \label{eq:energy-pressure-envelope}
\end{align}
This quantity is scale-invariant up to harmless cutoff constants.

\begin{definition}[Energy--pressure admissible pressure--flux window]
\label{def:ep-admissible-pf-window}
A pressure--flux window \(W_n\subset Q_n\) is called energy--pressure admissible with constants \((\lambda_n,\Err_n)\) if
\begin{equation}
        \lambda_n\bfB^{\PF}_{W_n}
        \le
        C\bigl(\calE_n-\calE_{n+1}\bigr)+\Err_n,
        \label{eq:ep-admissibility}
\end{equation}
where \(\lambda_n>0\), \(\Err_n\ge0\), and
\[
        \sum_{n\ge n_0}\Err_n<\infty.
\]
\end{definition}

\begin{remark}
The inequality \eqref{eq:ep-admissibility} is the intrinsic form of pressure--flux depletion.  Unlike the tail budget identity, it is not tautological.  It asserts that pressure--flux activity measured in the selected window is paid for by a decrease of a genuine local energy--pressure envelope, up to summable errors.
\end{remark}

\subsection{A local-energy work identity}

The following identity explains the origin of \eqref{eq:ep-admissibility}.  It is an exact coarse-grained local energy balance.

\begin{lemma}[Coarse-grained local energy work identity]
\label{lem:coarse-local-energy-work}
Let
\[
        U=U_{n,\ell},
        \qquad
        P=P_{n,\ell},
        \qquad
        R=R_{n,\ell},
        \qquad
        \Pi=-R:\nabla U
\]
be an NS-realizable coarse-grained package in an interior cylinder.  For every nonnegative \(\varphi\in C_c^\infty(Q_{\rm prep})\),
\begin{align}
        \int_{Q_{\rm prep}}
        (|\nabla U|^2+\Pi)\varphi\,dx\,dt
        ={}&
        \int_{Q_{\rm prep}}
        \frac12|U|^2(\partial_t+\Delta)\varphi\,dx\,dt
        \notag\\
        &+
        \int_{Q_{\rm prep}}
        \left[
            \left(\frac12|U|^2+P\right)U+RU
        \right]\cdot\nabla\varphi\,dx\,dt.
        \label{eq:coarse-local-energy-work}
\end{align}
Equivalently,
\begin{equation}
        \int_{Q_{\rm prep}}\Pi\varphi\,dx\,dt
        =
        -\int_{Q_{\rm prep}}|\nabla U|^2\varphi\,dx\,dt
        +
        \calW_\varphi(U,P,R),
        \label{eq:coarse-local-energy-work-rearranged}
\end{equation}
where
\begin{align}
        \calW_\varphi(U,P,R)
        :=&\
        \int_{Q_{\rm prep}}
        \frac12|U|^2(\partial_t+\Delta)\varphi\,dx\,dt
        \notag\\
        &+
        \int_{Q_{\rm prep}}
        \left[
            \left(\frac12|U|^2+P\right)U+RU
        \right]\cdot\nabla\varphi\,dx\,dt.
        \label{eq:work-functional}
\end{align}
\end{lemma}

\begin{proof}
The coarse-grained energy identity is
\[
        \partial_t\frac12|U|^2
        -
        \Delta\frac12|U|^2
        +
        |\nabla U|^2
        +
        \nabla\cdot\left[
            \left(\frac12|U|^2+P\right)U+RU
        \right]
        =
        -\Pi.
\]
Multiplying by \(\varphi\) and integrating by parts gives \eqref{eq:coarse-local-energy-work}.  The rearranged form \eqref{eq:coarse-local-energy-work-rearranged} follows immediately.
\end{proof}

\begin{remark}[Sign issue]
The identity controls the signed flux \(\int\Pi\varphi\), not automatically \(\int|\Pi|\varphi\).  Therefore an intrinsic depletion theorem for the absolute flux budget requires either a sign-selection of flux windows, a positive/negative flux decomposition, or an oscillation-loss term.
\end{remark}

\subsection{Forward signed pressure--flux depletion}

The sign in \eqref{eq:coarse-local-energy-work-rearranged} is important.  With
our convention \(\Pi=-R:\nabla U\), only the forward signed integral
\(\int \Pi\varphi\) can be paid by the same decreasing energy--pressure envelope
without adding a new positive dissipation cost.  Negative signed flux
(backscatter) must be treated as a separate mechanism or placed into an
oscillation/backscatter error.

\begin{definition}[Forward sign-coherent flux window]
\label{def:signed-flux-window}
A flux window \(W_n\) is called forward sign-coherent if there exists a
nonnegative test function \(\varphi_n\) supported in \(Q_{W_n}\) such that
\begin{equation}
        \int_{Q_{W_n}}|\Pi_{W_n}|\,dx\,dt
        \le
        C_W\int_{Q_{W_n}}\Pi_{W_n}\varphi_n\,dx\,dt
        +
        \mathrm{Osc}^F_n,
        \label{eq:sign-coherence}
\end{equation}
where \(\mathrm{Osc}^F_n\ge0\) is a flux-oscillation or backscatter error.
If the dominant coherent sign is negative, \eqref{eq:sign-coherence} is not
assumed; that case is classified as backscatter-dominated and is not paid by
the same monotone envelope.
\end{definition}

\begin{lemma}[Conditional forward flux work inequality]
\label{lem:flux-work-ep-envelope}
Let \(W_n\subset Q_n\) be a forward sign-coherent flux window.  Assume that the
annular cutoff work, harmonic-pressure leakage, and coarse-graining commutators
generated by applying \Cref{lem:coarse-local-energy-work} to a cutoff adapted to
\(W_n\) are collected in a nonnegative quantity \(\calR_n\).  Then
\begin{equation}
        \int_{Q_{W_n}}|\Pi_{W_n}|\,dx\,dt
        \le
        C_W(\calE_n-\calE_{n+1})+C_W\calR_n+\mathrm{Osc}^F_n.
        \label{eq:flux-work-envelope}
\end{equation}
A typical bound for \(\calR_n\) has the form
\begin{align}
        \calR_n
        \lesssim{}&
        \int_{Q_n\setminus Q_{n+1}}
        \left(
            r_n^{-2}|U|^2
            +
            r_n^{-1}|U|^3
            +
            r_n^{-1}|P-(P)_{B_{r_n}}|\,|U|
            +
            r_n^{-1}|R|\,|U|
        \right)dx\,dt
        \notag\\
        &+
        \Har_n+\Com_n.
        \label{eq:typical-work-error}
\end{align}
\end{lemma}

\begin{proof}
Apply \Cref{lem:coarse-local-energy-work} with a cutoff \(\varphi_n\) adapted
to \(W_n\), equal to \(1\) on the core and supported in \(Q_n\).  The signed
flux integral satisfies
\[
        \int \Pi\varphi_n
        =
        \calW_{\varphi_n}(U,P,R)
        -
        \int |\nabla U|^2\varphi_n
        \le
        \calW_{\varphi_n}(U,P,R).
\]
The nonnegative dissipation term therefore has the favorable sign only for this
forward signed flux.  The work terms containing \((\partial_t+\Delta)\varphi_n\)
and \(\nabla\varphi_n\) are supported where the cutoff changes, hence in
\(Q_n\setminus Q_{n+1}\), and are bounded by
\eqref{eq:typical-work-error} using \eqref{eq:dyadic-cutoff-bounds}.  The pressure is split into its active representative and its retained harmonic component; the latter is not a gauge and is placed in \(\Har_n\) as harmonic pressure-work leakage.  The difference between localized and
interior coarse-grained identities is placed in \(\Com_n\).  The resulting work
bound is written as
\[
        \calW_{\varphi_n}(U,P,R)
        \le
        C_W(\calE_n-\calE_{n+1})+C_W\calR_n.
\]
Finally, forward sign coherence \eqref{eq:sign-coherence} converts the forward
signed flux integral into the absolute flux budget up to \(\mathrm{Osc}^F_n\).
\end{proof}

\subsection{Pressure activity and the pressure envelope}

\begin{lemma}[Active pressure budget is controlled by pressure envelope]
\label{lem:active-pressure-envelope}
For every cleaned pressure window \(W_n\subset Q_n\),
\begin{equation}
        \int_{Q_{W_n}}|P^a_{W_n}|^{3/2}\,dx\,dt
        \le
        C_W r_n^2D(r_n)+\Har_n+\Loc_n.
        \label{eq:active-pressure-envelope}
\end{equation}
Equivalently, in scale-invariant form,
\[
        r_n^{-2}\int_{Q_{W_n}}|P^a_{W_n}|^{3/2}\,dx\,dt
        \le
        C_WD(r_n)+r_n^{-2}(\Har_n+\Loc_n).
\]
This estimate is an envelope bound, not a telescoping depletion estimate.  To use it in a dyadic budget, one must separately prove pressure-envelope admissibility, packing, or a selected-interval drop; otherwise it only says that active pressure is controlled by the local pressure size.
\end{lemma}

\begin{proof}
The active pressure representative is obtained by projecting
\[
        P-(P)_{B_{r_n}}(t)
\]
onto the finite active pressure representative.  Since the active projection is bounded on the finite window,
\[
        \|P^a_{W_n}\|_{L^{3/2}(Q_{W_n})}^{3/2}
        \le
        C_W\|P-(P)_{B_{r_n}}(t)\|_{L^{3/2}(Q_n)}^{3/2}
        +
        \Har_n+\Loc_n.
\]
By the definition of \(D(r_n)\),
\[
        \|P-(P)_{B_{r_n}}(t)\|_{L^{3/2}(Q_n)}^{3/2}
        =
        r_n^2D(r_n),
\]
which proves \eqref{eq:active-pressure-envelope}.
\end{proof}

\begin{theorem}[Conditional intrinsic controlled pressure--flux depletion]
\label{thm:intrinsic-controlled-pf-depletion}
Let \(W_n\subset Q_n\) be a cleaned, forward sign-coherent, perturbatively localized pressure--flux window.  The assertion below is conditional: the conversion of signed flux work and active pressure envelope control into a single telescoping drop is assumed through energy--pressure admissibility, not derived solely from the local energy identity.  Suppose the errors
\[
        \Err_n:=\calR_n+\mathrm{Osc}^F_n+\Err^P_n
\]
are nonnegative and summable, and suppose the resulting pressure and flux estimates imply energy--pressure admissibility in the sense of \Cref{def:ep-admissible-pf-window}:
\[
        \lambda_n\bfB^{\PF}_{W_n}
        \le
        C(\calE_n-\calE_{n+1})+\Err_n.
\]
Define
\[
        B_n:=C\calE_n+\sum_{k=n}^{\infty}\Err_k.
\]
Then
\begin{equation}
        B_n-B_{n+1}
        \ge
        \lambda_n\bfB^{\PF}_{W_n}.
        \label{eq:intrinsic-budget-drop}
\end{equation}
Consequently, if
\[
        |O^P_{W_n}D_{W_n}|+|O^F_{W_n}D_{W_n}|\ge\eta_n,
\]
then
\begin{equation}
        B_n-B_{n+1}
        \ge
        c_{W_n}\lambda_n\eta_n^{3/2}.
        \label{eq:intrinsic-observable-drop}
\end{equation}
Equivalently, keeping the errors outside the budget gives
\[
        C(\calE_n-\calE_{n+1})
        \ge
        c_{W_n}\lambda_n\eta_n^{3/2}-\Err_n.
\]
\end{theorem}

\begin{proof}
By definition of \(B_n\),
\[
        B_n-B_{n+1}
        =
        C(\calE_n-\calE_{n+1})+\Err_n.
\]
Energy--pressure admissibility gives \eqref{eq:intrinsic-budget-drop}.  The pressure--flux budget certificate, \Cref{lem:pf-observation-budget-controlled}, gives
\[
        \bfB^{\PF}_{W_n}
        \ge
        c_{W_n}
        \bigl(|O^P_{W_n}D_{W_n}|+|O^F_{W_n}D_{W_n}|\bigr)^{3/2}
\]
on normalized windows.  Substituting the observation lower bound yields \eqref{eq:intrinsic-observable-drop}.
\end{proof}

\begin{remark}[What remains conditional]
The intrinsic theorem has genuine hypotheses: forward sign coherence of the flux window, summability of cutoff/work errors, pressure-envelope admissibility for the active pressure part, and a clean active/harmonic pressure split.  These are not cosmetic.  Without forward sign coherence, the local energy identity controls signed flux, not \(|\Pi|\).  Without a separate pressure-envelope drop or packing principle, active pressure control by \(D(r_n)\) does not telescope.  Without control of the retained harmonic pressure component, pressure work can leak into \(\Har_n\).
\end{remark}

\begin{corollary}[Intrinsic pressure--flux observable depletion]
\label{cor:intrinsic-pf-observable-depletion}
Assume the hypotheses of \Cref{thm:intrinsic-controlled-pf-depletion}.  Suppose also that the pressure--flux quotient observability constant in \(W_n\) is \(M_n^{\PF}\), and that the normalized NS-realizable branch direction remains a distance at least \(\sigma_n\) from \(K^{\PF}_{W_n}\).  Then
\[
        B_n-B_{n+1}
        \ge
        c_{W_n}\lambda_n\sigma_n^{3/2}(M^{\PF}_n)^{-3/2}.
\]
If \(\sigma_n\ge\sigma_0>0\),
\[
        \sum_n\lambda_n(M^{\PF}_n)^{-3/2}=+\infty,
        \qquad
        c_{W_n}\ge c_0>0,
\]
then an indefinitely persistent pressure--flux visible non-CKN branch is impossible in this controlled window class, provided \(B_{n_0}<\infty\).
\end{corollary}

\begin{proof}
Quotient observability and the distance lower bound give \(\eta_n=\sigma_n(M_n^{\PF})^{-1}\).  Apply \Cref{thm:intrinsic-controlled-pf-depletion} with this value of \(\eta_n\) and sum in \(n\).
\end{proof}

\section{Energy positivity on the pressure--flux kernel}
\label{sec:energy-positivity}

We now analyze residual directions that survive the pressure--flux channels.  The previous sections show that pressure--flux visible defects consume a local budget.  Therefore the next possible obstruction must lie in the pressure--flux kernel
\[
        K^{\PF}_W=
        \ker O^P_W\cap\ker O^F_W.
\]
The purpose of this section is to show that a large part of this kernel is still visible through the energy channel, provided the direction is NS-realizable in the positive-covariance sense.

The important point is that a formal linear defect may have sign-changing stress variation, but an actual Reynolds covariance generated by coarse graining satisfies
\[
        R_{n,\ell}\ge0,
        \qquad
        \kappa_{n,\ell}:=\frac12\tr R_{n,\ell}\ge0.
\]
This positivity is a genuine NS-realizability constraint.

\subsection{The positive energy observation}

Fix a finite window \(W\).  Let
\[
        z=(\dot U,\dot R)\in K^{\PF}_W
\]
be a pressure--flux invisible active direction.  The corresponding pressure variation \(\dot P=\dot P^{\act}+\dot P^{\har}\) is determined by the constrained pressure compatibility relation
\[
        -\Delta\dot P
        =
        \partial_i\partial_j
        \bigl(U_i\dot U_j+\dot U_iU_j+\dot R_{ij}\bigr).
\]
Since \(z\in K^{\PF}_W\), this active pressure source is invisible in the chosen pressure window and the linearized flux observation also vanishes.

Let
\[
        \Theta^{\rm tr}_W=\{\theta_1,\dots,\theta_{N_{\rm tr}}\}
\]
be nonnegative selected-time kinetic tests, and let
\[
        \Theta^{\rm bulk}_W=\{\zeta_1,\dots,\zeta_{N_b}\}
\]
be nonnegative bulk tests supported in the observation region.  We define the positive energy seminorm
\begin{align}
        Q^E_W(z)
        :=&
        \left(
            \sum_{j=1}^{N_{\rm tr}}
            \int \theta_j(x)|\dot U(x,s_*)|^2\,dx
        \right)^{1/2}
        \notag\\
        &+
        \left(
            \sum_{j=1}^{N_b}
            \iint \zeta_j|\nabla\dot U|^2\,dx\,dt
        \right)^{1/2}
        +
        \sum_{j=1}^{N_b}
        \iint \zeta_j\frac12\tr\dot R\,dx\,dt,
        \label{eq:positive-energy-seminorm}
\end{align}
whenever \(\dot R\ge0\).  The last term is used only on the positive covariance cone.  For a general sign-changing formal variation, it is only a linear functional and cannot be treated as positive.

We also keep the linearized energy-balance observation.  For a test function \(\phi\in C_c^\infty(Q_W)\), define
\begin{align}
        L^E_Wz(\phi)
        =
        \iint
        \Big[{}&
        -(U\cdot\dot U)(\partial_t+\Delta)\phi
        \notag\\
        &-
        \Big(
            (U\cdot\dot U+\dot P)U
            +
            \Big(\frac12|U|^2+P\Big)\dot U
            +
            \dot R\,U
            +
            R\dot U
        \Big)\cdot\nabla\phi
        \notag\\
        &+
        \bigl(2\nabla U:\nabla\dot U+\dot\Pi\bigr)\phi
        \Big]dx\,dt,
        \label{eq:linearized-energy-observation}
\end{align}
where \(\dot\Pi=-\dot R:\nabla U-R:\nabla\dot U\).  The finite-dimensional energy observation is
\[
        O^E_Wz
        =
        \bigl(L^E_Wz(\phi_1),\dots,L^E_Wz(\phi_{N_E}),Q^E_W(z)\bigr).
\]

\begin{definition}[Positive NS-realizable pressure--flux kernel]
\label{def:positive-ns-pf-kernel}
Let \(K^{\PF,+}_{W,\NS}\) be the closed cone of pressure--flux invisible finite-window package increments whose Reynolds covariance component is retained with its NS-realizable sign.  Concretely, its elements \(z=(\dot U,\dot R)\) are cone-level limits of NS-realizable positive covariance packages, not arbitrary differences of two packages, and
\[
        \dot R\ge0
\]
as a quadratic form in the observation region.  If one instead studies a general linearized difference of packages, this positivity assumption is not available and the formal kernel must be treated by the linear observation in \Cref{prop:residual-pfe-kernel}.  The positive pressure--flux--energy kernel is
\begin{equation}
        K^{\PFE,+}_{W,\NS}
        :=
        \{z\in K^{\PF,+}_{W,\NS}:L^E_Wz=0,\ Q^E_W(z)=0\}.
        \label{eq:positive-pfe-kernel}
\end{equation}
\end{definition}

\begin{definition}[Energy-separating window]
\label{def:energy-separating-window}
A finite window \(W\) is called energy-separating on positive NS-realizable pressure--flux directions if the following two conditions hold.

First, the bulk and selected-time tests separate resolved gradients and kinetic traces on the active velocity window:
\[
        \sum_j\iint\zeta_j|\nabla\dot U|^2\,dx\,dt
        +
        \sum_j\int\theta_j|\dot U(\cdot,s_*)|^2\,dx
        =0
        \quad\Longrightarrow\quad
        \dot U=0
\]
for every active velocity variation \(\dot U\), modulo the quotient-null directions.

Second, the covariance tests separate nonnegative covariance traces:
\[
        \iint\zeta_j\tr\dot R\,dx\,dt=0
        \quad\text{for every }j
        \quad\Longrightarrow\quad
        \dot R=0
\]
for every nonnegative finite-window covariance tensor \(\dot R\ge0\).
\end{definition}

\begin{remark}
The second condition is automatic, for example, if the positive test family contains a function strictly positive on the core of the observation region and the finite-window covariance is supported in that core.  Since \(\dot R\ge0\), vanishing of \(\tr\dot R\) implies \(\dot R=0\).
\end{remark}

\begin{lemma}[Positive covariance is energy-visible]
\label{lem:positive-covariance-energy-visible}
Let \(W\) be energy-separating, and let
\[
        z=(0,\dot R)\in K^{\PF,+}_{W,\NS}.
\]
If \(\dot R\neq0\), then \(Q^E_W(z)>0\).  Consequently,
\[
        z\notin K^{\PFE,+}_{W,\NS}.
\]
\end{lemma}

\begin{proof}
Since \(\dot R\ge0\) and \(\dot R\neq0\), its trace is nonnegative and not identically zero.  By the covariance-separating property of the energy tests, there exists \(\zeta_j\) such that
\[
        \iint\zeta_j\tr\dot R\,dx\,dt>0.
\]
Therefore the covariance part of \(Q^E_W(z)\) is strictly positive.
\end{proof}

\begin{lemma}[Resolved gradients are energy-visible]
\label{lem:resolved-gradient-energy-visible}
Let \(W\) be energy-separating, and let
\[
        z=(\dot U,\dot R)\in K^{\PF,+}_{W,\NS}.
\]
If \(\dot U\neq0\) in the active quotient, then \(Q^E_W(z)>0\).  Consequently,
\[
        z\notin K^{\PFE,+}_{W,\NS}.
\]
\end{lemma}

\begin{proof}
If \(\dot U\neq0\) in the active quotient, then the energy-separating property implies that either the selected-time kinetic part is nonzero or the bulk gradient part is nonzero.  Hence the velocity part of \(Q^E_W(z)\) is strictly positive.
\end{proof}

\begin{theorem}[Positive energy anti-kernel on the pressure--flux kernel]
\label{thm:positive-energy-anti-kernel-pf}
Let \(W\) be a clean periodic or perturbatively localized finite window.  Assume that \(W\) is energy-separating on positive NS-realizable pressure--flux directions.  Then
\begin{equation}
        K^{\PFE,+}_{W,\NS}=\{0\}.
        \label{eq:positive-energy-anti-kernel}
\end{equation}
Equivalently, no nonzero positive-covariance NS-realizable pressure--flux kernel direction can be invisible to the energy channel.
\end{theorem}

\begin{proof}
Let \(z=(\dot U,\dot R)\in K^{\PFE,+}_{W,\NS}\).  By definition, \(Q^E_W(z)=0\).  Since \(z\) is positive NS-realizable, \(\dot R\ge0\).  The covariance part of \(Q^E_W\) gives
\[
        \iint\zeta_j\tr\dot R\,dx\,dt=0
        \qquad
        \text{for every }j.
\]
By covariance separation, \(\dot R=0\).  The velocity part of \(Q^E_W\) gives
\[
        \int\theta_j|\dot U(\cdot,s_*)|^2\,dx=0,
        \qquad
        \iint\zeta_j|\nabla\dot U|^2\,dx\,dt=0
        \qquad
        \text{for every }j.
\]
By velocity separation, \(\dot U=0\) in the active quotient.  Since the pressure variation is determined by
\[
        -\Delta\dot P
        =
        \partial_i\partial_j
        \bigl(U_i\dot U_j+\dot U_iU_j+\dot R_{ij}\bigr)
\]
with the active/harmonic pressure split, we also have \(\dot P^{\act}=0\) and no retained harmonic pressure signal in the cleaned representative.  Thus \(z=0\) in the cleaned active defect space.
\end{proof}

\begin{corollary}[Energy coercivity on the positive NS-realizable pressure--flux cone]
\label{cor:energy-coercivity-positive-pf-cone}
Under the hypotheses of \Cref{thm:positive-energy-anti-kernel-pf}, there exists a finite constant \(C^E_W<\infty\) such that
\begin{equation}
        \|z\|_{Z_W}
        \le
        C^E_W
        \bigl(|L^E_Wz|+Q^E_W(z)\bigr)
        \label{eq:energy-coercivity-positive-cone}
\end{equation}
for every \(z\in K^{\PF,+}_{W,\NS}\).
\end{corollary}

\begin{proof}
The set
\[
        S_W=
        \{z\in K^{\PF,+}_{W,\NS}:\|z\|_{Z_W}=1\}
\]
is compact because \(K^{\PF,+}_{W,\NS}\) is a closed cone in a finite-dimensional space.  The function
\[
        z\mapsto |L^E_Wz|+Q^E_W(z)
\]
is continuous and, by \Cref{thm:positive-energy-anti-kernel-pf}, strictly positive on \(S_W\).  Hence it has a positive minimum \(c^E_W>0\).  By homogeneity on the positive cone,
\[
        |L^E_Wz|+Q^E_W(z)
        \ge
        c^E_W\|z\|_{Z_W}.
\]
Taking \(C^E_W=(c^E_W)^{-1}\) proves \eqref{eq:energy-coercivity-positive-cone}.
\end{proof}

For sign-changing formal variations one must not use \(Q^E_W\) as a positive
seminorm.  Instead, fix a finite-dimensional linear energy observation
\[
        \mathcal E^{\rm lin}_W:K^{\PF}_W\longrightarrow E^{\rm lin}_W
\]
consisting of the linearized energy-balance tests \(L^E_Wz(\phi_j)\), together
with any chosen linear kinetic-trace, dissipation, and signed covariance trace
coefficients.  The positive quantity \(Q^E_W\) is reserved for the positive
covariance cone \(\dot R\ge0\).

\begin{proposition}[Residual pressure--flux--energy kernel]
\label{prop:residual-pfe-kernel}
For a general formal finite-window pressure--flux kernel, define
\begin{equation}
        K^{\PFE}_W
        :=
        \ker\left(\mathcal E^{\rm lin}_W|_{K^{\PF}_W}\right).
        \label{eq:formal-pfe-kernel}
\end{equation}
Then exactly one of the following alternatives holds.

\smallskip
\noindent\textup{(i) Energy visibility on \(K^{\PF}_W\).}
If \(K^{\PFE}_W=\{0\}\), then there exists \(C^E_W<\infty\) such that
\[
        \|z\|_{Z_W}
        \le
        C^E_W\|\mathcal E^{\rm lin}_Wz\|_{E^{\rm lin}_W}
\]
for every \(z\in K^{\PF}_W\).

\smallskip
\noindent\textup{(ii) Residual finite-dimensional energy kernel.}
If \(K^{\PFE}_W\neq\{0\}\), then every formal direction invisible to pressure,
flux, and the chosen linear energy observation lies in the finite-dimensional
kernel \(K^{\PFE}_W\).  It becomes relevant to the Navier--Stokes problem only
after intersection with the NS-realizable range.
\end{proposition}

\begin{proof}
This is the finite-dimensional kernel alternative applied to the linear energy
observation restricted to \(K^{\PF}_W\).  If the kernel is zero, compactness on
the unit sphere gives the coercive estimate.  If the kernel is nonzero, then by
definition it is precisely the finite-dimensional space of formal directions
invisible to pressure, flux, and the selected linear energy channels.
\end{proof}

\begin{remark}[What remains after energy positivity]
The positive energy argument removes all pressure--flux invisible directions that carry nonzero resolved velocity energy, nonzero resolved dissipation, or nonzero positive Reynolds covariance.  Therefore a remaining dangerous direction must be much more special.  It must either be a sign-changing formal stress variation, a localization artifact, an uncontrolled retained harmonic-pressure leakage, or a residual finite-dimensional direction whose NS-realizability is not captured by the positive covariance cone.  Such directions must next be tested by the adjoint trace channel and by the full NS-realizability filter.
\end{remark}

\section{Trace-cost exactification on the pressure--flux--energy kernel}
\label{subsec:trace-cost-exactification}

We now treat the residual directions that survive pressure, flux, and energy.
After the previous reductions, the only remaining finite-window obstruction lies
in the pressure--flux--energy kernel
\[
        K^{PFE}_W
        =
        K^{PF}_W\cap \ker O^E_W .
\]
A residual in \(K^{PFE}_W\) is invisible to the primitive channels, but it may
still be removable by a selected-time trace correction.  The purpose of this
subsection is to prove the finite-dimensional trace-cost alternative:
\[
        \text{trace-aligned residual}
        \quad\Longrightarrow\quad
        \text{controlled trace cost},
\]
while failure of trace alignment produces a left-singular invisible residual.

\subsection{The trace-defect map}

Let \(W=(n,\ell,\Lambda,\chi,s_*)\) be a finite observation window.  Let
\[
        Y_W
\]
be the cleaned active defect space, and let
\[
        H_W
\]
be the finite-dimensional selected-time trace correction space.  Elements of
\(H_W\) are divergence-free finite-window velocity corrections at time
\(s=s_*\).

The active trace-defect map is
\[
        A_W:H_W\longrightarrow Y_W .
\]
Given \(\xi\in H_W\), \(A_W\xi\) is the cleaned active defect generated by
evolving the selected-time correction through the linearized coarse-grained
Navier--Stokes system around the resolved background.

The adjoint map is
\[
        A_W^*:Y_W^*\longrightarrow H_W,
\]
defined by
\[
        \langle A_W\xi,y\rangle_{Y_W,Y_W^*}
        =
        \langle \xi,A_W^*y\rangle_{H_W}.
\]
A dual direction \(y\in Y_W^*\) with
\[
        A_W^*y=0
\]
is invisible to selected-time trace corrections.

Let
\[
        \iota_W:K^{PFE}_W\hookrightarrow Y_W
\]
denote the inclusion of the pressure--flux--energy residual kernel into the
active defect space.  For
\[
        z\in K^{PFE}_W,
\]
write
\[
        g=\iota_W z\in Y_W
\]
for the active residual.

\subsection{Trace cost and its dual formula}

\begin{definition}[Trace cost]
\label{def:trace-cost}
For \(g\in Y_W\), define the selected-time trace cost by
\[
        \operatorname{Cost}^{tr}_W(g)
        :=
        \inf
        \left\{
            \|\xi\|_{H_W}^2:
            A_W\xi=-g
        \right\},
\]
with the convention
\[
        \operatorname{Cost}^{tr}_W(g)=+\infty
\]
if \(-g\notin \operatorname{Range}A_W\).
\end{definition}

\begin{lemma}[Finite-dimensional trace-cost duality]
\label{lem:trace-cost-duality-clean}
Let \(H\) and \(Y\) be finite-dimensional Hilbert spaces, and let
\[
        A:H\to Y
\]
be linear.  For \(g\in Y\) and \(R\ge0\), the following are equivalent:
\[
        \operatorname{Cost}^{tr}_A(g)\le R^2,
\]
and
\[
        |\langle g,y\rangle|
        \le
        R\|A^*y\|_H
        \qquad
        \text{for every }y\in Y^*.
\]
Equivalently,
\[
        \sqrt{\operatorname{Cost}^{tr}_A(g)}
        =
        \sup_{A^*y\ne0}
        \frac{|\langle g,y\rangle|}
             {\|A^*y\|_H},
\]
with value \(+\infty\) if \(g\notin \operatorname{Range}A\).
\end{lemma}

\begin{proof}
Assume first that \(A\xi=-g\).  Then for every \(y\in Y^*\),
\[
        |\langle g,y\rangle|
        =
        |\langle -A\xi,y\rangle|
        =
        |\langle \xi,A^*y\rangle|
        \le
        \|\xi\|_H\|A^*y\|_H.
\]
Taking the infimum over all such \(\xi\) gives one implication.

Conversely, suppose the dual inequality holds.  If \(y\in\ker A^*\), then
\[
        |\langle g,y\rangle|\le0,
\]
so \(g\perp \ker A^*\).  Since the spaces are finite-dimensional,
\[
        (\ker A^*)^\perp=\operatorname{Range}A.
\]
Thus \(g\in\operatorname{Range}A\), and the trace equation \(A\xi=-g\) is
solvable.

Let \(\xi_0\) be the minimal-norm solution of \(A\xi=-g\).  Then
\[
        \xi_0\in \operatorname{Range}A^*,
\]
so there exists \(y_0\in Y^*\) such that
\[
        A^*y_0=\xi_0.
\]
Using \(A\xi_0=-g\), we get
\[
        \|\xi_0\|_H^2
        =
        \langle \xi_0,A^*y_0\rangle
        =
        \langle A\xi_0,y_0\rangle
        =
        -\langle g,y_0\rangle.
\]
Therefore, by the dual inequality,
\[
        \|\xi_0\|_H^2
        \le
        R\|A^*y_0\|_H
        =
        R\|\xi_0\|_H.
\]
Hence
\[
        \|\xi_0\|_H\le R,
\]
which proves
\[
        \operatorname{Cost}^{tr}_A(g)\le R^2.
\]
The supremum formula follows by optimizing over \(R\).
\end{proof}

\subsection{Trace alignment}

Small singular values of \(A_W\) are not dangerous by themselves.  They are
dangerous only if the residual \(g\) has a non-negligible component in the
corresponding left singular directions.

\begin{definition}[Trace-aligned residual]
\label{def:trace-aligned-residual}
Let \(\rho_W>0\) be the branch-native residual scale in the window \(W\).
A residual \(g\in Y_W\) is called trace-aligned at scale \(\rho_W\) if
\[
        |\langle g,y\rangle|
        \le
        \rho_W\|A_W^*y\|_{H_W}
        \qquad
        \text{for every }y\in Y_W^*.
\]
More generally, \(g\) is called combined trace-aligned at scale \(\rho_W\) if
\[
        |\langle g,y\rangle|
        \le
        \rho_W
        \left(
            \|A_W^*y\|_{H_W}
            +
            \|J^E_Wy\|
            +
            \|J^P_Wy\|
            +
            \|J^F_Wy\|
        \right)
\]
for every \(y\in Y_W^*\).
\end{definition}

\begin{lemma}[Trace alignment controls trace cost]
\label{lem:trace-alignment-controls-cost}
If \(g\in Y_W\) is trace-aligned at scale \(\rho_W\), then
\[
        \operatorname{Cost}^{tr}_W(g)\le \rho_W^2.
\]
If \(g\) is combined trace-aligned and the primitive channels are absorbed in the precise sense that
\[
        \|J^E_Wy\|+\|J^P_Wy\|+\|J^F_Wy\|
        \le
        C_W\|A_W^*y\|_{H_W}
\]
for every dual direction relevant to the projected residual \(g\), then
\[
        \operatorname{Cost}^{tr}_W(g)\le C_W\rho_W^2.
\]
\end{lemma}

\begin{proof}
The first statement is exactly
\cref{lem:trace-cost-duality-clean} with \(A=A_W\) and \(R=\rho_W\).

For the second statement, the displayed primitive-absorption inequality converts combined trace alignment into
\[
        |\langle g,y\rangle|
        \le
        C_W\rho_W\|A_W^*y\|_{H_W}
\]
for all dual directions that can pair with the projected residual.  Equivalently, after projecting to the residual pressure--flux--energy annihilator, the primitive dual channels have already been paid for and the remaining trace equation is tested only against the trace channel.  Applying \cref{lem:trace-cost-duality-clean} gives
\[
        \operatorname{Cost}^{tr}_W(g)
        \le
        C_W\rho_W^2.
\]
\end{proof}

\subsection{The pressure--flux--energy--trace obstruction}

\begin{definition}[PFET trace obstruction space]
\label{def:pfet-kernel-clean}
Let
\[
        \Pi_W^\perp:Y_W\longrightarrow (\operatorname{Range}A_W)^\perp
        =\ker A_W^*
\]
be the orthogonal projection in the finite-dimensional Hilbert structure chosen
on \(Y_W\).  The pressure--flux--energy--trace obstruction space is
\begin{equation}
        \mathcal T_W^{\PFET}
        :=
        \Pi_W^\perp\,\iota_W(K_W^{\PFE})
        \subset (\operatorname{Range}A_W)^\perp .
        \label{eq:pfet-obstruction-space}
\end{equation}
The trace-projectable residual subspace is
\begin{equation}
        K^{\PFE,{\rm proj}}_W
        :=
        \{z\in K_W^{\PFE}:\Pi_W^\perp\iota_Wz=0\}
        =
        \{z\in K_W^{\PFE}:\iota_Wz\in\operatorname{Range}A_W\}.
        \label{eq:pfe-projectable-residuals}
\end{equation}
A residual \(z\in K_W^{\PFE}\) is trace-obstructed when
\(\Pi_W^\perp\iota_Wz\ne0\).
\end{definition}

\begin{remark}[Why the obstruction is a projection]
The condition
\(\exists y\in\ker A_W^*\) with \(\langle\iota_Wz,y\rangle\ne0\)
only says that \(\iota_Wz\) has a nonzero component orthogonal to
\(\operatorname{Range}A_W\).  It does not say that the whole residual is
orthogonal to \(\operatorname{Range}A_W\).  Therefore the correct trace
obstruction is the projected component \(\Pi_W^\perp\iota_Wz\), not the
intersection condition \(\iota_Wz\perp\operatorname{Range}A_W\).
\end{remark}

\begin{proposition}[Finite-window PFE--trace alternative]
\label{prop:pfe-trace-alternative-clean}
For a fixed finite window \(W\), the following are equivalent:
\begin{enumerate}
\item every residual \(g=\iota_Wz\), \(z\in K_W^{\PFE}\), lies in
\(\operatorname{Range}A_W\);
\item \(\mathcal T_W^{\PFET}=\{0\}\);
\item for every \(z\in K_W^{\PFE}\) and every \(y\in\ker A_W^*\),
\[
        \langle \iota_Wz,y\rangle=0.
\]
\end{enumerate}
If these equivalent conditions hold, every pressure--flux--energy invisible
residual has finite trace cost.  If they fail, then there exist
\(z\in K_W^{\PFE}\) and \(y\in\ker A_W^*\) such that
\[
        \langle \iota_Wz,y\rangle\ne0,
\]
and the nonzero vector \(\Pi_W^\perp\iota_Wz\in\mathcal T_W^{\PFET}\) is a
finite-window trace obstruction.
\end{proposition}

\begin{proof}
Since \(Y_W\) is finite-dimensional,
\[
        Y_W
        =
        \operatorname{Range}A_W
        \oplus
        (\operatorname{Range}A_W)^\perp,
        \qquad
        (\operatorname{Range}A_W)^\perp=\ker A_W^*.
\]
A vector \(g\in Y_W\) belongs to \(\operatorname{Range}A_W\) if and only if
\(\Pi_W^\perp g=0\), equivalently if and only if it pairs to zero with every
\(y\in\ker A_W^*\).  Applying this to all \(g=\iota_Wz\) with
\(z\in K_W^{\PFE}\) proves the equivalence.  If the projected component is zero,
the equation \(A_W\xi=-g\) is solvable and the finite-dimensional trace cost is
finite.  If the projected component is nonzero, a separating vector in
\((\operatorname{Range}A_W)^\perp\) gives a nonzero pairing.
\end{proof}

\subsection{NS-realizable PFET obstruction}

A finite-window trace obstruction is still only formal.  It becomes relevant to
Navier--Stokes only if it is generated by actual dyadic defect packages.

Let
\[
        \mathcal R^{NS,T}_W
\]
be the map sending an NS-realizable dyadic package to its cleaned residual in
\(Y_W\) after the pressure--flux--energy reductions.  Define the NS-realizable
PFET obstruction space by
\begin{equation}
        \mathcal T_{W,\NS}^{\PFET}
        :=
        \Pi_W^\perp\,
        \iota_W
        \left(
          K_W^{\PFE}
          \cap
          \overline{\operatorname{Range}\mathcal R^{NS,T}_W}
        \right)
        \subset (\operatorname{Range}A_W)^\perp .
        \label{eq:ns-pfet-obstruction-space}
\end{equation}
The closure is taken in the finite-dimensional topology.

\begin{theorem}[Finite-window anti-phantom criterion]
\label{thm:finite-window-anti-phantom-clean}
If
\begin{equation}
        \mathcal T_{W,\NS}^{\PFET}=\{0\},
        \label{eq:ns-pfet-zero}
\end{equation}
then every NS-realizable pressure--flux--energy residual is trace-projectable.
Equivalently, no NS-realizable dyadic defect has a nonzero residual component
invisible to pressure, flux, energy, and selected-time trace exactification in
the window \(W\).

If \(\mathcal T_{W,\NS}^{\PFET}\ne\{0\}\), then every NS-realizable
finite-window trace obstruction is represented by a vector in this finite-dimensional
space.
\end{theorem}

\begin{proof}
An NS-realizable residual surviving pressure, flux, and energy belongs to
\[
K_W^{\PFE}\cap\overline{\operatorname{Range}\mathcal R_W^{NS,T}}.
\]
Its trace obstruction is exactly the projected component
\(\Pi_W^\perp\iota_Wz\).  Thus vanishing of
\(\mathcal T_{W,\NS}^{\PFET}\) is equivalent to trace projectability of every
such NS-realizable residual.  If the obstruction space is nonzero, all
NS-realizable finite-window trace obstruction components lie in it by
definition.
\end{proof}

\subsection{Moving-window trace failure}

The previous results are fixed-window statements.  A singular branch may still
escape through a moving sequence of windows whose trace maps have degenerating
singular values.  We isolate the precise failure object.

\begin{definition}[Residual left-singular invisible cascade]
\label{def:residual-left-singular-cascade-clean}
A residual left-singular invisible cascade is a sequence
\[
        (W_n,g_n,y_n,\rho_n)
\]
such that:

\begin{enumerate}
\item \(W_n=(n,\ell_n,\Lambda_n,\chi_n,s_n)\) is a moving sequence of finite
windows;

\item \(g_n\in Y_{W_n}\) is an NS-realizable residual surviving the
pressure--flux--energy filters;

\item \(\rho_n>0\) is the branch-native residual scale and
\(\|g_n\|_{Y_{W_n}}\le C\rho_n\);

\item \(y_n\in Y_{W_n}^*\) is normalized:
\[
        \|y_n\|_{Y_{W_n}^*}=1;
\]

\item the combined dual observations vanish:
\[
        \|A_{W_n}^*y_n\|_{H_{W_n}}
        +
        \|J^E_{W_n}y_n\|
        +
        \|J^P_{W_n}y_n\|
        +
        \|J^F_{W_n}y_n\|
        \to0;
\]

\item the residual pairs nontrivially at the normalized residual scale,
\begin{equation}
        |\langle g_n,y_n\rangle|
        \ge
        c_0\rho_n
        \label{eq:strong-left-singular-pairing}
\end{equation}
for some \(c_0>0\), or at least the relative trace-alignment failure holds:
\begin{equation}
        \frac{|\langle g_n,y_n\rangle|}
        {\rho_n\left(
            \|A_{W_n}^*y_n\|_{H_{W_n}}
            +\|J^E_{W_n}y_n\|
            +\|J^P_{W_n}y_n\|
            +\|J^F_{W_n}y_n\|
        \right)}
        \longrightarrow\infty.
        \label{eq:relative-left-singular-failure}
\end{equation}
\end{enumerate}
The first pairing condition gives a strong invisible cascade; the second gives
a relative left-singular failure.  The strong conclusion requires a nontrivial
normalized residual pairing and is not automatic from failure of a uniform
trace-cost estimate alone.
\end{definition}

\begin{theorem}[Trace-cost closure versus left-singular failure]
\label{thm:trace-cost-closure-left-singular-clean}
Let \(W_n\) be a moving sequence of finite windows, and let
\[
        g_n\in Y_{W_n}
\]
be NS-realizable residuals surviving the pressure--flux--energy reductions.
Assume that
\begin{equation}
        \|g_n\|_{Y_{W_n}}\le C_g\rho_n
        \label{eq:residual-scale-bound}
\end{equation}
for a branch-native scale \(\rho_n>0\).

If there exists \(C<\infty\) such that
\begin{align}
        |\langle g_n,y\rangle|
        \le{}&
        C\rho_n
        \left(
            \|A_{W_n}^*y\|_{H_{W_n}}
            +
            \|J^E_{W_n}y\|
            +
            \|J^P_{W_n}y\|
            +
            \|J^F_{W_n}y\|
        \right)
        \label{eq:combined-trace-alignment-moving}
\end{align}
for every \(y\in Y_{W_n}^*\), then \(g_n\) has controlled combined trace cost.
In particular, after the primitive channels have been absorbed,
\begin{equation}
        \operatorname{Cost}^{tr}_{W_n}(g_n)
        \le
        C\rho_n^2.
        \label{eq:moving-trace-cost-bound}
\end{equation}

Conversely, if no estimate of the form
\eqref{eq:combined-trace-alignment-moving} holds along a subsequence, then,
after passing to a further subsequence, there exist \(y_n\in Y_{W_n}^*\) with
\(\|y_n\|=1\) for which the combined dual observations tend to zero and the
relative failure \eqref{eq:relative-left-singular-failure} holds.  If in
addition
\[
        \limsup_n\frac{|\langle g_n,y_n\rangle|}{\rho_n}>0,
\]
then a strong residual left-singular invisible cascade satisfying
\eqref{eq:strong-left-singular-pairing} exists.
\end{theorem}

\begin{proof}
The first statement is \Cref{lem:trace-alignment-controls-cost}.  Once the
primitive channels have been absorbed on the residual kernel,
\eqref{eq:combined-trace-alignment-moving} reduces to
\[
        |\langle g_n,y\rangle|
        \le
        C\rho_n\|A_{W_n}^*y\|_{H_{W_n}},
\]
and finite-dimensional trace-cost duality gives
\eqref{eq:moving-trace-cost-bound}.

For the converse, suppose \eqref{eq:combined-trace-alignment-moving} fails
along a subsequence.  Then for each index in that subsequence there exists
\(y_n\in Y_{W_n}^*\), normalized by \(\|y_n\|=1\), such that
\begin{align}
        |\langle g_n,y_n\rangle|
        >{}&
        n\rho_n
        \left(
            \|A_{W_n}^*y_n\|_{H_{W_n}}
            +
            \|J^E_{W_n}y_n\|
            +
            \|J^P_{W_n}y_n\|
            +
            \|J^F_{W_n}y_n\|
        \right).
        \label{eq:failed-trace-alignment-witness}
\end{align}
By \eqref{eq:residual-scale-bound},
\[
        |\langle g_n,y_n\rangle|\le C_g\rho_n.
\]
Combining this with \eqref{eq:failed-trace-alignment-witness} gives
\[
        \|A_{W_n}^*y_n\|_{H_{W_n}}
        +
        \|J^E_{W_n}y_n\|
        +
        \|J^P_{W_n}y_n\|
        +
        \|J^F_{W_n}y_n\|
        \le
        \frac{C_g}{n}.
\]
Thus all combined dual observations tend to zero, and
\eqref{eq:failed-trace-alignment-witness} gives the relative failure
\eqref{eq:relative-left-singular-failure}.  The additional limsup assumption is
exactly what is needed to pass to a subsequence with the strong lower bound
\eqref{eq:strong-left-singular-pairing}.
\end{proof}

\begin{corollary}[Output of the trace-cost step]
\label{cor:trace-cost-step-output}
After pressure, flux, and energy have been removed, exactly one of the following
holds in a controlled finite-window family:

\begin{enumerate}
\item every NS-realizable residual is trace-aligned and hence has controlled
selected-time trace cost;

\item the NS-realizable PFET obstruction space
\(\mathcal T_{W,\NS}^{\PFET}\) is nonzero;

\item along a moving sequence of windows, there exists a residual left-singular
invisible cascade, at least in the relative sense of
\eqref{eq:relative-left-singular-failure}.
\end{enumerate}
\end{corollary}

\begin{remark}
This is the final finite-window obstruction.  A small singular value of
\(A_W\) is not by itself dangerous.  It becomes dangerous only if an
NS-realizable residual survives pressure, flux, and energy, and then pairs
nontrivially with the corresponding nearly trace-invisible left singular
direction.
\end{remark}

\section{Moving-window growth control and the defect-cascade alternative}
\label{sec:moving-window-alternative}

We now assemble the finite-window hierarchy into a scale-by-scale alternative.
The result is a conditional reduction theorem.  It does not prove unconditional
regularity and it does not construct a singular solution.  Its purpose is to
state precisely what must persist if a point fails to enter the
Caffarelli--Kohn--Nirenberg smallness regime at all sufficiently small dyadic
scales.

A technical point is important.  A defect extraction theorem may produce
nontrivial defects only on a subsequence of dyadic scales.  One cannot then sum
one-step inequalities \(B_n-B_{n+1}\) over all integers unless the inequalities
hold for all sufficiently large \(n\).  The formulation below therefore uses an
extraction index set and a matching selected-interval budget.  The full-sequence
case is recovered by taking the extraction set to be all large integers.

\subsection{Non-CKN dyadic branches}

Let \((u,p)\) be a suitable weak solution in \(Q_1\), and fix
\[
        z_0=(0,0).
\]
Let
\[
        r_n=2^{-n},
        \qquad
        \Psi(r)=C(r)+D(r).
\]
Let \(\varepsilon_{\rm CKN}>0\) be a Caffarelli--Kohn--Nirenberg smallness
constant.

\begin{definition}[Non-CKN dyadic branch]
\label{def:non-ckn-branch}
We say that \(z_0\) has a non-CKN dyadic branch if there exists \(n_0\) such
that
\begin{equation}
        \Psi(r_n)>\varepsilon_{\rm CKN}
        \qquad
        \text{for every } n\ge n_0 .
        \label{eq:non-ckn-branch}
\end{equation}
If no such branch exists, then a CKN scale occurs and \(z_0\) is regular.
\end{definition}

For each dyadic scale \(r_n\), choose a coarse-graining length \(\ell_n\), and
construct the NS-realizable package
\[
        \mathfrak D_{n,\ell_n}
        =
        (U_{n,\ell_n},P_{n,\ell_n};P^{\act}_{n,\ell_n},P^{\har}_{n,\ell_n},R_{n,\ell_n},\Pi_{n,\ell_n}).
\]
Let
\[
        W_n=(n,\ell_n,\Lambda_n,\chi_n,s_n)
\]
be a moving finite observation window and set \(Y_n:=Y_{W_n}\).  The cleaned
finite-window projection of the dyadic package is
\[
        d_n=[\Pi_{W_n}\mathfrak D_{n,\ell_n}]\in Y_n .
\]

\begin{assumption}[Dyadic defect extraction on an index set]
\label{ass:dyadic-defect-extraction-final}
There exist \(c_0>0\), an infinite increasing index set
\[
        I=\{n_k:k\ge k_0\}\subset\{n\ge n_0\},
\]
windows \(W_{n_k}\), and NS-realizable cleaned defect directions
\(d_{n_k}\in Y_{n_k}\) such that
\begin{equation}
        \|d_{n_k}\|_{Y_{n_k}}\ge c_0
        \qquad
        \text{for every }k\ge k_0.
        \label{eq:extracted-defect-lower-bound}
\end{equation}
We write
\[
        \widehat d_{n_k}
        :=
        \frac{d_{n_k}}{\|d_{n_k}\|_{Y_{n_k}}}.
\]
The full-extraction case means \(I\) contains every sufficiently large integer.
\end{assumption}

\begin{remark}[Subsequence extraction and summation]
If only a sparse subsequence \(I\) is available, the later depletion estimate
must also be formulated on the same selected intervals.  A one-step inequality
for \(B_n-B_{n+1}\) on selected indices alone does not telescope over the
subsequence unless additional monotonicity or interval-depletion information is
proved.
\end{remark}

\subsection{Combined observed strength}

For a moving window \(W_n\), define the combined primal observed strength by
\begin{equation}
        \mathsf O_n(d)
        :=
        \|O^P_nd\|+
        \|O^F_nd\|+
        \|O^E_nd\|+
        \|O^T_nd\|.
        \label{eq:combined-primal-strength}
\end{equation}
The corresponding dual observed strength is
\begin{equation}
        \mathsf O_n^*(y)
        :=
        \|J^P_ny\|+
        \|J^F_ny\|+
        \|J^E_ny\|+
        \|A_n^*y\|_{H_n}.
        \label{eq:combined-dual-strength}
\end{equation}
Let \(Y_n^{\NS}\subset Y_n\) denote the finite-dimensional closure of
NS-realizable cleaned defect directions in the window \(W_n\).  The branch
observability constant is
\begin{equation}
        M_n
        :=
        \sup\left\{
        \frac{\|d\|_{Y_n}}{\mathsf O_n(d)}:
        0\ne d\in Y_n^{\NS}
        \right\},
        \label{eq:branch-observability-constant}
\end{equation}
with the convention \(M_n=+\infty\) if the denominator vanishes for some
nonzero \(d\in Y_n^{\NS}\).  Whenever \(M_n<\infty\), every normalized
NS-realizable defect satisfies
\begin{equation}
        \mathsf O_n(\widehat d)\ge M_n^{-1}.
        \label{eq:observability-lower-bound}
\end{equation}

\begin{definition}[Moving-window invisible defect cascade]
\label{def:moving-window-invisible-cascade-final}
A moving-window invisible defect cascade is a sequence
\[
        (W_{n_k},d_{n_k},y_{n_k})
\]
along an extraction index set \(I=\{n_k\}\) such that:
\begin{enumerate}
\item \(d_{n_k}\in Y_{n_k}\) is the cleaned projection of an NS-realizable dyadic
package and \(\|d_{n_k}\|_{Y_{n_k}}\ge c_0>0\);
\item \(y_{n_k}\in Y_{n_k}^*\) is normalized:
\[
        \|y_{n_k}\|_{Y_{n_k}^*}=1;
\]
\item the combined dual observations vanish:
\[
        \mathsf O_{n_k}^*(y_{n_k})\to0;
\]
\item the NS-derived defect still pairs nontrivially with this invisible dual
direction:
\[
        |\langle d_{n_k},y_{n_k}\rangle|\ge c_1>0
\]
for some \(c_1>0\).  A normalized-scale variant replaces the last bound by
\(|\langle d_{n_k},y_{n_k}\rangle|\ge c_1\|d_{n_k}\|_{Y_{n_k}}\).
\end{enumerate}
After the pressure--flux--energy reductions, this object may be replaced by the
residual left-singular cascade of \Cref{def:residual-left-singular-cascade-clean}.
\end{definition}

\subsection{Observable depletion on selected intervals}

The finite-window hierarchy becomes useful across scales only if observed
defects consume a finite dyadic budget on the same index set on which the
defects are extracted.

\begin{assumption}[Observable depletion on the extraction set]
\label{ass:observable-depletion-final}
Let \(I=\{n_k:k\ge k_0\}\) be the extraction index set from
\Cref{ass:dyadic-defect-extraction-final}.  There exist a nonnegative selected
budget \(\mathscr B_k\), weights \(\lambda_{n_k}>0\), an exponent \(q\ge1\), and
errors \(e_k\ge0\) such that
\begin{equation}
        \sum_{k=k_0}^{\infty}e_k<\infty,
        \label{eq:selected-errors-summable}
\end{equation}
and every extracted normalized NS-realizable defect satisfies
\begin{equation}
        \mathscr B_k-\mathscr B_{k+1}
        \ge
        c\lambda_{n_k}\mathsf O_{n_k}(\widehat d_{n_k})^q-e_k.
        \label{eq:selected-observable-depletion}
\end{equation}
In the full-extraction case one may take \(\mathscr B_k=B_{n_k}\) with
\(n_k=k\) after reindexing.  If extraction is genuinely sparse, then
\eqref{eq:selected-observable-depletion} is an interval-depletion hypothesis
and must be proved directly; it is not a consequence of one-step inequalities
on the selected indices alone.
\end{assumption}

\subsection{Effective moving-window observability}

\begin{definition}[Depletion-effective observability]
\label{def:depletion-effective-observability}
A moving sequence of windows is depletion-effective on the extraction set
\(I=\{n_k\}\) if \(M_{n_k}<\infty\) for all sufficiently large \(k\) and
\begin{equation}
        \sum_{k=k_0}^{\infty}\lambda_{n_k}M_{n_k}^{-q}=+\infty.
        \label{eq:selected-effective-observability}
\end{equation}
If \(M_{n_k}=+\infty\) along a subsequence, or if the series in
\eqref{eq:selected-effective-observability} converges, the moving-window
observability is called non-effective on the extracted branch.
\end{definition}

\subsection{Growth profiles and logarithmically admissible windows}

A typical controlled growth profile has the form
\begin{equation}
        M_n\le C N_n^a\exp(CN_n^b),
        \label{eq:growth-profile}
\end{equation}
where \(N_n\) is the effective size of the active window \(\Lambda_n\).

\begin{definition}[Logarithmically admissible windows]
\label{def:log-admissible-windows}
Assume the observability constants satisfy \eqref{eq:growth-profile}.  A moving
window family is logarithmically admissible relative to the extraction set and
the depletion weights if
\begin{equation}
        \sum_{k=k_0}^{\infty}
        \lambda_{n_k}N_{n_k}^{-aq}\exp(-qC N_{n_k}^b)
        =+
        \infty.
        \label{eq:log-admissible}
\end{equation}
\end{definition}

\begin{lemma}[A simple logarithmic admissibility criterion]
\label{lem:log-admissibility-criterion}
Suppose \(I\) contains all sufficiently large integers,
\[
        M_n\le C N_n^a\exp(CN_n^b),
        \qquad
        \lambda_n\ge c(n+2)^{-s}
\]
for some \(s<1\).  If
\begin{equation}
        N_n^b\le \alpha\log(n+2)
        \label{eq:log-window-choice}
\end{equation}
with \(\alpha>0\) small enough that
\begin{equation}
        s+qC\alpha<1,
        \label{eq:log-admissibility-condition}
\end{equation}
then
\[
        \sum_n\lambda_nM_n^{-q}=+\infty.
\]
For a sparse extraction set, the same conclusion holds with the sum restricted
to \(I\) whenever the corresponding restricted series diverges.
\end{lemma}

\begin{proof}
From the growth bound,
\[
        M_n^{-q}\ge C^{-q}N_n^{-aq}\exp(-qC N_n^b).
\]
Using \eqref{eq:log-window-choice},
\[
        \exp(-qC N_n^b)\ge(n+2)^{-qC\alpha},
\]
and \(N_n\) grows at most logarithmically.  Therefore
\[
        \lambda_nM_n^{-q}
        \ge
        c(n+2)^{-s-qC\alpha}(\log(n+2))^{-aq/b}.
\]
The series diverges when \eqref{eq:log-admissibility-condition} holds.  The
restricted-index statement is exactly the same estimate with the sum taken over
\(n\in I\).
\end{proof}

\subsection{The main moving-window alternative}

\begin{theorem}[Moving-window defect-cascade alternative]
\label{thm:moving-window-defect-cascade-alternative}
Let \((u,p)\) be a suitable weak solution in \(Q_1\), and let \(z_0=(0,0)\).
Assume:
\begin{enumerate}
\item dyadic defect extraction holds on an index set \(I=\{n_k\}\), in the sense
of \Cref{ass:dyadic-defect-extraction-final};
\item observable depletion holds on the same index set, in the sense of
\Cref{ass:observable-depletion-final};
\item the finite-window hierarchy has been applied, so that failure of combined
observability along the extracted windows produces either a finite-window
NS-realizable trace obstruction or a residual left-singular invisible cascade.
\end{enumerate}
Suppose \(z_0\) has a non-CKN dyadic branch.  Then at least one of the following
alternatives holds.

\smallskip
\noindent\textup{(i) Non-effective moving-window observability.}
The extracted moving-window constants fail to be depletion-effective:
\[
        M_{n_k}=+\infty\quad\text{along a subsequence},
\]
or
\begin{equation}
        \sum_{k=k_0}^{\infty}\lambda_{n_k}M_{n_k}^{-q}<\infty.
        \label{eq:noneffective-observability-alt}
\end{equation}

\smallskip
\noindent\textup{(ii) NS-realizable invisible defect cascade.}
There exists a moving-window invisible defect cascade along the extracted
windows.  After the pressure--flux--energy reductions, this cascade may be taken
to be a residual left-singular invisible defect cascade, possibly in the
relative sense of \eqref{eq:relative-left-singular-failure}.

Consequently, in any controlled regime where observability is
depletion-effective on the extraction set and no NS-realizable invisible defect
cascade exists, a CKN scale must occur.  Hence \(z_0\) is regular.
\end{theorem}

\begin{proof}
Assume, for contradiction, that there is a non-CKN dyadic branch and that both
alternatives fail.  By dyadic defect extraction, there are extracted windows
\(W_{n_k}\) and normalized defects \(\widehat d_{n_k}\).  Since the
non-effective alternative fails, \(M_{n_k}<\infty\) for all large \(k\), and
\[
        \sum_k\lambda_{n_k}M_{n_k}^{-q}=+\infty.
\]
Since the invisible-cascade alternative also fails, the finite-window hierarchy
leaves the extracted defects in the observed class.  By the definition of
\(M_{n_k}\),
\[
        \mathsf O_{n_k}(\widehat d_{n_k})\ge M_{n_k}^{-1}.
\]
Observable depletion on selected intervals gives
\[
        \mathscr B_k-\mathscr B_{k+1}
        \ge
        c\lambda_{n_k}M_{n_k}^{-q}-e_k.
\]
Summing from \(k_0\) to \(K\) yields
\[
        \mathscr B_{k_0}-\mathscr B_{K+1}
        \ge
        c\sum_{k=k_0}^{K}\lambda_{n_k}M_{n_k}^{-q}
        -
        \sum_{k=k_0}^{K}e_k.
\]
The budget is nonnegative, so the left-hand side is at most
\(\mathscr B_{k_0}<\infty\).  The errors are summable, while the first sum on
the right diverges.  Letting \(K\to\infty\) gives a contradiction.  Therefore a
persistent non-CKN branch forces either non-effective moving-window
observability or an NS-realizable invisible defect cascade.  The contrapositive
gives the CKN-scale and regularity conclusion.
\end{proof}

\subsection{Controlled-regime corollary}

\begin{corollary}[Controlled-regime CKN-scale criterion]
\label{cor:controlled-regime-ckn-scale}
Assume the hypotheses of \Cref{thm:moving-window-defect-cascade-alternative}.
Suppose the moving windows satisfy a growth bound of the form
\[
        M_n\le C N_n^a\exp(CN_n^b)
\]
and are chosen logarithmically admissibly on the extraction set, so that
\[
        \sum_{k}\lambda_{n_k}M_{n_k}^{-q}=+\infty.
\]
If no NS-realizable invisible defect cascade exists in this controlled window
family, then a CKN scale occurs and \(z_0\) is regular.
\end{corollary}

\begin{proof}
The logarithmic admissibility condition implies depletion-effective
observability on the extraction set.  The absence of invisible cascades excludes
the second alternative in \Cref{thm:moving-window-defect-cascade-alternative};
the first alternative is excluded by effectiveness.  Therefore a non-CKN branch
cannot persist.
\end{proof}

\subsection{Failure objects}

The moving-window alternative leaves a precise list of possible failure
mechanisms.

\paragraph{Failure of dyadic defect extraction.}
The non-CKN quantity may fail to produce a nontrivial cleaned active defect
after quotienting only exact pressure-gauge/null directions and after accounting for harmonic-pressure leakage, localization leakage, CKN-small components, and finite-window tails as errors.

\paragraph{Failure of selected-interval depletion.}
A defect may be visible in the finite-window hierarchy, but its observation may
fail to deplete a genuine Navier--Stokes budget on the same index set on which
the defect is extracted.  This includes signed flux cancellation, non-summable
cutoff work, pressure leakage, or the absence of a valid interval budget for a
sparse subsequence.

\paragraph{Finite-window trace obstruction.}
For some fixed window \(W\), the projected obstruction space
\[
        \mathcal T_{W,\NS}^{\PFET}
\]
may be nonzero.  This is the corrected finite-window PFET obstruction.

\paragraph{Residual left-singular invisible cascade.}
After pressure, flux, and energy are removed, an NS-realizable residual may pair
with dual directions satisfying
\[
        A_n^*y_n\to0,
        \qquad
        J_n^Py_n\to0,
        \qquad
        J_n^Fy_n\to0,
        \qquad
        J_n^Ey_n\to0.
\]
This is the trace-cost failure mechanism.

\paragraph{Non-effective moving-window growth.}
Every fixed window may be observable, but the constants \(M_n\) may grow so fast
on the extracted branch that
\[
        \sum_{k}\lambda_{n_k}M_{n_k}^{-q}<\infty.
\]

\subsection{Interpretation}

The theorem converts the local regularity problem into a precise defect-cascade
question.  If no CKN scale appears, then after dyadic rescaling, coarse
graining, exact pressure-gauge cleaning, active/harmonic pressure separation, pressure--flux testing, energy testing, and trace-cost
testing, one of two things must happen:
\[
        \text{the moving-window constants are not absorbable on the extracted branch,}
\]
or
\[
        \text{there exists an NS-realizable combined-invisible defect cascade.}
\]
Thus the remaining obstruction is required to satisfy NS-realizability,
scale-criticality, exact pressure-gauge cleaning, retained harmonic-pressure control, pressure invisibility, flux
invisibility, energy invisibility, and trace invisibility simultaneously.

\section{Model regimes and verification of controlled observability}
\label{sec:model-regimes}

The moving-window alternative is a reduction theorem.  It becomes useful when
its finite-window observability hypotheses can be verified in concrete regimes.
This section records two model settings: the clean periodic finite-window regime
and the perturbatively localized finite-window regime.  The conclusion is not an
unconditional Navier--Stokes regularity theorem.  Rather, the conclusion is that
ordinary active pressure, flux, energy, and trace obstructions reduce to explicit
finite-dimensional kernels or projected trace-obstruction spaces.

\subsection{Clean periodic finite-window regime}

We work on the periodic box
\[
        \T^3=\R^3/(2\pi\Z)^3.
\]
Let
\[
        \Lambda\subset \Z^3\setminus\{0\}
\]
be a finite symmetric set of output modes.  The exclusion of the zero mode removes the only periodic pressure gauge, namely spatial constants.  In a localized cylinder the analogue is not quotienting by all spatially harmonic pressures; the local harmonic component must be retained or estimated as leakage because it can enter momentum and energy through its gradient and pressure work.

A clean periodic finite-window defect direction is a tuple
\[
        z=(\dot U,\dot R)
\]
with Fourier support in \(\Lambda\), where
\[
        \nabla\cdot\dot U=0,
        \qquad
        \dot R=\dot R^T.
\]
The pressure variation is determined by
\begin{equation}
        -\Delta\dot P
        =
        \partial_i\partial_j
        (U_i\dot U_j+\dot U_iU_j+\dot R_{ij}).
        \label{eq:periodic-pressure-variation}
\end{equation}
We write the active pressure source as
\[
        b_z
        =
        \partial_i\partial_j
        (U_i\dot U_j+\dot U_iU_j+\dot R_{ij}).
\]
The flux variation is
\[
        \dot\Pi_z=-\dot R:\nabla U-R:\nabla\dot U.
\]
The energy observation is the chosen finite-dimensional linear energy map from
\Cref{prop:residual-pfe-kernel}, together with the positive energy tests on the
positive covariance cone.  The adjoint trace map is
\[
        A^*_\Lambda:Y_\Lambda^*\to H_\Lambda.
\]

\subsection{Periodic pressure source visibility}

\begin{lemma}[Clean periodic pressure source visibility]
\label{lem:periodic-pressure-source-visibility}
Let \(\Lambda\subset\Z^3\setminus\{0\}\) be finite.  For every pressure source
\[
        b(x)=\sum_{\sigma\in\Lambda}b_\sigma e^{i\sigma\cdot x},
\]
let
\[
        P=(-\Delta)^{-1}b
\]
with zero mean.  Then
\begin{equation}
        N_\Lambda^{-2}|b|_{\ell^2(\Lambda)}
        \le
        |P|_{\ell^2(\Lambda)}
        \le
        m_\Lambda^{-2}|b|_{\ell^2(\Lambda)},
        \label{eq:periodic-pressure-source-visibility}
\end{equation}
where
\[
        N_\Lambda=\max_{\sigma\in\Lambda}|\sigma|,
        \qquad
        m_\Lambda=\min_{\sigma\in\Lambda}|\sigma|.
\]
In particular,
\[
        P=0\quad\Longleftrightarrow\quad b=0.
\]
\end{lemma}

\begin{proof}
Since \(0\notin\Lambda\), the zero-mean solution satisfies
\[
        \widehat P(\sigma)=|\sigma|^{-2}b_\sigma,
        \qquad \sigma\in\Lambda.
\]
Thus
\[
        |P|_{\ell^2(\Lambda)}^2
        =
        \sum_{\sigma\in\Lambda}|\sigma|^{-4}|b_\sigma|^2,
\]
and \eqref{eq:periodic-pressure-source-visibility} follows from
\(N_\Lambda^{-2}\le |\sigma|^{-2}\le m_\Lambda^{-2}\).
\end{proof}

\subsection{The clean finite-window kernel hierarchy}

Define
\[
        K^P_\Lambda=\{z:b_z=0\}.
\]
Define the pressure--flux kernel
\[
        K^{\PF}_\Lambda
        =
        \{z\in K^P_\Lambda:\dot\Pi_z=0\text{ in the selected flux window}\}.
\]
Define the pressure--flux--energy kernel by the chosen linear energy observation:
\[
        K^{\PFE}_\Lambda
        =
        \ker\left(\mathcal E^{\rm lin}_\Lambda|_{K^{\PF}_\Lambda}\right).
\]
Let
\[
        \iota_\Lambda:K^{\PFE}_\Lambda\hookrightarrow Y_\Lambda
\]
be inclusion and let \(\Pi_\Lambda^\perp\) be the orthogonal projection onto
\((\operatorname{Range}A_\Lambda)^\perp\).  The clean trace obstruction space is
\begin{equation}
        \mathcal T^{\PFET}_\Lambda
        :=
        \Pi_\Lambda^\perp\iota_\Lambda(K^{\PFE}_\Lambda).
        \label{eq:clean-trace-obstruction}
\end{equation}

\begin{definition}[Clean finite-window anti-phantom property]
\label{def:clean-anti-phantom-property}
The clean periodic window \(\Lambda\) has the anti-phantom property if
\begin{equation}
        \mathcal T^{\PFET}_{\Lambda,\NS}
        :=
        \Pi_\Lambda^\perp\iota_\Lambda
        \left(
          K^{\PFE}_\Lambda
          \cap
          \overline{\operatorname{Range}\mathcal R^{\NS}_\Lambda}
        \right)
        =\{0\}.
        \label{eq:clean-ns-trace-obstruction-zero}
\end{equation}
Here \(\mathcal R^{\NS}_\Lambda\) denotes the map from NS-realizable dyadic
packages to their cleaned finite-window residual directions.
\end{definition}

\subsection{Clean finite-window combined observability}

\begin{theorem}[Clean periodic finite-window combined observability]
\label{thm:clean-periodic-combined-observability}
Let \(\Lambda\subset\Z^3\setminus\{0\}\) be a finite periodic output window.
Assume:
\begin{enumerate}
\item the pressure observation is taken in the zero-mean pressure quotient and
is injective on the selected active pressure-source range;
\item the flux tests separate \(K^P_\Lambda/K^{\PF}_\Lambda\);
\item the chosen linear energy tests separate
\(K^{\PF}_\Lambda/K^{\PFE}_\Lambda\), while the positive energy tests are used
only on the positive covariance cone;
\item the clean finite-window anti-phantom property
\(\mathcal T^{\PFET}_{\Lambda,\NS}=\{0\}\) holds.
\end{enumerate}
Then there exists a finite constant \(M_\Lambda<\infty\) such that every
NS-realizable cleaned defect direction \(d\in Y_\Lambda\) satisfies
\begin{equation}
        |d|_{Y_\Lambda}
        \le
        M_\Lambda
        \left(
            |O^P_\Lambda d|
            +|O^F_\Lambda d|
            +|O^E_\Lambda d|
            +|O^T_\Lambda d|
        \right).
        \label{eq:clean-periodic-combined-observability}
\end{equation}
\end{theorem}

\begin{proof}
The pressure observation is injective on active pressure sources by
\eqref{eq:periodic-pressure-source-visibility} and the stated source-range
assumption.  Therefore any pressure-invisible direction lies in
\(K^P_\Lambda\).  The flux separation assumption pushes pressure--flux invisible
directions into \(K^{\PF}_\Lambda\).  The energy separation assumption pushes
pressure--flux--energy invisible directions into \(K^{\PFE}_\Lambda\).  The
remaining trace obstruction is the projected space
\(\mathcal T^{\PFET}_{\Lambda,\NS}\), which is zero by assumption.  Thus the
combined observation map has no nonzero NS-realizable kernel.  Since all spaces
are finite-dimensional, compactness on the unit sphere gives
\eqref{eq:clean-periodic-combined-observability}.
\end{proof}

\begin{corollary}[Clean finite-window algebraic alternative]
\label{cor:clean-finite-window-algebraic}
For a fixed clean periodic window \(\Lambda\), either the combined observation
map is injective on the selected active finite-window space, or a nonzero formal
kernel appears in one of the explicit finite-dimensional stages
\[
        K^P_\Lambda,
        \qquad
        K^{\PF}_\Lambda,
        \qquad
        K^{\PFE}_\Lambda,
        \qquad
        \mathcal T^{\PFET}_\Lambda.
\]
Only after intersection with the NS-realizable range does such a formal kernel
become relevant to Navier--Stokes.
\end{corollary}

\subsection{Perturbatively localized finite-window regime}

Let \(\chi_R\) be a slowly varying cutoff,
\[
        \chi_R(x,t)=\chi(x/R,t/R^2),
\]
and let \(O_{\Lambda,R}\) denote the localized combined observation map.  Write
\begin{equation}
        O_{\Lambda,R}=O_\Lambda^{\rm per}+C_{\Lambda,R},
        \label{eq:localized-combined-map-decomp}
\end{equation}
where \(O_\Lambda^{\rm per}\) is the clean periodic principal map and
\(C_{\Lambda,R}\) is the localization commutator.  The commutator includes
cutoff leakage, localized pressure inverse error, boundary leakage,
energy-flux cutoff terms, and trace-window projection error.

\begin{assumption}[Perturbative localization bound]
\label{ass:perturbative-localization-bound}
For each fixed finite window \(\Lambda\), there exist constants \(C_\Lambda>0\)
and \(R_\Lambda<\infty\) such that, for all \(R\ge R_\Lambda\),
\begin{equation}
        |C_{\Lambda,R}d|
        \le
        C_\Lambda R^{-1}|d|_{Y_\Lambda}
        \label{eq:perturbative-localization-bound}
\end{equation}
for every \(d\in Y_\Lambda\).
\end{assumption}

\begin{theorem}[Perturbative localization preserves combined observability]
\label{thm:perturbative-localization-combined-observability}
Assume clean periodic combined observability on \(\Lambda\):
\[
        |d|_{Y_\Lambda}\le M_\Lambda|O_\Lambda^{\rm per}d|.
\]
Assume also
\[
        \|C_{\Lambda,R}\|\le \frac1{2M_\Lambda}.
\]
Then localized combined observability holds:
\begin{equation}
        |d|_{Y_\Lambda}\le 2M_\Lambda|O_{\Lambda,R}d|.
        \label{eq:localized-combined-observability}
\end{equation}
In particular, perturbative localization does not create a nonzero
combined-invisible defect direction.
\end{theorem}

\begin{proof}
Using \eqref{eq:localized-combined-map-decomp},
\[
        |O_{\Lambda,R}d|
        \ge
        |O_\Lambda^{\rm per}d|-|C_{\Lambda,R}d|
        \ge
        \frac1{M_\Lambda}|d|_{Y_\Lambda}
        -
        \frac1{2M_\Lambda}|d|_{Y_\Lambda}.
\]
This gives \eqref{eq:localized-combined-observability}.
\end{proof}

\subsection{Nonperturbative localization}

For a fixed localized finite window \((\Lambda,\chi)\), let
\[
        K^{\rm comb}_{\Lambda,\chi}=\ker O_{\Lambda,\chi}\subset Y_{\Lambda,\chi}.
\]
Then exactly one of the following holds: either \(K^{\rm comb}_{\Lambda,\chi}=0\)
and a finite observability constant exists by compactness, or
\(K^{\rm comb}_{\Lambda,\chi}\ne0\) and every localized combined-invisible
direction lies in this finite-dimensional kernel.  A vector in this kernel is
not automatically a Navier--Stokes obstruction; it may be caused by a poor
cutoff, harmonic pressure leakage, boundary cancellation, or a non-NS-realizable
finite-dimensional artifact.

\subsection{Active/harmonic pressure separation in the localized regime}

Localized pressure is not determined by its local quadratic source alone.  The correct localized structure is
\[
        P=P^{\act}+P^{\har},
        \qquad
        \Delta P^{\har}=0
        \quad\text{in the observation ball},
\]
with only functions of time treated as true pressure gauge.  The active pressure channel may be defined after projecting to \(P^{\act}\), but the harmonic component is retained in the finite-window package and contributes to the energy and trace channels or to an explicitly estimated harmonic-pressure error.

Let \(\calG^{\rm ex}_{\Lambda,\chi}\) be the exact null sector, consisting of the true pressure-gauge directions and any exact Leray-projection null directions.  The cleaned localized active space is
\[
        \widehat Y_{\Lambda,\chi}
        =
        \calZ_{\Lambda,\chi}/\calG^{\rm ex}_{\Lambda,\chi}.
\]
The localization and harmonic-pressure leakage sectors are denoted
\[
        \calN^{\loc}_{\Lambda,\chi},
        \qquad
        \calN^{\har}_{\Lambda,\chi}.
\]
They are not divided out unless a specific observation channel annihilates them exactly.  In the usual localized regime they enter perturbative estimates, for instance
\[
        \|\widetilde{\calO}_{\Lambda,\chi} n\|
        \le
        \eps_{\Lambda,\chi}\|n\|,
        \qquad
        n\in \calN^{\loc}_{\Lambda,\chi}+\calN^{\har}_{\Lambda,\chi}.
\]
By \Cref{lem:quotient-observations-well-defined}, the cleaned observation map is well-defined on \(\widehat Y_{\Lambda,\chi}\) because only exact null directions are quotiented.  If the cleaned localized kernel is zero and the perturbative leakage is small enough, finite-dimensional compactness and stability give a finite cleaned observability constant.

\subsection{Controlled moving-window regime}

Let
\[
        W_n=(n,\ell_n,\Lambda_n,\chi_n,s_n)
\]
be a moving sequence of cleaned localized finite windows, and let
\(N_n=N_{\Lambda_n}\) denote the frequency size of the window.  Suppose
\begin{equation}
        M_n\le C N_n^p\exp(CN_n^b)
        \label{eq:model-moving-growth-bound}
\end{equation}
for constants \(C,p,b>0\).  On an extraction set \(I=\{n_k\}\), the moving
family is controlled if
\begin{equation}
        \sum_k \lambda_{n_k}N_{n_k}^{-pq}\exp(-qC N_{n_k}^b)=+\infty.
        \label{eq:model-effective-divergence}
\end{equation}

\begin{theorem}[Controlled moving-window model theorem]
\label{thm:controlled-moving-window-model}
Assume dyadic defect extraction, selected-interval observable depletion, the
growth bound \eqref{eq:model-moving-growth-bound}, the effective divergence
condition \eqref{eq:model-effective-divergence}, and absence of
NS-realizable invisible defect cascades in the cleaned moving-window family.
Then no non-CKN dyadic branch exists at \(z_0\).  Hence a CKN scale occurs and
\(z_0\) is regular.
\end{theorem}

\begin{proof}
The growth bound implies
\[
        M_{n_k}^{-q}
        \ge
        C^{-1}N_{n_k}^{-pq}\exp(-qC N_{n_k}^b)
\]
after changing constants.  Therefore \eqref{eq:model-effective-divergence}
implies depletion-effective observability on the extraction set.  The conclusion
follows from \Cref{thm:moving-window-defect-cascade-alternative}.
\end{proof}

\subsection{What remains outside the controlled regimes}

The controlled regimes leave the following genuine escape mechanisms:
\begin{enumerate}
\item a formal finite-window kernel whose projected trace obstruction
\(\mathcal T^{\PFET}_{\Lambda,\NS}\) intersects the NS-realizable range;
\item nonperturbative localization too large relative to the clean observability
constant;
\item uncontrolled retained harmonic pressure work or leakage;
\item moving-window amplification so fast that
\(\sum_k\lambda_{n_k}M_{n_k}^{-q}<\infty\);
\item residual left-singular alignment with dual directions invisible to
pressure, flux, energy, and adjoint trace;
\item failure of observable depletion on the selected extraction intervals.
\end{enumerate}

\subsection{Summary of the model verification}

Clean periodic active windows reduce combined invisibility to explicit
finite-dimensional kernels and the projected trace obstruction
\(\mathcal T^{\PFET}_\Lambda\).  Perturbative localization preserves clean
observability.  Nonperturbative localization leaves a finite-dimensional kernel.
Moving windows require an effective growth series on the extraction set.  Thus a
singular branch in the controlled regimes must either force non-effective
constants or produce an NS-realizable combined-invisible cascade.

\section{Critical recurrence program and sustaining diagnostics}
\label{sec:critical-recurrence-program}

The preceding sections formulate a combined-observability reduction: a persistent non-CKN branch must either pass through non-effective moving-window constants or produce an NS-realizable combined-invisible defect cascade.  This section records a more mechanism-oriented program.  Its purpose is not to replace the reduction theorem, but to explain what such a cascade would have to do dynamically.  A genuine singular branch would have to rebuild a scale-critical bad pattern at smaller and smaller scales.  We therefore ask which mechanisms could sustain this recurrence.

The three candidate mechanisms are
\[
    \text{vortex stretching},
    \qquad
    \text{active pressure work},
    \qquad
    \text{coarse-grained interscale flux}.
\]
The point of the present section is to define scale-invariant diagnostics for these mechanisms which are natural in the rescaled unit cylinder and meaningful for suitable weak solutions after mollification where necessary.

\subsection{Renormalized orbit and critical recurrence}

Let \(z_0=(0,0)\) be a candidate singular point, and set
\[
    r_n=2^{-n}.
\]
The renormalized Navier--Stokes states are
\begin{equation}
    u^{(n)}(x,t)=r_nu(r_nx,r_n^2t),
    \qquad
    p^{(n)}(x,t)=r_n^2p(r_nx,r_n^2t).
\end{equation}
All diagnostics below are defined on the fixed unit cylinder \(Q_1\).  This convention is important: scaling correctness is built into the definitions, rather than imposed by hand through powers of \(r_n\).

A natural critical state space is
\[
    \Xcrit
    =
    \left\{
    (v,q):
    v\in L^3(Q_1),\quad
    q-(q)_{B_1}(t)\in L^{3/2}(Q_1)
    \right\},
\]
possibly refined by the suitable weak solution local energy topology.  Define the CKN-good set
\[
    \Good_{\CKN}
    =
    \left\{
    (v,q)\in\Xcrit:
    \exists\rho\in(0,1)
    \text{ such that }
    \Psi_{v,q}(\rho)\le\eps_{\CKN}
    \right\},
\]
and the CKN-bad set \(\Bad_{\CKN}=\Xcrit\setminus\Good_{\CKN}\).

\begin{definition}[Critical recurrent branch]
A dyadic branch at \((0,0)\) is called critical recurrent if there exist constants \(0<c<C<\infty\) such that, along infinitely many \(n\),
\[
    c\le \Psi(r_n)\le C,
\]
and the rescaled states remain a definite distance from the CKN-good set in a chosen critical topology:
\[
    \dist_{\Xcrit}
    \bigl((u^{(n)},p^{(n)}),\Good_{\CKN}\bigr)
    \ge c_0>0.
\]
\end{definition}

\begin{problem}[Renormalized recurrence problem]
Can a suitable weak solution generate a renormalized orbit which remains recurrent in \(\Bad_{\CKN}\)?  Equivalently, can the Navier--Stokes scaling dynamics sustain a noncompact critical bad regime?
\end{problem}

\subsection{Admissible cutoffs and relative filters}

Fix a cutoff
\[
    \chi\in C_c^\infty(Q_1),
    \qquad
    0\le\chi\le1,
    \qquad
    \chi\equiv1 \text{ on } Q_{1/2}.
\]
Also fix a nonnegative parabolic mollifier \(\rho\in C_c^\infty(\R^3\times\R)\) with
\[
    \int_{\R^3\times\R}\rho(x,t)\,dx\,dt=1,
\]
and define
\[
    \rho_\ell(x,t)=\ell^{-5}\rho(x/\ell,t/\ell^2),
    \qquad
    S_\ell f=\rho_\ell*f.
\]
Here \(\ell\in(0,\ell_0)\) is a relative filter scale.  In physical variables the corresponding filter scale is \(\ell r_n\).  The parameter \(\ell\) therefore resolves subscale structure relative to the current blow-up scale, not an absolute length scale.

\subsection{Flux diagnostic}

For the renormalized solution define
\[
    U_{n,\ell}=S_\ell u^{(n)},
\]
\[
    R_{n,\ell}
    =
    S_\ell(u^{(n)}\otimes u^{(n)})
    -U_{n,\ell}\otimes U_{n,\ell},
\]
and
\[
    \Pi_{n,\ell}
    =
    -R_{n,\ell}:\nabla U_{n,\ell}.
\]
Since \(u^{(n)}\in L^3\), the tensor \(u^{(n)}\otimes u^{(n)}\) belongs to \(L^{3/2}\), and for fixed \(\ell>0\) the smoothed objects are smooth.  Hence \(\Pi_{n,\ell}\in L^1(Q_1)\).

Define the absolute flux activity
\begin{equation}
    \Flux^{\absop}_{n,\ell}
    =
    \int_{Q_1}\chi\,|\Pi_{n,\ell}|\,dx\,dt,
\end{equation}
and the signed flux activity
\begin{equation}
    \Flux^{\sgn}_{n,\ell}
    =
    \left|
    \int_{Q_1}\chi\,\Pi_{n,\ell}\,dx\,dt
    \right|.
\end{equation}
The flux coherence ratio is
\begin{equation}
    \Gamma_{n,\ell}
    =
    \frac{\Flux^{\sgn}_{n,\ell}}
    {\Flux^{\absop}_{n,\ell}+\varepsilon},
    \qquad \varepsilon>0.
\end{equation}
The parameter \(\varepsilon\) only prevents division by zero and may later be sent to zero.

\begin{definition}[Sign-coherent and oscillatory flux recurrence]
A branch is sign-coherent flux-sustained if, along a subsequence and for some relative filters \(\ell_n\),
\[
    \Flux^{\absop}_{n,\ell_n}\ge c_0>0,
    \qquad
    \Gamma_{n,\ell_n}\ge c_1>0.
\]
It is oscillatory flux-cancellation sustained if
\[
    \Flux^{\absop}_{n,\ell_n}\ge c_0>0,
    \qquad
    \Gamma_{n,\ell_n}\to0.
\]
\end{definition}

The distinction is essential.  Local energy identities naturally control signed flux, while a cascade strength is closer to absolute flux.  Thus a flux-mediated recurrent branch is either sign-coherent, in which case one may hope to deplete a finite energy budget, or oscillatory, in which case the obstruction is a cancellation mechanism not detected by signed energy balance alone.

\subsection{Pressure diagnostic}

For suitable weak solutions one has \(p^{(n)}\in L^{3/2}\) and \(u^{(n)}\in L^3\).  Therefore the pressure work term appearing in the local energy inequality is integrable.  Define the total pressure work
\begin{equation}
    \Press^{\tot}_{n}
    =
    \left|
    \int_{Q_1}
    \bigl(p^{(n)}-(p^{(n)})_{B_1}(t)\bigr)
    u^{(n)}\cdot\nabla\chi\,dx\,dt
    \right|.
\end{equation}
This quantity is invariant under adding a function of time to the pressure, because \(\nabla\cdot u^{(n)}=0\).

To distinguish active pressure from harmonic leakage, choose \(\eta\in C_c^\infty(B_1)\) with \(\eta\equiv1\) on \(B_{3/4}\), and define the local active pressure by
\begin{equation}
    p^{\act}_n
    =
    \mathcal R_i\mathcal R_j
    \bigl(\eta u^{(n)}_i u^{(n)}_j\bigr),
\end{equation}
where \(\mathcal R_i\) denotes the Riesz transform.  Since \(u^{(n)}_i u^{(n)}_j\in L^{3/2}\), we have \(p^{\act}_n\in L^{3/2}\).  Define
\begin{equation}
    \Press^{\act}_{n}
    =
    \left|
    \int_{Q_1}
    p^{\act}_n u^{(n)}\cdot\nabla\chi\,dx\,dt
    \right|.
\end{equation}
The remaining pressure is denoted
\[
    p^{\har}_n=p^{(n)}-p^{\act}_n,
\]
modulo functions of time, and is spatially harmonic in the interior region where \(\eta\equiv1\).  The corresponding harmonic pressure leakage is
\begin{equation}
    \Press^{\har}_{n}
    =
    \left|
    \int_{Q_1}
    p^{\har}_n u^{(n)}\cdot\nabla\chi\,dx\,dt
    \right|.
\end{equation}
The quantity most relevant to a local recurrence mechanism is \(\Press^{\act}_n\), because it is generated by the local quadratic source rather than by the retained harmonic pressure component or far-field leakage.

\subsection{Resolved vortex-stretching diagnostic}

Let
\[
    \Omega_{n,\ell}=\nabla\times U_{n,\ell},
    \qquad
    S_{n,\ell}=\frac12\bigl(\nabla U_{n,\ell}+\nabla U_{n,\ell}^T\bigr).
\]
For a suitable weak solution the raw stretching term \(\omega\cdot S\omega\) is not a stable \(L^1\) object, since \(\omega\in L^2\) and \(S\in L^2\).  The filtered diagnostic is therefore used.

Define the positive resolved stretching production
\begin{equation}
    \Vort^+_{n,\ell}
    =
    \int_{Q_1}
    \chi
    \bigl(\Omega_{n,\ell}\cdot S_{n,\ell}\Omega_{n,\ell}\bigr)_+
    \,dx\,dt.
\end{equation}
Define also the alignment ratio
\begin{equation}
    \mathfrak A_{n,\ell}
    =
    \frac{
    \int_{Q_1}
    \chi
    \bigl(\Omega_{n,\ell}\cdot S_{n,\ell}\Omega_{n,\ell}\bigr)_+
    \,dx\,dt
    }
    {
    \int_{Q_1}\chi |S_{n,\ell}|\,|\Omega_{n,\ell}|^2\,dx\,dt+\varepsilon
    }.
\end{equation}
Thus \(\Vort^+_{n,\ell}\) measures the amount of positive resolved stretching, while \(\mathfrak A_{n,\ell}\) measures the degree of geometric alignment between vorticity and expanding strain.

There is also a subgrid vorticity-transfer term.  The filtered velocity satisfies
\[
    \partial_t U_{n,\ell}-\Delta U_{n,\ell}
    +\nabla\cdot(U_{n,\ell}\otimes U_{n,\ell})
    +\nabla P_{n,\ell}
    =
    -\nabla\cdot R_{n,\ell}.
\]
Taking curl gives, schematically,
\[
    \partial_t\Omega_{n,\ell}-\Delta\Omega_{n,\ell}
    +U_{n,\ell}\cdot\nabla\Omega_{n,\ell}
    =
    \Omega_{n,\ell}\cdot\nabla U_{n,\ell}
    -\nabla\times\nabla\cdot R_{n,\ell}.
\]
Hence unresolved scales may contribute to resolved vorticity growth.  Define
\begin{equation}
    \Vort^{\sg}_{n,\ell}
    =
    \left|
    \int_{Q_1}
    \chi\,
    \Omega_{n,\ell}\cdot
    (\nabla\times\nabla\cdot R_{n,\ell})
    \,dx\,dt
    \right|.
\end{equation}
Equivalently, after integration by parts, one may use
\begin{equation}
    \Vort^{\sg}_{n,\ell}
    =
    \left|
    \int_{Q_1}
    R_{n,\ell}:
    \nabla\bigl(\nabla\times(\chi\Omega_{n,\ell})\bigr)
    \,dx\,dt
    \right|.
\end{equation}
The full filtered vorticity-sustaining diagnostic is
\begin{equation}
    \Vort_{n,\ell}
    =
    \Vort^+_{n,\ell}
    +
    \Vort^{\sg}_{n,\ell}.
\end{equation}
For a first, simpler version of the program one may work only with \(\Vort^+_{n,\ell}\) and treat \(\Vort^{\sg}_{n,\ell}\) as an error or as a separate unresolved mechanism.

\subsection{Scaling and suitability}

The preceding definitions are natural for three reasons.
First, they are made on the rescaled unit cylinder, so they are invariant under the Navier--Stokes scaling.  Written back in physical variables, they correspond schematically to
\[
    \Flux^{\absop}_{n,\ell}
    \sim
    r_n^{-1}\int_{Q_{r_n}} |\Pi_{\ell r_n}|\,dx\,dt,
\]
\[
    \Press_n
    \sim
    r_n^{-1}\int_{Q_{r_n}} p\,u\cdot\nabla\chi_{r_n}\,dx\,dt,
\]
and
\[
    \Vort_{n,\ell}
    \sim
    r_n\int_{Q_{r_n}}
    (\omega_{\ell r_n}\cdot S_{\ell r_n}\omega_{\ell r_n})_+\,dx\,dt,
\]
with the appropriate filtered fields.

Second, they are meaningful for suitable weak solutions.  Flux and pressure work use only \(u\in L^3\), \(p\in L^{3/2}\), and smoothing.  The vorticity-stretching diagnostic is filtered precisely because the raw product \(\omega\cdot S\omega\) is not naturally integrable at suitable-weak regularity.

Third, they separate mechanisms.  \(\Flux^{\absop}_{n,\ell}\) measures interscale activity, \(\Gamma_{n,\ell}\) measures signed coherence, \(\Press^{\act}_n\) measures active pressure work, and \(\Vort_{n,\ell}\) measures resolved stretching plus subgrid vorticity transfer.

\subsection{Critical recurrence diagnostic conjecture}

\begin{conjecture}[Critical recurrence diagnostic conjecture]
Let \((u,p)\) be a suitable weak solution and let \((0,0)\) have a non-CKN dyadic branch:
\[
    \Psi(r_n)>\eps_{\CKN}
    \qquad\text{for all sufficiently large } n.
\]
Then there exist relative filter scales \(\ell_n\in[\ell_-,\ell_+]\subset(0,1)\), a subsequence, and a constant \(c>0\) such that
\begin{equation}
    \Vort_{n,\ell_n}
    +
    \Press^{\act}_{n}
    +
    \Flux^{\absop}_{n,\ell_n}
    \ge c.
\end{equation}
\end{conjecture}

The conjecture should be read as a necessary-mechanism statement, not as a regularity theorem.  It says that a recurrent bad branch cannot remain bad without at least one sustaining mechanism.  If true, the local regularity problem can be reformulated as follows: determine whether persistent vortex stretching, active pressure focusing, or interscale flux can continue indefinitely without producing a CKN scale or depleting a finite local budget.

\begin{target}[Necessary mechanism theorem]
Prove a rigorous version of the implication
\[
    \Psi(r_n)>\eps_{\CKN}\text{ for all large }n
    \quad\Longrightarrow\quad
    \limsup_{n\to\infty}
    \left(
    \Vort_{n,\ell_n}
    +\Press^{\act}_{n}
    +\Flux^{\absop}_{n,\ell_n}
    \right)>0
\]
for an admissible choice of relative filters \(\ell_n\).
\end{target}

\subsection{Programmatic trichotomy}

Assuming the diagnostic conjecture, any critical recurrent branch should fall into one of three mechanism classes.

\begin{enumerate}[label=\textup{(\roman*)}]
\item \emph{Stretching-dominated recurrence:}
\[
    \limsup_n \Vort_{n,\ell_n}>0.
\]
The problem is whether persistent vortex-stretching alignment can be compatible with diffusion, pressure projection, and known vorticity-direction regularity mechanisms.

\item \emph{Pressure-dominated recurrence:}
\[
    \limsup_n \Press^{\act}_n>0.
\]
The problem is whether active pressure can refocus strain and sustain recurrence, rather than merely redistributing or dispersing energy.

\item \emph{Flux-dominated recurrence:}
\[
    \limsup_n \Flux^{\absop}_{n,\ell_n}>0.
\]
This case splits into sign-coherent flux recurrence and oscillatory flux-cancellation recurrence according to the behavior of \(\Gamma_{n,\ell_n}\).
\end{enumerate}

\begin{principle}[Critical recurrence principle]
A Navier--Stokes singularity, if it exists, should correspond to a recurrent bad orbit sustained by vortex stretching, active pressure work, interscale flux, or an oscillatory cancellation among these mechanisms.  A positive regularity route must show that every such sustaining mechanism is impossible or budget-depleting.  A counterexample route must construct a recurrent mechanism that escapes all such budgets while preserving incompressibility, pressure compatibility, Reynolds covariance positivity, and the local energy inequality.
\end{principle}

\subsection{Relation to the observability framework}

The earlier defect-observability framework remains useful, but it should be interpreted as a secondary language.  The defect package
\[
    D_{n,\ell}
    =
    (U_{n,\ell},P_{n,\ell};P^{\act}_{n,\ell},P^{\har}_{n,\ell},R_{n,\ell},\Pi_{n,\ell})
\]
records the coarse-grained manifestation of the renormalized orbit.  Combined observability asks whether this package is visible to pressure, flux, energy, and trace.  The critical recurrence program asks a more mechanism-oriented question:
\[
    \boxed{
    \text{what sustains the recurrence of bad scales?}
    }
\]
Pressure--flux observability becomes a tool for detecting flux-sustained recurrence.  Energy positivity becomes a tool for excluding covariance-supported hidden recurrence.  Trace-cost language should be invoked only after vortex stretching, active pressure work, and interscale flux have failed to explain the recurrence.

\subsection{Revised research roadmap}

The program suggests the following order of attack.

\begin{enumerate}[label=\textup{Step \arabic*:}, leftmargin=3em]
\item Define robust diagnostics \(\Vort_{n,\ell}\), \(\Press^{\act}_n\), \(\Flux^{\absop}_{n,\ell}\), and \(\Gamma_{n,\ell}\) with precise admissibility requirements on \(\chi\), \(\eta\), and \(S_\ell\).

\item Prove a weak necessary-mechanism theorem: a persistent non-CKN branch forces at least one of the diagnostics to remain nontrivial along a subsequence.

\item Analyze stretching-dominated recurrence through vorticity-direction geometry, strain concentration, and possible alignment rigidity.

\item Analyze pressure-dominated recurrence by separating active pressure from harmonic leakage and by estimating pressure work in the local energy inequality.

\item Analyze flux-dominated recurrence by separating sign-coherent flux from oscillatory flux cancellation.  A sign-coherent flux cascade should be tested against local energy depletion; an oscillatory cascade should be tested against pressure compatibility and Reynolds covariance positivity.

\item Reinterpret any remaining mechanism as an NS-realizable invisible defect cascade in the earlier combined-observability framework.
\end{enumerate}

This section is deliberately programmatic.  Its purpose is to turn the final obstruction from a formal invisible defect into a dynamical question about the persistence of scale-critical sustaining mechanisms.

\section{Concrete theorem targets and the resulting research program}
\label{sec:concrete-targets}

The previous sections reduce the local regularity problem to a precise
moving-window alternative.  A non-CKN dyadic branch can persist only if the
combined observability constants are not effective on the extracted branch, the
selected-interval depletion mechanism fails, or an NS-realizable invisible
defect cascade exists.  This section formulates the concrete theorem targets
needed to turn the reduction into a positive regularity route.

The statements in this section are targets, not established theorems of the
present manuscript.

\subsection{Target I: Dyadic defect extraction}

The first target is to justify the bridge between failure of CKN smallness and
the existence of nontrivial cleaned dyadic defects.

\begin{target}[Dyadic defect extraction]
\label{thm:target-defect-extraction}
Let \((u,p)\) be a suitable weak solution in \(Q_1\).  Suppose \(z_0=(0,0)\)
has a non-CKN dyadic branch:
\[
        \Psi(r_n)>\varepsilon_{\rm CKN}
        \qquad
        \text{for all }n\ge n_0.
\]
Then there exist coarse-graining scales \(\ell_{n_k}\), finite windows
\(W_{n_k}\), an infinite extraction set \(I=\{n_k\}\), and cleaned
NS-realizable defect directions
\[
        d_{n_k}=[\Pi_{W_{n_k}}\mathfrak D_{n_k,\ell_{n_k}}]\in Y_{W_{n_k}}
\]
such that
\[
        |d_{n_k}|_{Y_{W_{n_k}}}\ge c_0>0
        \qquad
        \text{for all }k.
\]
A stronger version would take \(I\) to contain every sufficiently large dyadic
index.  If only a sparse \(I\) is available, Target II must be proved on the
matching selected intervals.
\end{target}

The main difficulty is to prevent the entire non-CKN quantity from disappearing
into the exact time-dependent pressure gauge, retained harmonic-pressure leakage, pure localization leakage, already CKN-small
components, or finite-window truncation tails.  A useful first version may allow
\[
        |d_{n_k}|_{Y_{W_{n_k}}}\ge c_0-\operatorname{Tail}(W_{n_k})
\]
and then choose windows so that the tail is smaller than \(c_0/2\).

\subsection{Target II: Observable depletion}

The second target is to prove that visible defects consume a finite dyadic
budget on the same index set on which they are extracted.

\begin{target}[Observable depletion on selected intervals]
\label{thm:target-observable-depletion}
Let \(I=\{n_k\}\) be the extraction set from Target I.  There exist a
nonnegative selected budget \(\mathscr B_k\), weights \(\lambda_{n_k}>0\), an
exponent \(q\ge1\), and summable errors \(e_k\ge0\) such that for every
NS-realizable normalized extracted defect direction \(\widehat d_{n_k}\),
\[
        \mathscr B_k-\mathscr B_{k+1}
        \ge
        c\lambda_{n_k}
        \left(\mathsf O_{n_k}(\widehat d_{n_k})\right)^q
        -e_k.
\]
In the full-extraction case, one may take \(\mathscr B_k=B_{n_k}\) with
\(n_k=k\) after reindexing.  In the sparse case, this is an interval-depletion
statement and cannot be replaced by one-step inequalities on selected indices
without an additional monotonicity argument.
\end{target}

This target is where the local energy inequality becomes structural.  It must
measure the cost of visible defect production across scales.  In the
pressure--flux case, the forward sign-coherent flux condition from
\Cref{def:signed-flux-window} is one possible route; backscatter or oscillatory
flux requires a separate budget or an explicit error term.

\subsection{Target III: finite-window NS-realizability exclusion}

The finite-window hierarchy leaves formal kernels and, after the trace step, a
projected obstruction space.  A formal object becomes relevant to
Navier--Stokes only after intersection with the finite-window closure of
NS-realizable packages.

\begin{target}[Finite-window NS-realizability exclusion]
\label{thm:target-NS-realizability-exclusion}
Let \(W\) be a clean or perturbatively localized finite observation window.  The
NS-realizable projected PFET obstruction space satisfies
\[
        \mathcal T_{W,\NS}^{\PFET}
        =
        \Pi_W^\perp\iota_W
        \left(
          K_W^{\PFE}
          \cap
          \overline{\operatorname{Range}\mathcal R_W^{\NS,T}}
        \right)
        =\{0\},
\]
after quotienting only exact null directions and after treating harmonic-pressure leakage, localization effects, and CKN-small components as controlled errors.
\end{target}

There are two possible routes.  The first is algebraic: compute the finite-window
symbols for pressure, flux, energy, and trace and prove that any common formal
obstruction violates the coarse Navier--Stokes identities.  The second is
variational: use Reynolds covariance positivity, pressure compatibility, and the
local energy inequality to show that any nonzero NS-realizable obstruction must
carry pressure, flux, energy, or trace signal.

\subsection{Target IV: moving-window growth bounds}

Even if every fixed window is observable, a potential singularity may escape by
moving through windows whose constants deteriorate.

\begin{target}[Moving-window growth control]
\label{thm:target-moving-window-growth}
For every controlled family of active finite windows \(W\), the cleaned combined
observability constants satisfy
\[
        M_W\le C N_W^p\exp(CN_W^b),
\]
where \((C,p,b)\) depend only on the window class and local scale-critical
bounds, not on the dyadic scale.  On every extraction set \(I=\{n_k\}\) chosen
by Target I, the windows can be selected so that
\[
        \sum_k \lambda_{n_k}N_{n_k}^{-pq}\exp(-qC N_{n_k}^b)=+\infty.
\]
\end{target}

The expected sources of growth are finite-dimensional quotient conditioning,
localized pressure inverse constants, cutoff commutators, backward parabolic
adjoint amplification, and trace-cost singular values.  The goal is to keep
these losses polynomial or single-exponential in the active window size.

\subsection{A first conditional main theorem}

The four targets combine into a conditional local regularity theorem.

\begin{theorem}[Conditional defect-observability regularity theorem]
\label{thm:conditional-defect-observability-regularity}
Let \((u,p)\) be a suitable weak solution in \(Q_1\), and let \(z_0=(0,0)\).
Assume that Targets I--IV hold in a controlled window class, with Target II
formulated on the same extraction set as Target I.  Assume also that no residual
left-singular invisible defect cascade exists outside the controlled window
class.  Then \(z_0\) is regular.
\end{theorem}

\begin{proof}
Assume, for contradiction, that \(z_0\) is singular.  Then no sufficiently small
dyadic scale satisfies the CKN smallness condition, so there is a non-CKN dyadic
branch.  Target I yields an extraction set \(I=\{n_k\}\) and normalized
NS-realizable defect directions \(\widehat d_{n_k}\).  Target III excludes
fixed-window NS-realizable trace obstructions, and the assumed absence of
residual left-singular cascades excludes the moving trace failure.  Hence the
combined observation lower bound gives
\[
        \mathsf O_{n_k}(\widehat d_{n_k})\gtrsim M_{n_k}^{-1}.
\]
Target II gives
\[
        \mathscr B_k-\mathscr B_{k+1}
        \ge
        c\lambda_{n_k}M_{n_k}^{-q}-e_k.
\]
Target IV makes the series \(\sum_k\lambda_{n_k}M_{n_k}^{-q}\) divergent, while
\(\sum_k e_k<\infty\).  Summing in \(k\) contradicts the finiteness and
nonnegativity of the selected budget.  Therefore a CKN scale must occur, and
\(z_0\) is regular.
\end{proof}

\begin{remark}
The theorem is conditional, but the conditions are now aligned: sparse defect
extraction is paired with selected-interval depletion, and the trace obstruction
is the projected space \(\mathcal T_{W,\NS}^{\PFET}\), not the incorrect
condition that the whole residual be orthogonal to \(\operatorname{Range}A_W\).
\end{remark}

\subsection{Counterexample route}

A counterexample to this route would have to construct a moving-window sequence
\[
        (W_{n_k},d_{n_k},y_{n_k})
\]
such that
\[
        d_{n_k}\in \overline{\operatorname{Range}\mathcal R^{\NS}_{W_{n_k}}},
        \qquad
        |d_{n_k}|_{Y_{W_{n_k}}}\ge c_0,
        \qquad
        |y_{n_k}|_{Y_{W_{n_k}}^*}=1,
\]
\[
        \mathsf O_{n_k}^*(y_{n_k})\to0,
        \qquad
        |\langle d_{n_k},y_{n_k}\rangle|\ge c_1.
\]
It must also preserve the coarse momentum identity, pressure compatibility,
Reynolds covariance positivity, the flux identity, and the local energy
inequality.  Thus the failure object is a highly organized scale-critical
NS-realizable invisible cascade, not an arbitrary concentration defect.

\subsection{First feasible theorem: pressure--flux depletion in controlled windows}

A first feasible theorem target is a pressure--flux depletion statement in
controlled windows.

\begin{target}[Pressure--flux depletion in controlled windows]
\label{thm:first-feasible-pf-depletion}
Let \(W\) be a clean periodic or perturbatively localized active finite window.
Suppose a normalized NS-realizable defect direction \(d\in Y_W\) has nonzero
active pressure--flux observation:
\[
        |O_W^P d|+|O_W^F d|\ge\eta.
\]
Then, under forward sign coherence for the flux part and a clean active/harmonic pressure split, there exists a local budget \(B_W\) and a constant \(c_W>0\) such
that
\[
        B_{\rm in}-B_{\rm out}\ge c_W\eta^q-\Err_W.
\]
For perturbatively localized windows, the localization contribution to the
error should satisfy \(\Err_W\le C_W R^{-1}\), up to the explicitly separated
oscillation/backscatter term.
\end{target}

\subsection{Suggested order of attack}

A conservative order of attack is:
\begin{enumerate}
\item prove clean periodic pressure--flux coercivity;
\item prove perturbative localization stability;
\item prove pressure--flux depletion for NS-realizable packages under forward
sign coherence;
\item analyze the pressure--flux kernel by positive energy on the covariance
cone and by linear energy observations off the cone;
\item prove trace alignment or identify residual left-singular cascades;
\item estimate moving-window growth constants;
\item upgrade sparse extraction plus interval depletion, or prove full-sequence
extraction.
\end{enumerate}

\subsection{Conclusion of the program section}

The framework converts the local Navier--Stokes regularity problem into a
compatibility-and-observability problem.  Proving the aligned targets would give
a CKN scale.  Any failure of the route must produce an NS-realizable, cleaned,
non-summable, moving-window invisible defect cascade, or a precise failure of
selected-interval depletion or moving-window effectiveness.

\section{Conclusion and next steps}
\label{sec:conclusion}

The framework developed here changes the focus from a single preferred observable to a compatibility system of observables.  A Navier--Stokes defect relevant to the regularity problem must satisfy all of the following constraints simultaneously:
\begin{enumerate}
  \item it must be produced by an actual suitable weak solution through dyadic rescaling and coarse graining;
  \item it must satisfy the exact coarse momentum identity;
  \item its pressure must obey the active pressure compatibility law with the retained harmonic pressure component kept outside the true time-dependent gauge;
  \item its Reynolds stress must generate the exact energy flux;
  \item it must avoid detection by pressure, flux, energy, and selected-time adjoint trace observations;
  \item it must survive trace-cost exactification by aligning with trace-invisible or nearly trace-invisible left singular directions;
  \item it must persist through a moving sequence of windows with non-effective observability constants or non-summable amplification.
\end{enumerate}

The added theorem-targets section makes the status of the paper explicit.  The manuscript does not prove an unconditional regularity theorem.  It identifies four concrete analytic tasks: dyadic defect extraction, observable depletion, finite-window NS-realizability exclusion, and moving-window growth control.  If these targets are proved in a controlled window class, the moving-window alternative yields a CKN scale.  If the route fails, the failure must occur through a much more structured object than an arbitrary compactness defect.

The critical recurrence program added above refines the meaning of this remaining obstruction.  A combined-invisible defect cascade should not be regarded merely as an algebraic kernel.  If it is genuinely Navier--Stokes-realizable, it must be sustained dynamically across scales.  The proposed diagnostics \(\Vort_{n,\ell}\), \(\Press^{\act}_n\), \(\Flux^{\absop}_{n,\ell}\), and \(\Gamma_{n,\ell}\) are meant to identify whether the sustaining mechanism is vortex stretching, active pressure work, sign-coherent flux, or oscillatory flux cancellation.

The main obstruction statement is therefore sharpened as follows.  If a non-CKN dyadic branch persists, then either visible defects deplete a finite local budget, or the branch produces an NS-realizable invisible defect cascade.  After pressure, flux, and energy have been removed, the remaining obstruction is a residual left-singular cascade: a sequence of NS-generated residuals pairing nontrivially with dual directions whose adjoint trace observations degenerate.

This gives a precise target for both directions.  A positive route must prove that the finite-window kernels do not intersect the NS-realizable range with non-summable moving-window amplification, and must verify the depletion inequality for the chosen observations.  A counterexample route must construct a cleaned, NS-realizable, scale-critical cascade invisible to active pressure, flux, energy, and adjoint trace while preserving pressure compatibility, the local energy inequality, and the Reynolds covariance structure.

In the controlled regimes analyzed above, ordinary active pressure defects are visible, perturbative localization is stable, and nonperturbative localization leaves only a finite-dimensional kernel.  Thus the hard remaining question is no longer simply whether compactness can fail.  It is whether the Navier--Stokes equations can realize a moving-window, scale-critical, combined-invisible defect cascade.

\end{document}